\documentclass[oneside, a4paper,11pt,reqno]{amsart}
\usepackage{yfonts}

\RequirePackage[colorlinks,citecolor=blue,urlcolor=blue]{hyperref}

\usepackage[english,activeacute]{babel}
\usepackage{amssymb,amsmath,amsthm,amsfonts,mathrsfs,hyperref,mathtools,stmaryrd}
\usepackage[font=footnotesize]{caption}
\usepackage{tikz}
\usetikzlibrary{intersections,shapes,trees,arrows,spy,positioning,decorations,arrows.meta,decorations.pathreplacing,patterns,calc}
\usepackage{tikz-cd}
\usepackage{tkz-euclide}
\usepackage{pgfplots}
\usepackage{xcolor}
\usepackage[numbers,square,sort]{natbib}
\usepackage{tabularx}
\usepackage{enumitem}  %Proper enumerate 
\usepackage[T1]{fontenc}
\usepackage[foot]{amsaddr}
\usepackage{url}  
\makeatletter
\def\timenow{\@tempcnta\time
  \@tempcntb\@tempcnta
  \divide\@tempcntb60
  \ifnum10>\@tempcntb0\fi\number\@tempcntb
  \multiply\@tempcntb60
  \advance\@tempcnta-\@tempcntb
  :\ifnum10>\@tempcnta0\fi\number\@tempcnta}
\makeatother

\usepackage[ddmmyyyy,hhmmss]{datetime}
\usepackage{color}

\usepackage{euscript}

\usepackage[applemac]{inputenc}
\usepackage{amssymb,amsmath,amsthm,dsfont}
\usepackage{graphics,graphicx}
\usepackage[T1]{fontenc}
\usepackage{color}
\usepackage{epsfig,psfrag}

\usepackage{tabularx}
\usepackage{booktabs}
\usepackage[labelfont=bf,format=plain,justification=raggedright,singlelinecheck=false]{caption}

\newtheorem{theo}{Theorem}[section]
\newtheorem{prop}[theo]{Proposition}
\newtheorem{lemme}[theo]{Lemma}

\newtheorem{cor}[theo]{Corollary}

\newtheorem{remark}[theo]{Remark}

\newtheorem{defi}[theo]{Definition}

\title{Combinatorics of ancestral lines for a Wright-Fisher diffusion with selection in a L\'evy environment}

\author{Gr\'egoire V\'echambre$^{1}$}
\address{$^1$ Hua Loo-Keng Center for Mathematical Sciences, Academy of Mathematics and Systems Science, Chinese Academy of Sciences, No. 55, Zhongguancun East Road, Haidian District, Beijing, China}
\email{vechambre@amss.ac.cn}

\begin{document} 

\maketitle

\begin{abstract}
Wright-Fisher diffusions describe the evolution of the type composition of an infinite haploid population with two types (say type $0$ and type $1$) subject to neutral reproductions, and possibly selection and mutations. 
%For most Wright-Fisher diffusions studied so far only one type was allowed to have a selective advantage. 
In the present paper we study a Wright-Fisher diffusion in a L\'evy environment that gives a selective advantage to sometimes one type, sometimes the other. Classical methods using the Ancestral Selection Graph (ASG) fail in the study of this model because of the complexity, resulting from the two-sided selection, of the structure of the information contained in the ASG. We propose a new method that consists in encoding the relevant combinatorics of the ASG into a function. We show that the expectations of the coefficients of this function form a (non-stochastic) semigroup and deduce that they satisfy a linear system of differential equations. As a result we obtain a series representation for the fixation probability $h(x)$ (where $x$ is the initial proportion of individuals of type $0$ in the population) as an infinite sum of polynomials whose coefficients satisfy explicit linear relations. 
%A surprising characteristic of the model is that, contrarily to most classical models, it seems that $h(x)$ is not analytic in general so we cannot provide a power series representation for $h(x)$. 
Our approach then allows to derive Taylor expansions at every order for $h(x)$ near $x=0$ and to obtain an explicit recursion formula for the coefficients. 
\end{abstract}

%\tableofcontents

\pagestyle{myheadings}
\markboth{Right}{Combinatorics of ancestral lines for a Wright-Fisher diffusion with selection in a L\'evy environment}

\section{Introduction} \label{firstteps}

Wright-Fisher diffusions model the type-frequency evolution of an essentially infinite  haploid population. Individuals are of either one of two types, say type~$0$ or type~$1$; the biological interpretation usually being that of two different genotypes within the same species or of two different competing species. The basic reproduction mechanism in the population is neutral, i.e. it is independent of the type; but selective effects and mutations can be included in the model. Selective pressure can originate from the environment. In many biological situations, the environment is not stable and the effect of its fluctuations on the type-frequency process is complex. Models involving fluctuating selection have been extensively studied in the past (see e.g.~\cite{G72,KL74,KL74b,KL75,Bu87, BG02, SJV10}) and there is currently a renewed interest for such models (see e.g.~\cite{BCM19,CSW19, BEK19, ChK19, GJP18, GGP19, GPS19}). In \cite{cordvech}, the author studied Wright-Fisher diffusions with L\'evy environments in the case where selection always favors the same type. 
There, the selective advantage of fit individuals is boosted at punctual exceptional environmental events (that may represent peaks of temperature, precipitations, availability of resources, etc.) modeled by the jumps of the L\'evy environment, and those fit individuals may additionally have a permanent selective advantage that is expressed continuously. In many population genetics models that include environmental effects on selection, always the same type is favored by selection. This is less because it is more realistic, but rather due to the technical difficulties that arise in the analysis of models where both types can be favored. In practice, changing environmental situations may very well favor sometimes one type and sometimes the other (see Section \ref{motbiol} for examples). 

In the present paper we are interested in a generalization of the model of \cite{cordvech} but where the L\'evy environment has two types of jumps: jumps that give a selective advantage to individuals of type $0$ and jumps that give a selective advantage to individuals of type $1$. We call this feature \textit{two-sided selection}, as opposed to \textit{one-sided selection} where always the same type is favored by selection. More precisely, we study the following SDE: 
\begin{align}
dX(s) = X(s-)(1-X(s-)) dL(s) + \sqrt{2X(s)(1-X(s))} dB(s), \label{levymodelsdesimp}
\end{align}
where the L\'evy process $L$ is defined as the sum of a compound Poisson process with jumps in $(-1,1)$ and of a non-positive drift (see Section \ref{diffusion}) and $B$ is an independent Brownian motion. $X(s)$ represents the proportion of individuals of type $0$ at time $s$ in the infinite population. The diffusion term in \eqref{levymodelsdesimp} represents the effect of neutral reproductions. The L\'evy process $L$ models the effect of selection and is called the environment. Its non-positive drift component (let us denote it by $-\sigma$) represents the rate at which individuals of type $1$ are subject to selective reproductions, thus modeling their permanent selective advantage. The resulting drift term $-\sigma X(s)(1-X(s)) ds$ is classical for continuous Wright-Fisher diffusions. Each positive (resp. negative) jump of $L$ represents an exceptional environmental event that favors individuals of type $0$ (resp. $1$). Let $(t_k,j_k)_{k \in I}$ be the Poisson point process of the jumps of $L$ (so that $\Delta L(t_k)=j_k$). The effect of these jumps may be heuristically understood as follows: at each time~$t_k$ such that $j_k>0$ (resp. $j_k<0$), type $0$ (resp. $1$) gets a high fitness equal to $|j_k \epsilon^{-1}|$ on the time interval $[t_k,t_k+\epsilon)$, which means the SDE has an additional term $j_k \epsilon^{-1} X(s)(1-X(s)) ds$ in this interval. For infinitesimal~$\epsilon$, this amounts to $\Delta X(t_k) = X(s-)(1-X(s-)) \Delta L(t_k)$, leading to the jump term in \eqref{levymodelsdesimp}. For simplicity, we work in a setting without mutations (but see Section \ref{othermodgene} where extensions to more general models are discussed, in particular to the case with mutations). The model considered in~\cite{cordvech} (in the case without mutations) can be recovered from~\eqref{levymodelsdesimp} by considering $1-X(s)$ with $L(t) := -J(t) - \sigma t$, where $\sigma$ is non-negative and $J$ has only jumps in $(0,1)$. Applications of our results to this particular case are discussed in Section~\ref{exapponesidedenv}.

Biological motivations of the model include understanding the effect, on evolution, of random occurrence of extreme events that provoke shifts in the type distribution of a population, and also understanding the combined effect of different kinds of selective pressures possibly acting in opposite directions. This includes bi-directional selective pressure induced by extreme events and permanent environmental conditions. More details on the biological relevance of the model and examples are discussed in Section~\ref{motbiol}. 

As time goes to infinity, a solution of \eqref{levymodelsdesimp} almost surely has a limit that belongs to $\{0, 1\}$, as proved in Proposition \ref{limht} of Section \ref{ccl}. The main object of study of the paper is the fixation probability associated with the SDE with jumps \eqref{levymodelsdesimp}, that is, 
\begin{align}
h(x) := \mathbb{P} \left ( \lim_{t \rightarrow \infty} X(t) = 1 | X(0) = x \right ). \label{defh(x)}
\end{align}
In other words, $h(x)$ is the probability that type $0$ eventually takes over the entire population. Such a quantity is of interest for biologists in several contexts like evolutionary rescue, competition with an invasive species, or Muller's ratchet. In the study of $h(x)$, our main tool is a slight modification of the so-called Ancestral Selection Graph (ASG) associated with~\eqref{levymodelsdesimp}. Its combinatorial properties are at the center of our study and they allow us to circumvent difficulties caused by the unavailability of classical genealogical methods (see Section \ref{classicalmomdual} for more details) leading to results in the involved case of two-sided selection. More precisely, we proceed as follows. As a first step, we encode the relevant combinatorics of the ASG into a function. This allows to establish a duality between moments of the jump-diffusion \eqref{levymodelsdesimp} and coefficients defined in terms of the encoding function (Theorem \ref{momdiff}). We then establish a semigroup property for those coefficients (Proposition \ref{semigroupprop}) and determine their small-time behavior (Lemma \ref{asymptcoeffnew}). This allows to show that those coefficients satisfy a system of linear differential equations (Theorem \ref{equadiffcoeffnew}). We also prove that those coefficients converge as time goes to infinity (Theorem \ref{cvcoeff}). Combining the above steps we then obtain a series representation for $h(x)$ as an infinite sum of polynomials whose coefficients satisfy explicit linear relations (Theorem \ref{finalformula}). Finally, we derive Taylor expansions of every order for $h(x)$ near $x=0$ and provide an explicit recursion formula for the coefficients (Theorem \ref{taylorexp}). 

To sum up, our motivation to study models with two-sided selection is twofold. On the one hand, mathematical challenges arising in the study of these models require the development of new methods (see Section~\ref{classicalmomdual}). On the other hand, SDE~\eqref{levymodelsdesimp} captures realistic situations arising in important practical problems (see Section~\ref{motbiol}). 

\subsection{Limit of finite population models in L\'evy environment} \label{finitepopmodel}

The heuristic justification for the different terms in the SDE \eqref{levymodelsdesimp} can be made rigorous by passing through the Moran model in L\'evy environment. In the latter, the ASG also arises naturally~(see Section \ref{asg}). 
We now describe this finite population model and explain its relation with \eqref{levymodelsdesimp}. 

\smallskip

Consider a population of size $N$ with two types, type $0$ and type $1$, subject to random reproduction and environmental effects. The environment is modeled by a Poisson point process $(t_k,j_k)_{k \in I}$ on $[0,\infty) \times (-1,1)$ with intensity measure $dt \times \nu_N$, where $\nu_N$ is a finite measure on~$(-1,1)$. The population undergoes the following dynamic. Individuals of type $0$ reproduce (neutrally) at rate $1$. Type-$1$ individuals reproduce at rate $1+\sigma_N$, where $\sigma_N\geq0$. Here, the rate for \emph{neutral reproductions} is~$1$, and $\sigma_N$ for \emph{selective reproductions}. In addition, at each time $t_k$ for $k \in I$ such that $j_k >0$ (resp. $j_k <0$), each individual of type $0$ (resp. $1$) reproduces with probability $|j_k|$, independently from the others. At any reproduction time: (a) each individual produces at most one offspring which inherits the parent's type, and (b) if  $n$ individuals are born, $n$ individuals are randomly sampled without replacement from the extant population to die, hence keeping the size of the population constant. For any $s \geq 0$, let $L^N(s) := -\sigma_N s + \sum_{k:\,t_k \leq s} j_k$. Clearly, $L^N$ is the sum of a compound Poisson process with jumps in $(-1,1)$ and of a non-positive drift (and is thus a L\'evy process). A trajectory of $L^N$ contains all the information on $(t_k,j_k)_{k \in I}$; therefore, for convenience, we also refer to $L^N$ as the environment.  We see that positive (resp. negative) jumps of $L^N$ give a selective advantage to type $0$ (resp. $1$). 

The Moran model admits a classical graphical representation, see Fig.~\ref{particlepicture}. Here, each of the~$N$ individuals is represented by a horizontal line. Time runs from left to right. A (potential) reproduction is represented by an arrow from the (potential) parent to the (potential) offspring. There are three types of arrowheads: triangle, filled star-shaped, or unfilled star-shaped. Neutral reproductions are represented by the triangle heads. Potential selective reproductions that favor type~$0$ (resp. type $1$) are represented by the filled (resp. unfilled) star-shaped heads. Arrows with triangle arrowhead appear for each ordered pair of lines independently at rate $1/N$. Arrows with unfilled star-shaped arrowhead occur on each ordered pair of lines independently at rate $\sigma_N/N$. In addition, for each $k \in I$, each line is independently included into a (random) set $S_k$ with probability $|j_k|$. Let $\tilde S_k$ be a set of lines chosen uniformly at random among all sets of lines having the same cardinality as $S_k$. Among all matchings between $S_k$ and~$\tilde S_k$, choose one uniformly. Then, if $j_k >0$ (resp. $j_k <0$), draw at time $s=t_k$ arrows from elements of $S_k$ to elements $\tilde S_k$ according to the matching; the arrowheads are filled (resp. unfilled) star-shaped. Each arrow with filled (resp. unfilled) star-shaped arrowhead corresponds to an actual reproduction event only if the line at the arrowtail has type~$0$ (resp. type $1$), and it is void otherwise. 

\begin{figure}[t!]
    \centering
    \scalebox{0.8}{\begin{tikzpicture} 
        
        \draw[dashed, opacity=0.4] (1,-0.2) --(1,1) (1,2)--(1,3);
        \draw[dashed, opacity=0.4] (6.2,-0.2) --(6.2,0) (6.2,1) --(6.2,2);
        
        \node [right] at (-0.2,-0.5) {$0$};
        \node [right] at (0.8,-0.5) {$t_1$};
        \node [right] at (6,-0.5) {$t_2$};
        \node [right] at (8.3,-0.5) {$T$};
        
        \draw[-{triangle 45[scale=5]},thick] (4.5,2) -- (4.5,1) node[text=black, pos=.6, xshift=7pt]{};
        \draw[thick] (0,1) -- (2.35,1);
        \draw[thick] (2.65,1) -- (6.05,1);
        \draw[thick] (6.35,1) -- (8.5,1);
        \draw[thick] (0,2) -- (6.05,2);
        \draw[thick] (6.35,2) -- (8.5,2);
        \draw[thick] (8.5,3) -- (0,3);
        \draw[thick] (0,4) -- (1.55,4);
        \draw[thick] (1.85,4) -- (7.45,4);
        \draw[thick] (7.75,4) -- (8.5,4);
        \draw[thick] (0,0) -- (6.85,0);
        \draw[thick] (7.15,0) -- (8.5,0);
        
        \draw[-{triangle 45[scale=5]},thick] (.4,1) -- (.4,0);
        \draw (7.6,1) -- (7.6,3.9);
        \node [star, star points=5, star point ratio=2.25, draw][scale=0.3](key) at (7.6,4) {S};
        \draw (1.7,2) -- (1.7,3.9);
        \node [star, star points=5, star point ratio=2.25, draw][scale=0.3](key) at (1.7,4) {S};
        \draw (1,3) -- (1,4);
        \node [star, star points=5, star point ratio=2.25, fill=black, draw][scale=0.3](key) at (1,4) {S};
        \draw (1,1) -- (1,2);
        \node [star, star points=5, star point ratio=2.25, fill=black, draw][scale=0.3](key) at (1,2) {S};
        \draw (6.2,4) -- (6.2,2.2);
        \node [star, star points=5, star point ratio=2.25, draw][scale=0.3](key) at (6.2,2) {S};
        \draw (6.2,0) -- (6.2,0.9);
        \node [star, star points=5, star point ratio=2.25, draw][scale=0.3](key) at (6.2,1) {S};
        \draw (2.5,0) -- (2.5,0.9);
        \node [star, star points=5, star point ratio=2.25, draw][scale=0.3](key) at (2.5,1) {S};
        \draw (7,2) -- (7,0.2);
        \node [star, star points=5, star point ratio=2.25, draw][scale=0.3](key) at (7,0) {S};
        \draw[-{triangle 45[scale=5]},thick] (3,3) -- (3,1);
        
        \draw[-{angle 60[scale=5]}] (3.25,-0.6) -- (5.25,-0.6) node[text=black, pos=.5, yshift=6pt]{};
        
        %%%%mutations%%%%%%
%        \node[ultra thick] at (3.5,4) {$\bigtimes$} ;
%        \node[ultra thick] at (6.6,4) {$\bigtimes$} ; 
%        \draw (2.1,2) circle (1.5mm)  [fill=white!100];  
%        \draw (4.1,0) circle (1.5mm)  [fill=white!100];
        
        %%%%%%%types%%%%%
        \node [right] at (-0.5,0) {$1$};
        \node [right] at (-0.5,1) {$1$};
        \node [right] at (-0.5,2) {$1$};
        \node [right] at (-0.5,3) {$0$};
        \node [right] at (-0.5,4) {$1$};
        \node [right] at (8.5,0) {$1$};
        \node [right] at (8.5,1) {$1$};
        \node [right] at (8.5,2) {$1$};
        \node [right] at (8.5,3) {$0$};
        \node [right] at (8.5,4) {$0$};

        \end{tikzpicture}    }
    \caption{A realization of the Moran interacting particle system with $N=5$. Time $s$ runs forward from left to right. The environment has a jump at time $s=t_1$ that favors type $0$ and a jump at time $s=t_2$ that favors type $1$.}
    \label{particlepicture}
\end{figure}
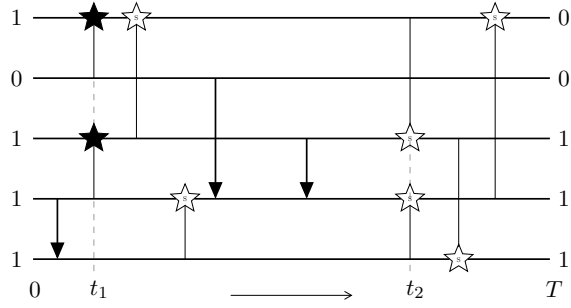

Let $X_N(s)$ denote the proportion of individuals of type $0$ at time $s$. Proposition~\ref{propcv} of Section~\ref{diffusion} says that there is a L\'evy process $L$ such that $X_N$ converges, as $N$ goes to infinity and after appropriate renormalization, to the solution of \eqref{levymodelsdesimp}. 

\subsection{Classical genealogical techniques and difficulties with two-sided selection} \label{classicalmomdual}

Classically, $h(x)$ is studied using the Ancestral Selection Graph (ASG).  
%In the case of the diffusion \eqref{onesidedmodelsde}, the dynamic of the ASG is as follows: it starts at time $0$ with a number $n \geq 1$ of lines, coalescence of any pair of two lines occurs at rate $2$, each line branches at rate $\sigma$ and, at each jump $\Delta J(r)$ of $J$, each line branches, independently from the others, with probability $\Delta J(r)$. 
The intuition behind this object and its rigorous definition are given in Section \ref{asg}. 
%at time $0$. Assigning \textit{iid} types at time $-t$ (with probability $x$ to be of type $0$ and $1-x$ to be of type $1$) and propagating them, 
%following the rules, 
%this allows to determine the type of the randomly chosen individual from time $t$. We denote by $h_t(x)$ the probability that the type, determined by this procedure, of the chosen individual from time $0$ is $0$. It can be seen that $h_t(x)$ converges to $h(x)$ when $t$ goes to infinity. An important point is thus to be able to relate $h_t(x)$ to the dynamic of the ASG. 
In the case of one-sided selection studied in \cite{cordvech}, a \emph{moment duality} holds; that is,  
\begin{equation}\label{mdsm}
\forall\, x\in[0,1],\, n \geq 1,\qquad \mathbb{E}[X(t)^n|X(0)=x]=\mathbb{E}[x^{G_t}|G_0=n], 
\end{equation}
where $G_t$ denotes the number of lines in the ASG at instant $t$ \cite[Thm.~2.3 applied to $1-X$]{cordvech}. 
%(\textbf{Expliquer ce truce du temps qui run forward et backward}). 
%Evaluating \eqref{mdsm} at $n=1$ we obtain $\mathbb{E}[X(t)|X(0)=x] = \mathbb{E}[x^{G_t}|G_0=1]$. By dominated convergence, t
Such a relation usually forms the core of a \emph{genealogical technique}. On one hand establishing such a relation is classically the way to relate rigorously a Wright-Fisher diffusion with its ASG. On the other hand having such a relation allows to analyze the long-time behavior of $X$ through $(G_t)_{t \geq 0}$. Indeed, as $t$ goes to infinity, the left-hand side of \eqref{mdsm} converges to 
\[ \mathbb{E}[X(\infty)^n|X(0)=x]=\mathbb{P}(X(\infty)=1|X(0)=x)=h(x), \]
where we have written $X(\infty)$ for $\lim_{t \rightarrow \infty} X(t)$, whereas the right-hand side converges to $\mathbb{E}[x^{G_{\infty}}]$ with $G_{\infty}$ being a random variable that follows the stationary distribution of $(G_t)_{t \geq 0}$. This allows us to write 
\begin{eqnarray} 
h(x) = 1-\sum_{n \geq 0} a_n (1-x)x^{n}, \label{mdsm3}
\end{eqnarray}
where $a_n = \mathbb{P}(G_{\infty} > n)$. These coefficients are known to satisfy a recurrence relation \cite[Eq.~(2.18) with $\theta=0$]{cordvech}. 

For the general Wright-Fisher diffusion~\eqref{levymodelsdesimp} it is still possible to define an ASG; we do so in Section \ref{asg}. We can also rigorously relate \eqref{levymodelsdesimp} with its ASG, but in a way that is more abstract than \eqref{mdsm}; this is done in Section \ref{relwf-asgbis}, but see Section \ref{relwf-asg} for a heuristic. The difference with the model studied in \cite{cordvech} is that there are two types of branchings in the ASG of \eqref{levymodelsdesimp}: 
%(POUR QUE CA SOIT COMPREHENSIBLE IL FAUT PARLER DES assignmentS OF TYPES): 
those favoring type $0$ and those favoring type $1$ (see Section \ref{asg}), while only one type of branchings in the ASG of the model of \cite{cordvech}. 
%More precisely the ASG is defined as follows: coalescence of any pair of two lines occurs at rate $2$ and, at each jump $\Delta L(r)$ of $L$, each line branches, independently from the others, with probability $|\Delta L(r)|$. If $\Delta L(r) > 0$, it represents an environmental event favoring individuals of type $0$. If $\Delta L(r) < 0$, it represents an environmental event favoring individuals of type $1$. 
As a result, the structure of the information contained in the ASG of \eqref{levymodelsdesimp} is much more complicated and we have to take the whole combinatorics of this ASG into account. In fact, it is no longer possible to express the moments of $X(t)$ via the distribution of the number of lines in its ASG, not even after allowing modifications of the ASG. In situations with one-sided selection and mutations, modifications of the ASG were successfully used for deriving the common ancestor's type distribution (one would then work with a pruned LD-ASG \cite[Sect.~2.6 ]{cordvech}), or for the long run type distribution (one would then work with the killed ASG \cite[Sect.~2.5]{cordvech}). However, in general, it will not work to extend \eqref{mdsm} and derive a representation analogous to \eqref{mdsm3} for $h(x)$ in the setting of~\eqref{levymodelsdesimp}, with coefficients $a_n$ that are probabilities related to a modification of the ASG. Indeed, we show that
%, instead of $G_{\infty}$, a random variable distributed as the stationary distribution of a modification of the ASG. 
the coefficients appearing in a Taylor expansion of $h(x)$ are not always probabilities as some of them can be negative (see Theorem \ref{taylorexp} and Remark \ref{calcexprbk} in Section \ref{mainresults}). Moreover, it seems that the Taylor series of $h(x)$ at $0$ can in general be divergent (see Subsection \ref{numappl}). 

It is natural to study models with two-sided selection. However, as explained above, the lack of the classical moment duality in this general case prevents us from using genealogical techniques and leads to serious difficulties in the analysis. In particular, studying such models requires a new set of methods. In this spirit, we propose the combinatorial approach outlined above. 
%An interesting fact is that, even though the coefficients appearing in our decomposition of $h(x)$ are not limits of probabilities associated with a Markov chain, they are the limits of coefficients of a non-stochastic semigroup. 
This method seems relatively robust and can be extended to more general models. For example, in Section \ref{othermodgene} we explain how the ideas can be adapted to the case of an inhomogeneous environment, the case with mutations, and the case of a population divided into several colonies. 

\subsection{Biological motivations} \label{motbiol}

In this subsection we discuss some biological considerations that motivate the different aspects of~\eqref{levymodelsdesimp} where fluctuating selection is driven by a L\'evy environment with jumps of both signs. Moreover, we provide examples of biological situations that can be captured by the model and describe the corresponding parameters settings. 

Determining the impact of extreme events on evolution is of high relevance in biology~\cite{motbiol}. Here, extreme events refer to strong perturbations of the environment that are relatively rare and punctual, but that may influence long-term evolution. Examples include heat waves, freezing events, floods, droughts, exceptional rain falls, hurricanes, fires, pest outbreaks, etc. 
%It may be tempting to consider those events as rare abnormalities and therefore to neglect them. However, they may have lasting effects over a population and it becomes more and more apparent that their 
\cite{motbiol} makes the distinction between two types of environmental perturbations: \textit{pulses} that are episodic and \textit{presses} that are prolongated. This motivates having a model where there are two types of environmental influence that occur on two different time scales. In our case, the drift of the L\'evy process $L$ corresponds to the constant selective advantage of one type (due to the environment in normal conditions, or to a prolongated environmental perturbation). The Poisson point process $(t_k,j_k)_{k \in I}$ of the jumps of $L$ models punctual extreme events that have an immediate impact on the type frequency in the population (see below \eqref{levymodelsdesimp}). By nature, occurrences and effects of extreme events are random, so it makes sense to model them by a Poisson point process. 

Whether some events can be considered to be punctual depends on the time scale over which the population is observed, on the speed of evolution between those events, or also on the generation time of the species involved. There are documented examples of episodic events that caused non-negligible genetic shifts in populations, without immediate return of the population to its initial state after the event. One is the effect of a heat wave in Europe, in spring 2011, on Drosophila subobscura \cite{motbiol, droso}. Another is the effect of an algal bloom along the California coast, in 2011, on abalone \cite{motbiol, alga}. It seems that for such events, a mathematical model with instantaneous jumps of type frequencies is relevant. Even if, in a large time scale, an extreme event is considered as punctual, it may in practice span over a few generations. During this brief period, the least affected type has much higher fitness than the other type, of which many individuals do not survive long enough to reproduce. As explained below \eqref{levymodelsdesimp}, such a greatly enhanced fitness over a brief period may lead to the jump mechanism of \eqref{levymodelsdesimp}. 
%It thus seems reasonable, in \eqref{levymodelsdesimp} and in the supporting finite population model (see Section \ref{finitepopmodel}), to model extreme event by an instantaneous time with many reproductions that respect the usual replacement mechanism, but where only one type reproduces. 

In nature, it may very well happen that the selective pressure induced by extreme events acts opposite to the one induced by normal environmental conditions. In the model \eqref{levymodelsdesimp}, this corresponds to $\sigma>0$ 
%which models a selective pressure in normal conditions that is favorable to type $1$, 
and $L$ having positive jumps. 
%that model extreme events favorable to type $0$. 
An example that seems to correspond to such a situation is as follows. In Alaska, seeds of sedge Eriophorum vaginatum from the south were planted further north and compared with local plants \cite{motbiol, alaska}. It was first observed that normal conditions favored the local type, but that the southern type was then favored in turn, seemingly due to increasing frequency of heat waves. Another example is that of infectious agents getting resistant to medicine. Usually, normal conditions favor non-resistant individuals because their metabolism is optimally adapted. Then, massive use of medicine is an extreme event favoring resistant individuals. 

A particularly important feature of the model \eqref{levymodelsdesimp} is to allow both positive and negative jumps for $L$ (and, therefore, for type frequencies). A biological motivation for this is the abundance of situations where the environment is subject to several types of extreme events with opposing selective effects on populations. \cite{motbiol} mentions in particular cases with succession of abundant rains and intensive droughts and the effects of those events on allele frequencies in bird populations and on species frequencies in plant populations. We can also mention the evolution of the frequency of a type, in populations of Mediterranean wild thyme, that is frost-sensitive but summer drought-tolerant \cite{motbiol, thyme}. Two types of extreme events (freezing events and droughts) have opposing selective effects on such populations. Another example of extreme events of two types having opposing effects on plant populations is given by peaks of abundance or rarefaction of herbivores \cite{motbiol, herbivores1, herbivores2}. 

\subsection{Organization of the paper} \label{orgpap1}

The rest of the paper is organized as follows. In Section \ref{toolsmethres} we introduce the main objects that we use all along the paper and state our main results. In Section \ref{exappgen} we apply our results in some simple or particular cases and discuss some extensions to more general models. Section \ref{combinatorialfuct} is mainly dedicated to studying the combinatorics of the ancestral structure. In Section \ref{acotcfnew} we study thoroughly the coefficients that form the dual of the diffusion \eqref{levymodelsdesimp} and then prove the series representation of the fixation probability $h(x)$. Section \ref{dlh(x)} is dedicated to the Taylor expansions of $h(x)$. Some technical proofs are given in Appendix~\ref{append}. A table of notations is given in the end of the paper. More details about the content of sections are given in the end of Section \ref{toolsmethres}. 

\section{Main tools, methods, and results} \label{toolsmethres}

In this section we describe the main objects that we use in our analysis and state our main results. More precisely, in Subsection~\ref{diffusion} we state some general facts about the jump-diffusion \eqref{levymodelsdesimp}. In Subsection \ref{asg} we define the ASG and provide some intuition for it. In Subsection \ref{relwf-asg} we describe the relation between \eqref{levymodelsdesimp} and the ASG (the rigorous relation between the two objects is established later in Section \ref{relwf-asgbis}). In Subsection \ref{enlargedasg} we introduce an \textit{Enlarged ASG} 
%and the labels. In Subsection \ref{fctenlargedasg} we introduce
and the main tool for our analysis, a function of the Enlarged ASG that encodes the relevant information of its combinatorics. In Subsection \ref{mainresults} we state our main results including the announced representation of the fixation probability and its Taylor expansions near $0$. 
%In Subsection \ref{numappl} we present some numerical applications of our results. 
%In Subsection \ref{orgpap} we discuss the organization of the rest of the paper. 

\subsection{The jump-diffusion} \label{diffusion}

The relation between the Wright-Fisher diffusion \eqref{levymodelsdesimp} and the Moran model defined in Section \ref{finitepopmodel} is done by the following convergence result that can be obtained as a generalization of Theorem 2.2 of \cite{cordvech}: 
\begin{prop} \label{propcv}
Let $J$ be a compound Poisson process with jumps in $(-1,1)$. Assume that $X_N(0)\rightarrow x$ and $N \sigma_N\rightarrow \sigma$ for some $\sigma \geq 0$, as $N\to\infty.$
Then the type-frequency process $(X_N(s))_{s\geq 0}$ in environment $L^N(s):=-\sigma_N s + J(s/N)$ converges in distribution to $(X(s))_{s\geq 0}$ where $X$ is the solution of \eqref{levymodelsdesimp} with $X(0)=x$ and $L(s) := -\sigma s + J(s)$. 
\end{prop}
Throughout this paper, we fix $\sigma \geq 0$, $\nu$ a finite measure on $(-1,1)$, and $\lambda := \nu((-1,1))$. Let $(t_k,j_k)_{k \in I}$ be a Poisson point process on $[0,\infty) \times (-1,1)$ with intensity measure $dt \times \nu$. We define the L\'evy process $L$ by $L(s) := -\sigma s + \sum_{k:\,t_k \leq s} j_k$. We study the Wright-Fisher diffusion \eqref{levymodelsdesimp} with initial condition $X(0)=x \in (0,1)$. Existence and pathwise uniqueness of the solution to \eqref{levymodelsdesimp} are classical and can be proved similarly as in Proposition 3.3 of \cite{cordvech}. 
%or, alternatively, since we assume the L\'evy measure of $L$ to be finite, the set of jumping times of $L$ is almost surely discrete so existence and pathwise uniqueness for \eqref{levymodelsdesimp} follows from the well-known existence and pathwise uniqueness for the SDE \eqref{sdewoj} below and interlacing. 
We define the annealed probability measure $\mathbb{P}(\cdot)$ as the law of $(X(s))_{s \geq 0}$. $\mathbb{E}[\cdot]$ denotes the associated expectation. 

Next, we define the jump-diffusion in a quenched setting, that is, in a fixed deterministic environment. Since $L$ is the sum of a compound Poisson process with jumps in $(-1,1)$ and of a non-positive drift, any realization of $L$ is a c\`ad-l\`ag piecewise linear function (with slope $-\sigma$) with finitely many jumps on finite intervals, and all jump sizes in $(-1,1)$. We fix such a function $\omega = (\omega(s))_{s \geq 0}$ and refer to it as a fixed environment. Let $(t^{\omega}_n)_{n \geq 1}$ be the discrete sequence of the jumping times of $\omega$. For convenience, set $t^{\omega}_0 := 0$. The jump-diffusion in the fixed environment $\omega$ is denoted by $(X(\omega,s))_{s \geq 0}$ and defined as follows: $X(\omega,t^{\omega}_0)=x$ and for $i \geq 1$, $(X(\omega,s))_{s \in [t^{\omega}_{i-1},t^{\omega}_i)}$ is distributed as a solution of 
\begin{align}
dX(r) = -\sigma X(r)(1-X(r)) dr + \sqrt{2X(r)(1-X(r))} dB(r), \label{sdewoj}
\end{align}
with initial value $X(\omega,t^{\omega}_{i-1})$ and $X(\omega,t^{\omega}_i):= X(\omega,t^{\omega}_i-)(1-X(\omega,t^{\omega}_i-)) \Delta \omega(t^{\omega}_i)$. The quenched probability measure $\mathbb{P}^{\omega}(\cdot)$ is defined as the law of $(X(\omega,s))_{s \geq 0}$. $\mathbb{E}^{\omega}[\cdot]$ denotes the associated expectation. Note that $\mathbb{P}^{\omega}(\cdot)$ is the law of \eqref{levymodelsdesimp} conditionally on $L=\omega$. More precisely, if $P(\cdot)$ denotes the law of $(L(s))_{s \geq 0}$, then 
\begin{eqnarray}
\int \mathbb{P}^{\omega}(\cdot) P(d \omega) = \mathbb{P}(\cdot). \label{qtoannealed}
\end{eqnarray}

\subsection{The ASG} \label{asg}

The ASG is a Markovian graph-valued process that was introduced by Krone and Neuhauser \citep{KroNe97,NeKro97}. The idea behind this object is to start at an instant $s=T$ with a finite number of lines that represent randomly chosen individuals in the infinite population and, by analogy with the Moran model, to draw lines of potential ancestors. 
%This can be made precise even in the diffusion limit of a Moran model. 
Let us explain the intuition behind the ASG of the diffusion \eqref{levymodelsdesimp} using the Moran model defined in Section \ref{finitepopmodel} and represented in Figure \ref{particlepicture}. Consider a realization of the Moran model on $[0,T]$ and a sample of $l$ lines at time $T$. Then go backward in time, that is, from right to left in Figure \ref{particlepicture}, to trace the lines of their potential ancestors, ignoring the types. 

We observe the following dynamic. When two potential ancestors are connected by an arrow with triangle-shaped head, they are both replaced by the single line at the tail of the arrow; that is, the two lines coalesce. When a potential ancestor is connected by an arrow with triangle-shaped head to a line that is outside the set of current potential ancestors, the potential ancestor, if it is at the tip of the arrow, is replaced by the line at the tail of the arrow. When a potential ancestor is hit by an arrow with filled (resp. unfilled) star-shaped head, the ancestor of that potential ancestor is either the \textit{incoming line} at the tail or the \textit{continuing line} at the tip. Which one is the actual ancestor depends on the type of the incoming branch. For the moment, we ignore types so the incoming and continuing lines become (if not already) potential ancestors. If the incoming line was not already a potential ancestor, we observe a \emph{branching}, in the sense that the initial potential ancestor splits into two potential ancestors. If the incoming line was already a potential ancestor, we observe a \emph{collision}. This procedure defines a dynamical graph, the \textit{Moran-ASG in $[0,T]$}, that contains all the lines that are potentially ancestral to the $l$ lines chosen at time $s=T$. 

Intuitively, the ASG associated to~\eqref{levymodelsdesimp} traces back potential ancestors in the infinite population limit of the Moran model. It is thus natural to define this ASG as the Markovian graph-valued process whose transition rates are the limits of the transition rates of the Moran-ASG (after speeding up time by $N$ as in Proposition \ref{propcv}). This motivates the following definition. 

\begin{defi}[The quenched/annealed ASG] \label{defquenchedasg}
Let $\omega = (\omega(s))_{s \geq 0}$ be a fixed environment, $T > 0$ and $l\geq 1$. The quenched ASG on $[0,T]$ in environment $\omega$ starting with $l$ lines is the branching-coalescing particle system denoted by $(A^{\omega,T}_{s})_{s \in [0,T]}$ and defined as follows. It starts with $l$ lines at time $s=T$ (i.e. $A^{\omega,T}_{T}$ contains $l$ lines) and, between jumping times of $\omega$, has the following dynamic as $s$ decreases: 
\begin{itemize}
\item[(i)] Any pair of lines coalesces into a single line at rate $2$, independently from other pairs. 
\item[(ii)] Any line splits into two lines, an incoming line and a continuing line, at rate $\sigma$, independently from other lines. We refer to this as a \textit{single branching favoring type $1$}. 
\end{itemize}
Additionally, if at a time $s\in[0,T]$ we have $\Delta \omega(s)> 0$ (resp. $\Delta \omega(s)< 0$), then $A^{\omega,T}_{s-}$ is obtained from $A^{\omega,T}_{s}$ as follows: 
\begin{itemize}
\item[(iii)] Every line of $A^{\omega,T}_{s}$, independently from the others, splits with probability $|\Delta \omega(s)|$ into two lines, an incoming line and a continuing line.
%, or, with probability $1-|\Delta \omega(s)|$, to remain unchanged. 
We call this a \textit{simultaneous branching favoring type $0$} (resp. $1$). 
\end{itemize}
Let $l \geq 1$. The annealed ASG starting with $l$ lines is the branching-coalescing particle system denoted by $(A_{\beta})_{\beta \geq 0}$ and defined as follows. It starts with $l$ lines at time $\beta=0$ (i.e. $A_0$ contains $l$ lines) and, as $\beta$ increases, it satisfies (i), (ii) and 
\begin{itemize}
\item[(iii')] If there are currently $n$ lines in the system, for any $k \in \{1,...,n\}$ and any group of $k$ lines independently at rate $\int_{(0,1)}|y|^k (1-|y|)^{n-k}\nu(dy)$ (resp. $\int_{(-1,0)}|y|^k (1-|y|)^{n-k}\nu(dy)$), any line in the group branches into two: an incoming line and a continuing line. We refer to this as a \textit{simultaneous branching favoring type $0$} (resp. $1$).
\end{itemize}
\end{defi}
We denote by $\mathbb{P}^{\omega,T}_l(\cdot)$ (resp. $\mathbb{P}_l(\cdot)$) the probability measure associated with $(A^{\omega,T}_{s})_{s \in [0,T]}$ (resp. $(A_{\beta})_{\beta \geq 0}$), and $\mathbb{E}^{\omega,T}_l[\cdot]$ (resp. $\mathbb{E}_l[\cdot]$) is the associated expectation. For the quenched ASG, the environment is fixed and the process evolves backward in time (it starts at $s=T$ and ends at $s=0$). In the annealed case, the environment is random. Since the Poisson point process defining the environment has the same law in forward and backward directions, and since the dynamic of the annealed ASG does not depend on the starting time $T$, the annealed ASG is defined as a process $(A_{\beta})_{\beta \geq 0}$ that has its own timeline. Running $(A_{\beta})_{\beta \geq 0}$, starting with $l$ lines at time $\beta=0$, until time $\beta = T$, corresponds heuristically to choosing uniformly at random $l$ individuals at time $s=T$ in the infinite population of the model \eqref{levymodelsdesimp} and then tracing back their potential ancestors until time $s=0$. In other words, the timelines $s$ and $\beta$ run in opposite directions. 

%In the case of the model \eqref{levymodelsdesimp}, we can still define the ASG. The difference with the model \eqref{onesidedmodelsde} is that there are now two types of star-shaped arrows: those that can only be used by individuals of type $0$ and those that can only be used by individuals of type $1$. More precisely t
%Recall that, for our model, the ASG is defined as follows: the ASG starts with $l \geq 1$ lines, coalescence of any pair of two lines occurs at rate $2$ and, at each jump $\Delta L(r)$ of $S$, each line branches, independently from the others, with probability $|\Delta L(r)|$. If $\Delta L(r) > 0$, it represents a selective event favoring individuals of type $0$. If $\Delta L(r) < 0$, it represents a selective event favoring individuals of type $1$. 

In the graphical representation of the ASG in Figure~\ref{ASGpicture}, we use the same convention as for the Moran model: two lines involved into a coalescence event are joined by an arrow with triangle-shaped head, and branchings that favor type $0$ (resp. $1$) are represented by arrows with filled (resp. unfilled) star-shaped head. More precisely, a line subject to a branching turns into a continuing line and an incoming line appears. The arrow with star-shaped head goes from the incoming line to the continuing line. 

\begin{figure}[t!]
    \centering
    \scalebox{0.8}{\begin{tikzpicture}   
       %%%%%jumps of the environment%%%%  
        \draw[dashed,thick,opacity=0.3] (1,-0.5) --(1,1) (1,2)--(1,3) (1,4)--(1,4.5);
        \draw[dashed,thick,opacity=0.3] (6.2,-0.5) --(6.2,0) (6.2,1) --(6.2,2) (6.2,4)--(6.2,4.5);
   %%%%%time%%%%%%%%%%%%%%%%%%%%%%     
        \node [right] at (-0.2,-0.8) {$0$};
        \node [right] at (0.7,-0.8) {$t_1$};
        \node [right] at (5.9,-0.8) {$t_2$};
        \node [right] at (8.3,-0.8) {$T$};
        \node [right] at (3.5,-0.8) {$s$};
       
        \node [right] at (-0.2,4.8) {$T$};
        \node [right] at (0.4,4.8) {$T-t_1$};
        \node [right] at (5.6,4.8) {$T-t_2$};
        \node [right] at (8.3,4.8) {$0$};
        \node [right] at (3.5,4.8) {$\beta$};
       
        \draw[-{triangle 45[scale=5]}] (2.5,-0.5) -- (4.5,-0.5) node[text=black, pos=.6, xshift=7pt]{};
        \draw[-{triangle 45[scale=5]}] (4.5,4.5) -- (2.5,4.5) node[text=black, pos=.6, xshift=7pt]{};
       %%%%%%%non-ancestral lines%%%%%%%%%%%%      
%        \draw[thick,opacity=0.3] (0,1) -- (8.5,1);
%        \draw[thick,opacity=0.3] (8.5,2) -- (0,2);
%        \draw[thick,opacity=0.3] (8.5,3) -- (0,3);
%        \draw[thick,opacity=0.3] (0,0) -- (8.5,0);
%        \draw[thick,opacity=0.3] (0,4) -- (8.5,4);
   %%%%%%%%ancestral lines%%%%%%%%%%%%%%%%%     
        \draw[thick] (0,4) -- (6.2,4);
        \draw[thick] (0,2) -- (6.05,2);
        \draw[thick] (6.35,2) -- (8.5,2);
        \draw[thick] (4.5,1) -- (6.05,1);
        \draw[thick] (6.35,1) -- (8.5,1);
        \draw[thick] (0.4,0) -- (6.2,0);
        \draw[thick] (1,1) -- (0.0,1);
        \draw[thick] (1,3) -- (0.0,3);
   %%%%%%%%%%%%% relevant arrows%%%%%%%%%%
        \draw[-{triangle 45[scale=5]},thick] (4.5,2) -- (4.5,1) node[text=black, pos=.6, xshift=7pt]{};
        \draw[-{triangle 45[scale=5]},thick] (0.4,1) -- (0.4,0);
%        \draw[-{open triangle 45[scale=5]},thick] (1.7,2) -- (1.7,4);
        \draw[thick] (1,3) -- (1,4);
        \node [star, star points=5, star point ratio=2.25, fill=black, draw][scale=0.3](key) at (1,4) {S};
        \draw[thick] (1,1) -- (1,2);
        \node [star, star points=5, star point ratio=2.25, fill=black, draw][scale=0.3](key) at (1,2) {S};
        \draw[thick] (6.2,4) -- (6.2,2.2);
        \node [star, star points=5, star point ratio=2.25, draw][scale=0.3](key) at (6.2,2) {S};
        \draw[thick] (6.2,0) -- (6.2,0.9);
        \node [star, star points=5, star point ratio=2.25, draw][scale=0.3](key) at (6.2,1) {S};
   %%%%%%%%%%%%irrelevant arrows%%%%%%%%%
%        \draw[-{triangle 45[scale=5]},thick, opacity=0.3] (3,3) -- (3,1);
%        \draw[-{open triangle 45[scale=5]},thick, opacity=0.3] (7.6,1) -- (7.6,4);
%        \draw[-{open triangle 45[scale=5]},thick, opacity=0.3] (2.5,0) -- (2.5,1);
%        \draw[-{open triangle 45[scale=5]},thick, opacity=0.3] (7,2) -- (7,0);
          %%%%%mutations%%%%%%
%        \node[ultra thick] at (3.5,4) {$\bigtimes$} ;
%        \draw[thick] (2.1,2) circle (1.5mm)  [fill=white!100];  
%        \draw[thick] (4.1,0) circle (1.5mm)  [fill=white!100];
%        \node[ultra thick,opacity=0.3] at (6.6,4) {$\bigtimes$} ; 
        \end{tikzpicture}    }
    \caption{A realization of the ASG. The timeline of $X$ runs from left to right and the timeline of the ASG from right to left. The environment has a positive jump at time $s=t_1$ and a negative jump at time $s=t_2$.}
    \label{ASGpicture} 
\end{figure}
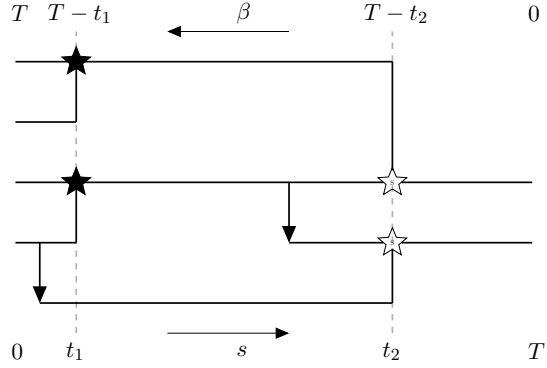

\subsection{Relation between Wright-Fisher diffusion and ASG} \label{relwf-asg}

In this subsection we explain in which way the Wright-Fisher diffusion \eqref{levymodelsdesimp} and the ASG from Definition \ref{defquenchedasg} are related. 

The interest in tracing back potential ancestors of a set of individuals via the ASG is that it allows to analyze their types. The analogy to the Moran model motivates the following type assignment procedure for lines of the ASG. 
\begin{defi}[Type assignment procedure for the ASG] \label{typeaspr}
For $T >0$ and $x \in [0,1]$, the type assignment procedure with initial condition $x$ for the annealed (resp. quenched) ASG on $[0,T]$ is defined as follows. 
\begin{itemize}
\item[(i)] At instant $\beta = T$ (resp. instant $s=0$) lines in the annealed (resp. quenched) ASG receive \textit{iid} types with law $x \delta_0 + (1-x) \delta_1$. 
\item[(ii)] Types propagate as $\beta$ decreases (resp. as $s$ increases). 
\item[(iii)] If a line resulting from a coalescence is of type $i \in \{0,1\}$, the two lines involved in the coalescence event receive type $i$. 
%\item If there are branchings favoring type $0$ at instant $\beta$ (resp. branchings resulting from a jump $\Delta \omega(s) > 0$ at instant $s$), for each of the individual branchings we have: If the incoming line is of type $0$, then the line that branches receives type $0$. If the incoming line is of type $1$, then the line that branches receives the same type as the continuing line. 
\item[(iv)] If, in a branching favoring type $i \in \{0,1\}$, the incoming line is of type $i$, then the line that branches receives type $i$. If the incoming line is of type $1-i$, then the line that branches receives the type of the continuing line. 
%\item If there are branchings favoring type $1$ at instant $\beta$ (resp. branchings resulting from a jump $\Delta \omega(s) < 0$ at instant $s$), for each of the individual branchings we have: If the incoming line is of type $1$, then the line that branches receives type $1$. If the incoming line is of type $0$, then the line that branches receives the same type as the continuing line. 
%\item In a branching favoring type $1$ we have: If the incoming line is of type $1$, then the line that branches receives type $1$. If the incoming line is of type $0$, then the line that branches receives the same type as the continuing line. 
\end{itemize}
\end{defi}

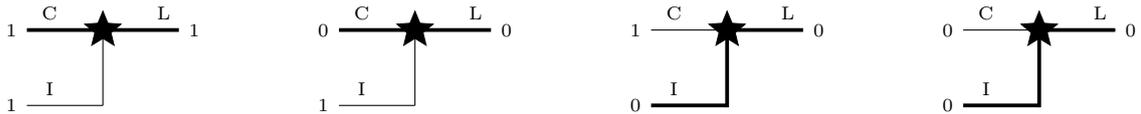
\begin{figure}[b!]
	\begin{minipage}{0.23 \textwidth}
		\centering
		\scalebox{1}{
			\begin{tikzpicture}
			\draw[line width=0.5mm] (0,1) -- (2,1);
			\draw[color=black] (0,0) -- (1,0);
			\node [star, star points=5, star point ratio=2.25, fill=black, draw][scale=0.3](key) at (1,1) {S};
			\draw[color=black] (1,0) -- (1,1) node[text=black, pos=.6, xshift=7pt]{};
			\node[above] at (1.8,1) {\tiny L};
			\node[above] at (0.3,1) {\tiny C};
			\node[above] at (0.3,0) {\tiny I};
			\node[left] at (0,1) {\tiny $1$};
			\node[left] at (0,0) {\tiny $1$};
			\node[right] at (2,1) {\tiny $1$};
			\end{tikzpicture}}
	\end{minipage}\hfill
	\begin{minipage}{0.23 \textwidth}
		\centering
		\scalebox{1}{
			\begin{tikzpicture}
			\draw[line width=0.5mm] (0,1) -- (2,1);
			\draw[color=black] (0,0) -- (1,0);
			\node [star, star points=5, star point ratio=2.25, fill=black, draw][scale=0.3](key) at (1,1) {S};
			\draw[color=black] (1,0) -- (1,1) node[text=black, pos=.6, xshift=7pt]{};
			\node[above] at (1.8,1) {\tiny L};
			\node[above] at (0.3,1) {\tiny C};
			\node[above] at (0.3,0) {\tiny I};
			\node[left] at (0,1) {\tiny $0$};
			\node[left] at (0,0) {\tiny $1$};
			\node[right] at (2,1) {\tiny $0$};
			\end{tikzpicture}}
	\end{minipage}\hfill
	\begin{minipage}{0.23 \textwidth}
		\centering
		\scalebox{1}{
			\begin{tikzpicture}
			\draw[] (0,1) -- (2,1);
			\draw[line width=0.5mm] (0,0) -- (1,0);
			\draw[line width=0.5mm] (1,1) -- (2,1);
			\node [star, star points=5, star point ratio=2.25, fill=black, draw][scale=0.3](key) at (1,1) {S};
			\draw[color=black,line width=0.5mm] (1,-0.025) -- (1,1) node[text=black, pos=.6, xshift=7pt]{};
			\node[above] at (1.8,1) {\tiny L};
			\node[above] at (0.3,1) {\tiny C};
			\node[above] at (0.3,0) {\tiny I};
			\node[left] at (0,1) {\tiny $1$};
			\node[left] at (0,0) {\tiny $0$};
			\node[right] at (2,1) {\tiny $0$};
			\end{tikzpicture}}
	\end{minipage}\hfill
	\begin{minipage}{0.23 \textwidth}
		\centering
		\scalebox{1}{
			\begin{tikzpicture}
			\draw[] (0,1) -- (2,1);
			\draw[line width=0.5mm] (0,0) -- (1,0);
			\draw[line width=0.5mm] (1,1) -- (2,1);
			\node [star, star points=5, star point ratio=2.25, fill=black, draw][scale=0.3](key) at (1,1) {S};
			\draw[color=black,line width=0.5mm] (1,-0.025) -- (1,1) node[text=black, pos=.6, xshift=7pt]{};
			\node[above] at (1.8,1) {\tiny L};
			\node[above] at (0.3,1) {\tiny C};
			\node[above] at (0.3,0) {\tiny I};
			\node[left] at (0,1) {\tiny $0$};
			\node[left] at (0,0) {\tiny $0$};
			\node[right] at (2,1) {\tiny $0$};
			\end{tikzpicture}}
	\end{minipage}
	\caption{Case of a branching favoring type $0$. The line that branches (L) splits into the continuing line (C) and the incoming line (I). Since the branching favors type $0$, the incoming line transmits its type if and only if it is of type $0$.}
	\label{fig:peckingorderposjump}
\end{figure}

The first point in Definition \ref{typeaspr} is related to the initial condition $X(0)=x$ for \eqref{levymodelsdesimp}. The propagation rules are illustrated in Figure \ref{fig:peckingorderposjump} for the case of a branching favoring type $0$. We define the \textit{backward type distribution} as follows. 
\begin{defi}[Backward type distribution] \label{defquenchedasgbtd}
Let $T >0, l \geq 1$ and $x \in [0,1]$. We consider the annealed ASG $(A_{\beta})_{\beta \in [0,T]}$ (resp. the quenched ASG $(A^{\omega,T}_{s})_{s \in [0,T]}$) starting with $l$ lines at time $\beta=0$ (resp. time $s=T$), and apply the type assignment procedure on $[0,T]$ with initial condition $x$ (see Definition \ref{typeaspr}). We define the annealed (resp. quenched) backward type distribution $h^l_T(x)$ (resp. $h^{l, \omega}_{0,T}(x)$) to be the $\mathbb{P}_l$-probability (resp. $\mathbb{P}^{\omega,T}_l$-probability) that all the $l$ lines from time $\beta = 0$ (resp. time $s=T$) receive type $0$ at the end of this procedure. For $0 < T_1 < T_2 $ we define $h^{l, \omega}_{T_1,T_2}(x)$ similarly as we defined $h^{l, \omega}_{0,T}(x)$. We similarly define $h^{l, \omega}_{T_1-,T_2}(x)$, $h^{l, \omega}_{T_1,T_2-}(x)$ and $h^{l, \omega}_{T_1-,T_2-}(x)$. 
\end{defi}
Note that, 
\begin{eqnarray}
\int h^{l, \omega}_{T_1,T_2}(x) P(d \omega) = h^l_{T_2-T_1}(x). \label{recoverannealed0}
%\int h^{l, \omega}_{0,T}(x) P(d \omega) = h^l_{T}(x). \label{recoverannealed0}
\end{eqnarray}
Heuristically, the procedure defining $h^l_T(x)$ and $h^{l, \omega}_{0,T}(x)$ can be interpreted as choosing randomly~$l$ individuals at instant $s=T$ in the infinite population, tracing back their potential ancestors until time $s=0$, assigning \textit{iid} types to potential ancestors from time $s=0$ (taking into account that $X(0)=x$), and propagating the types forward as in the Moran model. Thus, $h^l_T(x)$ (resp. $h^{l, \omega}_{0,T}(x)$) can informally be understood as the annealed (resp. quenched) probability that $l$ randomly chosen individuals in the infinite population at time $s=T$ are all of type $0$, given that $X(0)=x$. We can therefore expect that the rigorous relation between the ASG and the diffusion \eqref{levymodelsdesimp} should be $h^l_T(x) = \mathbb{E} [ (X(T))^l | X(0)=x ]$ in the annealed setting and $h^{l, \omega}_{0,T}(x)=\mathbb{E}^{\omega} [ (X(\omega,T))^l \mid X(\omega,0)=x ]$ in the quenched setting. This turns out to be true and is the content of Proposition \ref{h(x)ht(x)} from Section \ref{relwf-asgbis}. That proposition rigorously relates $h^l_T(x)$ and $h^{l, \omega}_{0,T}(x)$, which are defined via the ASG, to $X$. Most of the time, in this paper, we do not work with the jump-diffusion $X$ itself but with the ASG (or a slightly modified version of it) and study the quantity $h^l_T(x)$. In particular, we will use it in Section \ref{ccl} to prove a series representation for $h(x)$ (see Theorem \ref{finalformula}). 

In the case of one-sided selection studied in \cite{cordvech}, $h^l_T(x)$ (resp. $h^{l, \omega}_{0,T}(x)$) turns out to be the generating function of the line counting process of the annealed (resp. quenched) ASG. In that case, Proposition \ref{h(x)ht(x)} results into the classical moment duality \eqref{mdsm}. 

\subsection{The Enlarged ASG and a useful function} \label{enlargedasg}
%Subsection \ref{enlargedasg}

\subsubsection{Definition} \label{enlargedasgdef}
It will be convenient to work with a simple extension of the ASG, which we call \textit{Enlarged ASG} (E-ASG). 
\begin{defi}[The quenched/annealed E-ASG] \label{defquenchedeasg}
Let $\omega = (\omega(s))_{s \geq 0}$ be a fixed environment, $T > 0$ and $m \geq 1$. The quenched E-ASG on $[0,T]$ in environment $\omega$ starting with $m$ lines is the branching-coalescing particle system denoted by $(G^{\omega,T}_{s})_{s \in [0,T]}$ and defined as follows. It starts with $m$ ordered lines at time $s=T$ (i.e. $G^{\omega,T}_{T}$ contains $m$ ordered lines) and, between jumping times of $\omega$, has the following dynamic as $s$ decreases: 
\begin{itemize}
\item[(i)] Any pair of lines coalesces into a single line at rate $2$, independently from other pairs. 
\item[(ii)] Any line splits into two lines, an incoming line and a continuing line, at rate $\sigma$, independently from other lines, and such a branching is assigned the \textit{weight} $-1$. 
\end{itemize}
Additionally, if at a time $s\in[0,T]$ we have $\Delta \omega(s)\neq 0$, then $G^{\omega,T}_{s-}$ is obtained from $G^{\omega,T}_{s}$ as follows: 
\begin{itemize}
\item[(iii)] All lines of $G^{\omega,T}_{s}$ simultaneously split into two: an incoming line and a continuing line for each line of $G^{\omega,T}_{s}$. Each branching that is part of this simultaneous branching event is assigned the \textit{weight} $\Delta \omega(s)$. 
\end{itemize}
Let $m \geq 1$. The annealed E-ASG starting with $m$ lines is the branching-coalescing particle system denoted by $(G_{\beta})_{\beta \geq 0}$ and defined as follows. 
%We let $(T_i)_{i \geq 1}$ be the sequence of arrival times of a standard Poisson process on $[0,\infty)$ with parameter $\lambda$ and for convenience we let $T_0:=0$. 
It starts with $m$ ordered lines at time $\beta=0$ (i.e. $G_0$ contains $m$ ordered lines) and, as $\beta$ increases, it satisfies (i), (ii), and 
\begin{itemize}
%\item[(iii')] Each line of $G_{\beta-}$ splits into two lines, an incoming line and a continuing line. This simultaneous branching is assigned a random \textit{weight} that follows the distribution $\nu(\cdot)/\lambda$. 
\item[(iii')] At rate $\lambda$, all lines split simultaneously into two: an incoming line and a continuing line for each existing line. A common \textit{weight}, chosen according to the distribution $\nu(\cdot)/\lambda$, is assigned to each branching that is part of this simultaneous branching event. 
\end{itemize}
\end{defi}
In particular, in the E-ASG all lines split when there is a jump of the environment (as opposed to just a subset of lines in the ASG). This is independent of the jump size. However, the E-ASG keeps track of the jump size at each such branching. Except at time $s=T$/$\beta=0$, the order of lines in Definition \ref{defquenchedeasg} is irrelevant. The purpose of the ordering is that it will allow to define without ambiguity a function of the E-ASG in Section~\ref{enlargedasgencodingfct}. 

We still denote by $\mathbb{P}^{\omega,T}_m(\cdot)$ (resp. $\mathbb{P}_m(\cdot)$) the probability measure associated with $(G^{\omega,T}_{s})_{s \in [0,T]}$ (resp. $(G_{\beta})_{\beta \geq 0}$), and $\mathbb{E}^{\omega,T}_m[\cdot]$ (resp. $\mathbb{E}_m[\cdot]$) the associated expectation. 
%This is why, when defining the annealed E-ASG we had to couple it with a realization of the environment that lives on the same timeline, and give a definition analogous to the quenched setting. Note that, in Definition \ref{defquenchedasg}, we could have similarly defined the annealed ASG via a coupling with a realization $(L(\beta))_{\beta \geq 0}$ of the environment. 
%We can also define a graphical representation for the E-ASG by taking the same conventions as for the graphical representation of the ASG: two lines involved into a coalescence are joined by a triangle-shaped arrows, branchings that favor type $0$ (resp. type $1$) are represented by filled (resp. unfilled) star-shaped arrows. 
%$(L(\beta))_{\beta \in [0,T]}$ correspond to $(\omega(T) - \omega((T-\beta))-)_{\beta \in [0,T]}$ where $\omega$ is chosen according to the law $P$. 
%at time $0$ we start with $m$ lines ($m \geq 1$) and choose uniformly randomly an ordering of these $m$ lines that we denote by $L_1,...,L_m$. 
%jumps of $S$ occur at rate $\lambda$ 

\subsubsection{Line counting process} \label{enlargedasglcp}
The line counting process of the annealed E-ASG is a continuous-time Markov process with values on $\mathbb{N}=\{1,2,...\}$ and infinitesimal rates: 
\[ q(i,j):=\left\{\begin{array}{ll}
            i(i-1) &\text{if $j=i-1$},\\
            i\sigma &\text{if $j=i+1$},\\
            \lambda &\text{if $j=2i$}. 
            \end{array}\right. \]
We can see that it is a positive recurrent irreducible Markov chain (see for example \cite[Lem.~5.2 with $\theta=0$ and $\mu=\lambda \delta_1$]{cordvech}). In particular, it admits a stationary distribution that we denote by~$\pi$. The probabilities $\pi(k)$ satisfy a recursion formula and the right tail of $\pi$ can be controlled. These results are gathered in the following proposition. 

%VOIR COMMENT LA PROP SE MODIFIE DU FAIT QU'ON A SUPPRIME LE $/2$ DANS LE TAUX. ON A CORRIGE -> OK
\begin{prop} \label{recpilambda}
\begin{align}
\forall k \geq 2, \ \pi(k) & = \frac{\sigma}{k} \pi(k-1) + \frac{\lambda}{k(k-1)} \left ( \pi\left (\left \lfloor \frac{k+1}{2} \right \rfloor \right ) + ... + \pi(k-1) \right ). \label{recpik} 
\end{align}
There are two explicit positive constants $C_1 := C_1(\sigma,\lambda)$ and $C_2 := C_2(\sigma,\lambda)$ such that we have 
\begin{align}
\forall k \geq 1, \ \sum_{j \geq k} \pi(j) \leq C_1 e^{-C_2 (\log(k))^2}. \label{majoqueuepiknew}
\end{align}
%For any $k \geq 4$ we have 
%\begin{align}
%\pi(k) \leq \frac{\sigma(\sigma + \lambda)+\lambda}{k(k-1)} \exp \left ( (\lfloor \log_2(k) \rfloor -2) \left (\log(\sigma + \lambda) - \frac{\log(2)}{2} (\lfloor \log_2(k) \rfloor - 3) \right ) \right ). \label{majopik} 
%\end{align}
%For any $k \geq 4 \vee 2^{7/2} (\sigma + \lambda)$ we have 
%\begin{align}
%\sum_{j \geq k} \pi(j) & \leq \frac{\sigma(\sigma + \lambda)+\lambda}{k-1} \exp \left ( (\lfloor \log_2(k) \rfloor -2) \left (\log(\sigma + \lambda) - \frac{\log(2)}{2} (\lfloor \log_2(k) \rfloor - 3) \right ) \right ). \label{majoqueuepik}
%\end{align}
\end{prop}
%ATTENTION, POUR \eqref{majoqueuepik}, IL FAUT QUE L'EXPONENTIELLE SOIT DECROISSANTE EN $k$ OR C'EST LE CAS A PARTIR D'UN INDICE QUI DEPEND DE $\lambda$ MAIS PAR FORCEMENT A PARTIR DE $4$. OK

The proof of Proposition \ref{recpilambda} is quite computational and thus we shift it to Appendix~\ref{A2}. The ratios $\pi(j)/\pi(1)$ can be recursively computed via \eqref{recpik} for all $j \geq 2$. Then, using \eqref{majoqueuepiknew} for some large $k$, together with $\sum_{j \geq 1} \pi(j) = 1$, one can subsequently deduce a good approximation for~$\pi(1)$, and then for $\pi(j)$ and $\sum_{k \geq j} \pi(k)$. 

In several of our proofs, it is crucial that the tail distribution of the line counting process of the E-ASG decays faster than polynomially. Thus, the bound \eqref{majoqueuepiknew} is key.
Moreover \eqref{majoqueuepiknew} will allow to control approximations of $h(x)$ (see \eqref{boundfiniteapprox} in Theorem \ref{finalformula}). 

\subsubsection{Relation with ASG and backward type distribution} \label{enlargedasgrelasg}
We now relate the E-ASG to $h^l_T(x)$ and $h^{l, \omega}_{0,T}(x)$. To this end, we define a type assignment procedure for the E-ASG: 
\begin{defi}[Type assignment procedure for the E-ASG] \label{typeaspreasg}
For $T >0$ and $x \in [0,1]$, the type assignment procedure with initial condition $x$ for the annealed (resp. quenched) E-ASG on $[0,T]$ is defined as follows. For each branching we draw a Bernoulli random variable with parameter the absolute value of its weight. The branching is labeled \textit{real} or \textit{virtual} depending on whether the random variable equals $1$ or $0$. The Bernoulli random variables associated to the different branchings are independent. Note that at an event corresponding to (iii) or (iii') in Definition \ref{defquenchedeasg}, there are several branchings that occur simultaneously and that have a common weight, so we emphasis that independent Bernoulli random variables are assigned to them. Then, types are assigned according to the rules (i),(ii),(iii) of Definition \ref{typeaspr} but, instead of the rule (iv) of that definition, we have 
%contains two kinds of lines: \textit{real} lines and \textit{virtual} lines. Among these lines, 
%DIRE DANS QUEL ORDRE ON LE FAIT (LES JEUNES D'ABORD OU LES ANCETRES ?) ON LE FAIT DES LA CREATION -> D'ABORD LES JEUNES
%If the new line is declared virtual, then so are all its ancestors. 
%ON SUPPRIME LES LABELS DES LIGNES, ON NE MET DE LABELS QU'AUX BRANCHINGS ET ON DIT COMMENT CA INFLUENCE LA PROPAGATION DES TYPES, CA PERMET DE RESOUDRE LE PROBLEME DES LIGNES DE A QUI DOIVENT ETRE FIXES "RELLES" QUAND ON VEUT CONNAITRE LES TYPES SUR A. EN PLUS, LE TRUC D'AVANT POSAIT UN PROBLEMME POUR ASSIGNER SON TYPE A UNE LIGNE VIRTUELLE QUI SE DIVISE .
%We now assign labels to the lines in the following way: The $m$ starting lines at time $0$ are all real. If there is coalescence between two real line, the resulting line is real. If there is coalescence between two virtual line, the resulting line is virtual. If there is coalescence between a real line and a virtual line, the resulting line is real. If a virtual line branches, the two resulting lines are virtual (regardless to the label of the branching). If a real line branches, the continuing lines is real and the incoming line  is real (resp. virtual) if the branching is real (resp. virtual). 
\begin{itemize}
\item In a branching labeled real, with weight of sign $(-1)^i$, $i \in \{0,1\}$, we have: If the incoming line is of type $i$, then the line that branches receives type $i$. If the incoming line is of type $1-i$, then the line that branches receives the type of the continuing line. 
\item In a branching labeled virtual the line that branches receives the type of the continuing line. 
\end{itemize}
\end{defi}

The labels of the E-ASG are defined so that for $m \geq l \geq 1$, the annealed (resp. quenched) E-ASG starting with $m$ lines contains the annealed (resp. quenched) ASG starting with $l$ lines. Indeed, in the annealed (resp. quenched) E-ASG starting with $m$ lines, let us color in grey the last $m-l$ lines from time $\beta=0$ (resp. time $s=T$), the incoming lines that arise from branchings labeled virtual, and all lines that arise from branchings of grey lines. If two grey (resp. non-grey) lines coalesce we set the resulting line to be grey (resp. non-grey). If a grey line coalesces with a non-grey line, we set the resulting line to be non-grey. Then the system of non-grey lines is a realization of the annealed (resp. quenched) ASG starting with $l$ lines. Moreover, the rules from Definition \ref{typeaspreasg} ensure that, for the realization of the ASG starting with $l$ lines that is contained in the E-ASG starting with $m$ lines, the type assignment procedure is the same as the one given by Definition \ref{typeaspr}, and the types are not influenced by the grey lines. Consequently, we obtain the following: 
\begin{lemme} \label{backwardtypedistribcoincide}
If we consider the annealed (resp. quenched) E-ASG on $[0,T]$, starting with $m$ lines at time $\beta=0$ (resp. $s=T$), and apply the type assignment procedure on $[0,T]$ with initial condition $x$ from Definition \ref{typeaspreasg}, then $h^l_T(x)$ (resp. $h^{l, \omega}_{0,T}(x)$) is the probability that the first $l$ lines from time $\beta=0$ (resp. $s=T$) all receive type $0$. 
\end{lemme}

The E-ASG has several advantages over the ASG. It will become apparent in Section~\ref{mainresults} that many of our expressions decompose according to the number of lines in the ancestral graph that we use (see for example \eqref{decompattendueent}). In the proofs we will require bounds on the tail distribution of this number of lines and, in the final results, these bounds also allow to measure the quality of the approximation of $h(x)$ by finitely many terms from its series representation (see \eqref{boundfiniteapprox} in Theorem \ref{finalformula}). The simplicity of the structure and of the line counting process of the E-ASG makes it possible to derive useful explicit formulas and bounds in Proposition~\ref{recpilambda}. Moreover, a disadvantage of the ASG is that, at simultaneous branchings, both the number of branchings and the combinatorics need to be taken into account. In this sense, working with the E-ASG and keeping track of the weights of branchings is simpler. This significantly reduces the number of cases that have to be considered in the analysis of the small-time behavior of the ancestral graph in Section~\ref{behaviour01}, and the number of terms in the ODE system coding for the combinatorics of the graph (see Theorem \ref{equadiffcoeffnew}). The information contained in the weights of branchings is not too bothersome and it will lead to coefficients $\tau(i,j)$ from Definition \ref{defcoefedo}. We suspect that our methodology extends to the case of an infinite measure $\nu$, at the cost of more complexity because one then probably needs to work directly with the ASG.
%allows to have independence between 1) the values of the jumps (sizes and signs), and 2) the graph structure of the E-ASG. In particular the \textit{iid} sequence $(W_n)_{n \geq 1}$ of the values of the jumps is independent from the sequence $(T_n)_{n \geq 1}$ of the arrival times of transitions of the E-ASG. 
%%, and the randomness from the coalescences. 
%The function that we define below, that encodes for the relevant information of the ASG, is defined from the set of subsets of lines of the E-ASG (that is related to the graph structure of the E-ASG) to the set of real numbers, and the values of the function are related to the values of the jumps. This is why it is interesting to separate the two randomness, especially when we deal with expectations involving this function. 

Consistency is another useful property of the E-ASG. More precisely, being able to run the process with $m$ lines for some $m \geq l$, while we are only interested in the types of $l$ lines in order to determine $h^l_T(x)$ or $h^{l, \omega}_{0,T}(x)$, makes it possible to formulate a crucial branching property (see Lemma \ref{markpropf} of Section \ref{behaviour02}) and to define some coefficients for which we can establish a semigroup property (see Proposition \ref{semigroupprop} of Section \ref{mainresults}). 

%Let us show how the E-ASG allows to express $h(x)$. We have 
%%$h(x) = \lim_{t \rightarrow \infty} \mathbb{E}[X_t | X_{0} = x]$. 
%$h(x) = \lim_{T \rightarrow \infty} \mathbb{E}[X_0 | X_{-T} = x]$. 
%%\textbf{(Note : Il faudra utiliser une autre d\'efinition quand on ajoutera les mutations)}. 
%Let us put 
%%$h_t(x) = \mathbb{E}[X_t | X_{0} = x]$. 
%$h_T(x) = \mathbb{E}[X_0 | X_{-T} = x]$. 
%In order to determine $h_T(x)$, we start with one individual at time $0$ and run the E-ASG until time $-T$. Then apply the procedure to assign the labels (real or virtual) to the lines and we assign \textit{iid} type to the real lines at type $-T$ (each line has probability $x$ to be of type $0$ and probability $1-x$ to be of type $1$) and we propagate types forward in time following the rules described above. $h_T(x)$ is then the probability that the line at time $0$ is of type $0$. 
%%If the $E-ASG$ starts with $n$ lines at instant $0$, let us denote by $H_{T}(n,x)$ the probability that the $n$ lines at time $0$ are all of type $1$. In particular we have $1-h_T(x) = H_{T}(1,x)$ and for any $n \geq 1, H_{0}(1,x) = (1-x)^n$. 
%%If we do this procedure starting at time $0$ with $N$ lines instead of $1$, then let us denote by 

\subsubsection{Some definitions and notations} \label{enlargedasgdefnot}
We use the terminology of trees for the E-ASG: For $n \geq 1$ (resp. $n=0$), we define generation $n$ of the E-ASG as the set of pieces of lines present in the E-ASG just after its $n^{th}$ transition (resp. starting time). When looking at a realization of the E-ASG as such a pseudo-tree, we view pieces of lines from a given generation as vertices and forget about their lengths. If a line is unaffected by the $n^{th}$ transition, we say that the part of this line lying in generation $n-1$ (resp. $n$) is \textit{parent} (resp. \textit{son}) to the part of this line lying in generation $n$ (resp. $n-1$). From now, we abusively refer to pieces of lines from a given generation as \textit{lines}. We say that a branching line is \textit{parent} to the incoming line and the continuing line and that the laters are its \textit{sons}. Similarly, we say that two coalescing lines are \textit{parents} to the resulting line, and that the later is their \textit{son}. For $m \geq 1$, let $\mathbb{G}_m$ denote the family of finite graphs containing possible realizations of the E-ASG on finite time-intervals, starting with $m$ ordered lines. A graph $G\in \mathbb{G}_m$ has $m$ ordered lines in generation $0$, finitely many generations, and only three possible types of generations that we call \textit{multiple branching generations}, \textit{single branching generations}, and \textit{coalescing generations}. 

For $n \geq 0$, the $n^{th}$ generation of $G\in \mathbb{G}_m$ is assigned a \textit{weight} $S_n\in[-1,1)$. This weight is always $0$ (resp. $ -1$) if the $n^{th}$ generation is a coalescing (resp. a single branching) generation. In each branching generation (multiple or single), we distinguish two kinds of lines: \textit{incoming lines} and \textit{continuing lines}. Each pair of sons of an line from the previous generation contains an incoming line and a continuing line. When a line from the previous generation has a single son, the son is a continuing line. 

Clearly, for any $t \geq 0$, the random finite graph generated by the annealed E-ASG $(G_{\beta})_{\beta \in [0,t]}$ starting with $m$ ordered lines at time $\beta=0$ is an element of $\mathbb{G}_m$ that we identify with $G_t$ (formally $G_t$ is a more complicated object containing information about lengths of lines but we will not need it, so we can make this identification). Similarly, for any fixed environment $\omega$, $T > 0$ and $r \in [0,T]$, the random finite graph generated by the quenched E-ASG $(G^{\omega,T}_{s})_{s \in [r,T]}$ starting with $m$ ordered lines at time $s=T$ is an element of $\mathbb{G}_m$ that we identify with $G^{\omega,T}_{r}$. In both cases, the weight of a multiple branching generation is set to be the common weight assigned to the corresponding simultaneously occurring branchings (in Definition \ref{defquenchedeasg}). 

For $G \in \mathbb{G}_m$, let $\mathsf{depth}(G)$ denote the number of generations of $G$, not counting generation $0$. For any integer $n\geq0$, $\mathbb{G}_m^n$ denotes the set of elements of $\mathbb{G}_m$ with depth $n$. For $n \leq \mathsf{depth}(G)$, let $\pi_n(G)$ denote the projection of $G \in \mathbb{G}_m$ on $\mathbb{G}_m^n$, that is, the graph obtained by restricting $G$ to its first $n+1$ generations (including generation $0$). 
%The idea is to express $h_t(x)$ as a function of $G_t$. 
For each $G \in \mathbb{G}_m$, let $V_G$ denote the set of lines in the last generation of $G$ 
%(i.e. those with depth exactly $\mathsf{depth}(G)$) 
and let $\mathcal{P}(V_G)$ denote the family of non-empty subsets of $V_G$. Note that, if $\mathsf{depth}(G) \geq 1$, $V_{\pi_{\mathsf{depth}(G)-1}(G)}$ is the set of lines in the generation before the last generation of $G$. For $A \in \mathcal{P}(V_G)$, let $P(A) \in \mathcal{P}(V_{\pi_{\mathsf{depth}(G)-1}(G)})$ be the set of parents of elements of $A$. For $A \in \mathcal{P}(V_{\pi_{\mathsf{depth}(G)-1}(G)})$, let $D(A) \in \mathcal{P}(V_G)$ be the set of sons of elements of $A$ in $V_G$. For a set of lines $A \in \mathcal{P}(V_G)$, $|A|$ denotes its cardinality. 

\begin{defi}[Type $0$ events] \label{type0events}
Let $T>0$, $t \in [0,T]$, $x \in [0,1]$, $(G_{\beta})_{\beta \in [0,T]}$ be the annealed E-ASG on $[0,T]$, and $A \in \mathcal{P}(V_{G_{t}})$. We apply the type assignment procedure from Definition \ref{typeaspreasg} on $[0,T]$ with initial condition $x$. We define $E(t,T,A,x)$ to be the event where all the lines belonging to $A$ receive type $0$ after this procedure. 
\end{defi} 
Note from Definition \ref{type0events} and Lemma \ref{backwardtypedistribcoincide} that for $m \geq l \geq 1$, $h^l_T(x)=\mathbb{P}_m(E(0,T,\{ L_1,..., L_l \},x))$, if $L_1,..., L_l$ denote the first $l$ lines in the annealed E-ASG at time $\beta=0$. The events from Definition \ref{type0events} will be useful to study the expression of $h^l_T(x)$ (see the proof of Theorem \ref{propagationformule}). 
%Let us denote by $H(t,T,A,x) := \mathbb{P}_m(E(t,T,A,x) | \mathcal{F}_t)$. $H(t,T,A,x)$ is a function of the environment and of the E-ASG on $[0,t]$. 
%Note that when we work conditionally on $\mathcal{F}_t$, and run the E-ASG between instants $-t$ and $-T$ with the initial set of lines at $-t$ being the lines in $V_{G_{t}}$ then, after assigning labels to the branchings on $[-T,0]$, and assigning the types, the probability that all the lines belonging to a set $A \subset V_{G_{t}}$ are of type $0$ is $H(t,T,A,x)$. PAS NECESSAIRE DE PRECISER. 

Let $(\mathcal{F}_t)_{t \geq 0}$ be the filtration generated by the annealed E-ASG, i.e. $\mathcal{F}_t := \sigma((G_{\beta})_{\beta \in [0,t]})$ for any $t \geq 0$. Note that $\mathcal{F}_t$ contains the information of the weights assigned at transitions of the E-ASG on~$[0,t]$, but no information related to the type assignment procedure from Definition \ref{typeaspreasg} (i.e. no information on labels 
%since these labels are assigned after the E-ASG has been run on the whole interval $[-T,0]$. 
%$(\mathcal{F}_t)_{t \geq 0}$ neither contains information 
or types). In the quenched setting, for any $0 \leq r \leq t$, let $\mathcal{F}^{\omega,t}_{r} := \sigma((G^{\omega,t}_{s})_{s \in [r,t]})$. We also define $\mathcal{F}^{\omega,t}_{r-} := \cap_{s<r}\mathcal{F}_s^{\omega,t}$. 
%$\mathcal{F}^{\omega,t-}_{r}$, and $\mathcal{F}^{\omega,t-}_{r-}$. 
%If $\Delta \omega (r) \neq 0$, $\mathcal{F}^{\omega,t}_{r-}$ and $\mathcal{F}^{\omega,t-}_{r-}$ take into account the branchings at $s=r$ while $\mathcal{F}^{\omega,t}_{r}$ and $\mathcal{F}^{\omega,t-}_{r}$ do not. If $\Delta \omega (t) \neq 0$, $\mathcal{F}^{\omega,t}_{r}$ and $\mathcal{F}^{\omega,t}_{r-}$ take into account the branchings at $s=t$ while $\mathcal{F}^{\omega,t-}_{r}$ and $\mathcal{F}^{\omega,t-}_{r-}$ do not. 

For $\beta \geq 0$, $(|V_{G_{\beta}}|)_{\beta \geq 0}$ is just the line counting process of the E-ASG (recall $V_{G_{\beta}}$ is the set of lines of the E-ASG at instant $\beta$). In particular, its stationary distribution is $\pi$ (defined in Section~\ref{enlargedasglcp}). This implies the following estimate, which is useful for bounds that are uniform in~$\beta$. 
\begin{lemme} \label{lemmemajointemporelle}
For any $m, k \geq 1$ and $\beta \geq 0$ we have 
\begin{eqnarray}
\mathbb{P}_m \left ( |V_{G_{\beta}}| = k \right ) \leq \frac{\pi(k)}{\pi(m)}. \label{majointemporelle}
\end{eqnarray}
\end{lemme}

%\subsection{A function of the enlarged ASG} \label{fctenlargedasg}
%(resp. $h^{l, \omega}_{0,T}(x)$), (resp. $G^{\omega,T}_{0}$)
\subsubsection{Encoding function} \label{enlargedasgencodingfct}
The idea is to express $h^l_T(x)$ as a function of $G_T$. To this end, we introduce an object that is one of the main tools for our analysis. More precisely, to each $G \in \mathbb{G}_m$ we deterministically associate a function $F^l_G : \mathcal{P}(V_G) \rightarrow \mathbb{R}$ that contains the relevant information about the combinatorics of $G$. Let $m \geq 1$ and $1 \leq l \leq m$. Given a graph $G \in \mathbb{G}_m$, let $\{ L_1,..., L_m \}$ denote the ordered lines in generation $0$ of $G$. Then, $F^l_G$ is defined recursively on $\mathsf{depth}(G)$. If $G \in \mathbb{G}_m^0$ (i.e. $\mathsf{depth}(G)=0$), then $V_G = \{ L_1,..., L_m \}$ and $F^l_{G}(A) := \mathds{1}_{A=\{ L_1,..., L_l \}}$ for $A \in \mathcal{P}(V_G)$. If, for some $n\geq 0$, $F^l_G$ is defined for all $G \in \mathbb{G}_m^n$, then for $G \in \mathbb{G}_m^{n+1}$, we define $F^l_G$ as follows. 
\begin{itemize}
\item If $n+1$ is a multiple branching generation of $G$, then for any $A \in \mathcal{P}(V_G)$ we set 
%\begin{align}
%F^l_G(A) := F_{\pi_n(G)}(P(A)) \times (1-S_{n+1} \mathds{1}_{S_{n+1} > 0})^{\alpha(A)} \times (-S_{n+1} \mathds{1}_{S_{n+1} < 0})^{\beta(A)} \times S_{n+1}^{\gamma(A)}, \label{defFjump}
%\end{align}
\begin{align}
F^l_G(A) := F^l_{\pi_n(G)}(P(A)) \times (1+S_{n+1} \mathds{1}_{S_{n+1} < 0})^{\alpha(A)} \times (S_{n+1} \mathds{1}_{S_{n+1} > 0})^{\beta(A)} \times (-S_{n+1})^{\gamma(A)}, \label{defFjump}
\end{align}
where $\alpha(A)$ is the number of continuing lines in $A$ whose incoming brothers are in $A^c$, $\beta(A)$ is the number of incoming lines in $A$ whose continuing brothers are in $A^c$, and $\gamma(A)$ is the number of pairs of brothers that are both in $A$. In \eqref{defFjump} we use the convention $0^0=1$. 
\item If $n+1$ is a single branching generation of $G$, then for any $A \in \mathcal{P}(V_G)$ we set 
%\begin{align}
%F^l_{G}(A) := F^l_{\pi_n(G)}(P(A)) \times \left ( \mathds{1}_{N(A)\in \{0,1\}} - \mathds{1}_{N(A)=2} \right ), \label{defFcontsel}
%\end{align}
\begin{align}
F^l_{G}(A) := F^l_{\pi_n(G)}(P(A)) \times \mathds{1}_{N(A)\in \{0,2\}}, \label{defFcontsel}
\end{align}
where $N(A) \in \{0,1,2\}$ is the number of sons of the branching line in $A$. 
\item If $n+1$ is a coalescing generation of $G$, then for any $A \in \mathcal{P}(V_G)$ we set 
\begin{align}
F^l_G(A) := \sum_{B \in \mathcal{P}(V_{\pi_n(G)}) ; \, D(B) = A} F^l_{\pi_n(G)}(B). \label{defFcoal}
\end{align}
\end{itemize}

Note that under $\mathbb{P}_m$, for any $t \geq 0$, $F^l_{G_t}$ is a deterministic function of the random object $(G_{\beta})_{\beta \in [0,t]}$ (which contains information on the weights of transitions of $(G_{\beta})_{\beta \in [0,t]}$) and does not depend on labels and types that are assigned in the type assignment procedure from Definition~\ref{typeaspreasg}. 

The updating rules of $F^l_{G}(\cdot)$ are chosen such that they allow deriving expression \eqref{propagationformule1mix} for the probability $h^l_T(x)$ as an expectation involving the coefficients $F^l_{G_{T}}(A)$ (see Theorem~\ref{propagationformule} of Section \ref{mainresults}). The idea is that~\eqref{propagationformule1mix} is obtained by updating the expression of $h^l_T(x)$ along events of the E-ASG. It will become apparent in the proof of Theorem \ref{propagationformule} that \eqref{defFjump}--\eqref{defFcoal} form the only possible definition so that $F^l_{G}(\cdot)$ is updated accordingly at each event of the E-ASG, in order for \eqref{propagationformule1mix} to hold true. 

Finally, let us mention that an extension of the function $F^l_{G}(\cdot)$ will be defined in Section \ref{behaviour021}, in order to formulate and establish a renewal property for $(F^{i}_{G_{\beta}}(\cdot))_{\beta \geq 0}$ in Section \ref{behaviour022}.

\subsection{Main results} \label{mainresults}

Our approach is based on deriving an exact expression of the probability $h^l_T(x)$ in terms of the E-ASG. The next theorem is a key point in our analysis. It relates $h^l_T(x)$ to $F^l_G(\cdot)$ introduced in the previous subsection. 
%SOIT DIRE QUE CETTE PROBA S'EXPRIME COMME LA SOMME DES CONFIGURATIONS COMPATIBLES ET QU'ON PREFERE L'EXPRIMER COMME UNE SOMME SUR LES ENSEMBLES DE LIGNES. SOIT DIRE QU'ON ITERE L'EXPRESSION A CHAQUE EVENEMENT. 

\begin{theo} \label{propagationformule}
For any $l \geq 1, m \geq l, x \in [0,1]$ and $T \geq 0$ we have the following two expressions for $h^l_T(x)$. 
%\begin{align}
%h^l_T(x) &= \mathbb{E}_m \left [ \sum_{A \in \mathcal{P}(V_{G_{T}})} F^l_{G_{T}}(A) \mathds{1}_{E({T},T,A,x)} \right ], \label{egsuren00} \\
%h^l_T(x) &= \mathbb{E}_m \left [ \sum_{A \in \mathcal{P}(V_{G_T})} F^l_{G_T}(A) x^{|A|} \right ]. \label{propagationformule1}
%\end{align}
\begin{align}
h^l_T(x) = \mathbb{E}_m \left [ \sum_{A \in \mathcal{P}(V_{G_{T}})} F^l_{G_{T}}(A) \mathds{1}_{E({T},T,A,x)} \right ] = \mathbb{E}_m \left [ \sum_{A \in \mathcal{P}(V_{G_T})} F^l_{G_T}(A) x^{|A|} \right ], \label{propagationformule1mix}
\end{align}
where $E({T},T,A,x)$ is as in Definition \ref{type0events}. Moreover, we have $\mathbb{P}_m$-almost surely 
%\begin{align}
%& \sum_{A \in \mathcal{P}(V_{G_{T}})} F^l_{G_{T}}(A) \mathds{1}_{E(T,T,A,x)} \in [0,1], \label{entre0et10} \\
%& \sum_{A \in \mathcal{P}(V_{G_{T}})} F^l_{G_{T}}(A) x^{|A|} \in [0,1], \label{entre0et10inttype}
%\end{align}
\begin{align}
\sum_{A \in \mathcal{P}(V_{G_{T}})} F^l_{G_{T}}(A) \mathds{1}_{E(T,T,A,x)} \in [0,1], \ \text{and} \ \sum_{A \in \mathcal{P}(V_{G_{T}})} F^l_{G_{T}}(A) x^{|A|} \in [0,1], \label{entre0et10mix}
\end{align}
and 
\begin{eqnarray}
\sum_{A \in \mathcal{P}(V_{G_{T}})} F^l_{G_{T}}(A) = 1. \label{sommetotale=1}
\end{eqnarray}
\end{theo}
The combination of \eqref{entre0et10mix} with a combinatorial argument plays a crucial role in controlling the absolute values of the coefficients $F^l_{G_{T}}(A)$ and sums thereof (see Lemmas \ref{lemmecomb} and \ref{sgwelldef} for details).

\begin{remark} \label{propagationformulequenched0}
%The proof of Theorem \ref{propagationformule} does not use anywhere the fact that the environment $(L(\beta))_{\beta \in [0,T]}$ is a L\'evy process. Actually the same proof works for $(L(\beta))_{\beta \in [0,T]}$ replaced by any deterministic, piecewise constant, c\`ad-l\`ag function with finitely many jumps. 
The proof of Theorem \ref{propagationformule} also works in the quenched setting. 
%, c\`ad-l\`ag piecewise linear function (with slope $-\sigma$) that has finitely many jumps (all in $(-1,1)$) on $[0,T]$. 
We thus obtain for any fixed environment $\omega$, $l \geq 1, m \geq l, x \in [0,1]$, $r\geq 0$ and $t>r$, 
\begin{eqnarray}
h^{l, \omega}_{r,t}(x) = \mathbb{E}^{\omega,t}_m \left [ \sum_{A \in \mathcal{P}(V_{G^{\omega,t}_r})} F^l_{G^{\omega,t}_r}(A) x^{|A|} \right ]. \label{propagationformulequenched}
\end{eqnarray}
\end{remark}

Let us fix $i \geq 1$ and $m \geq i$. From \eqref{propagationformule1mix} we have $h^i_T(x) = \mathbb{E}_m[ \sum_{j \geq 1} \sum_{A \in \mathcal{P}(V_{G_T}) ;\, |A| = j} F^i_{G_T}(A) x^{j}]$. One could be tempted to rewrite this as $\sum_{j \geq 1} \mathbb{E}_m[ \sum_{A \in \mathcal{P}(V_{G_T}) ;\, |A| = j} F^i_{G_T}(A)] x^{j}$ and 
%\begin{align}
%h^i_T(x) = \mathbb{E}_m \left [ \sum_{j \geq 1} \sum_{A \in \mathcal{P}(V_{G_T}) ;\, |A| = j} F^i_{G_T}(A) x^{j} \right ] "=" \sum_{j \geq 1} \ \mathbb{E}_m \left [ \sum_{A \in \mathcal{P}(V_{G_T}) ;\, |A| = j} F^i_{G_T}(A) \right ] x^{j}. \label{usefulfordl}
%\end{align}
thus expect that $h(x) = \sum_{j \geq 1} b_j x^j$ where $b_j := \lim_{T \rightarrow \infty} \mathbb{E}_m [ \sum_{A \in \mathcal{P}(V_{G_T}) ; |A| = j} F^i_{G_T}(A) ]$. 
The problem with this heuristic reasoning is that the inversion of the sum and expectation, and then passage to the limit, seem to be not valid in general (but they are in some particular cases, see Sections \ref{casenoenv} and \ref{exapponesidedenv}). Moreover, even though we show below that the coefficients $b_j$ are well-defined, it seems that 
%taking the term-by-term limit as $T$ goes to infinity into the series is not valid either and that 
the power series $\sum_{j \geq 1} b_j y^j$ may have null radius of convergence. Therefore, we decompose $h^i_T(x)$ in a more suitable basis of polynomials. We first define some useful coefficients: 
\begin{defi} [Duality coefficients] \label{defcoefsg}
For any $i,j\geq 1$ and $t \geq 0$, we define 
\begin{eqnarray}
Q_t(i,j) := \mathbb{E}_i \left [ \sum_{A \in \mathcal{P}(V_{G_{t}}) ; |A|=j} F^i_{G_{t}}(A) \right ]. \label{coeff1indice}
\end{eqnarray}
For any $m,k\geq 1, i \in \{ 1,...,m \}, j \in \{ 1,...,k \}$, $t \geq 0$, we define 
\begin{eqnarray}
R^{m,k}_t(i,j) := \mathbb{E}_m \left [ \mathds{1}_{|V_{G_{t}}| = k} \sum_{A \in \mathcal{P}(V_{G_{t}}) ; |A|=j} F^i_{G_{t}}(A) \right ]. \label{defcoeff}
\end{eqnarray}
For any fixed environment $\omega$, $m,k\geq 1, i \in \{ 1,...,m \}, j \in \{ 1,...,k \}$, $r\geq 0$ and $t>r$, we define 
\begin{eqnarray}
R^{m,k, \omega}_{r,t}(i,j) := \mathbb{E}^{\omega,t}_m \left [ \mathds{1}_{|V_{G^{\omega,t}_{r}}| = k} \sum_{A \in \mathcal{P}(V_{G^{\omega,t}_r}) ; |A|=j} F^i_{G^{\omega,t}_r}(A) \right ]. \label{defcoeffquenched}
\end{eqnarray}
We similarly define $R^{m,k, \omega}_{r-,t}(i,j)$, $R^{m,k, \omega}_{r,t-}(i,j)$, and $R^{m,k, \omega}_{r-,t-}(i,j)$. 
\end{defi} 
The expectations in Definition \ref{defcoefsg} are well-defined, we prove this in Lemma \ref{sgwelldef} of Section~\ref{dawd}. Let us show how the quantities $R^{m,k}_t(i,j)$ are related to $h^i_t(x)$. Using \eqref{propagationformule1mix} and that, by \eqref{entre0et10mix}, $\sum_{A \in \mathcal{P}(V_{G_t})} F^i_{G_t}(A) x^{|A|} \in [0,1]$ $\mathbb{P}_m$-almost surely, we get for all $i \geq 1$, $m \geq i$, and $t \geq 0$, 
%\[ h^i_t(x) = \mathbb{E}_m \left [ \left ( \sum_{k=1}^{\infty} \mathds{1}_{|V_{G_{t}}| = k} \right ) \sum_{A \in \mathcal{P}(V_{G_t})} F^i_{G_t}(A) x^{|A|} \right ] \]
%\begin{align}
%h^i_t(x) & = \sum_{k=1}^{\infty} \mathbb{E}_m \left [ \mathds{1}_{|V_{G_{t}}| = k} \sum_{A \in \mathcal{P}(V_{G_t})} F^i_{G_t}(A) x^{|A|} \right ] = \sum_{k=1}^{\infty} \sum_{j=1}^k \mathbb{E}_m \left [ \mathds{1}_{|V_{G_{t}}| = k} \sum_{A \in \mathcal{P}(V_{G_t}) ; |A| = j} F^i_{G_t}(A) \right ] x^{j} \nonumber \\
%& = \sum_{k=1}^{\infty} \sum_{j=1}^k R^{m,k}_t(i,j) x^{j}. \label{decompattendueent}
%\end{align}
\begin{align}
h^i_t(x) & = \sum_{k=1}^{\infty} \mathbb{E}_m \left [ \mathds{1}_{|V_{G_{t}}| = k} \sum_{A \in \mathcal{P}(V_{G_t})} F^i_{G_t}(A) x^{|A|} \right ] = \sum_{k=1}^{\infty} \sum_{j=1}^k R^{m,k}_t(i,j) x^{j}. \label{decompattendueent}
\end{align}
Similarly, for any fixed environment $\omega$, $i \geq 1, m \geq i$, and $0 \leq r < t$, 
\begin{align}
h^{i, \omega}_{r,t}(x) = \sum_{k=1}^{\infty} \sum_{j=1}^k R^{m,k, \omega}_{r,t}(i,j) x^{j}. \label{decompattendueentquenched}
\end{align}
Combining Proposition \ref{h(x)ht(x)} from Section \ref{relwf-asgbis}, Theorem \ref{propagationformule}, Remark \ref{propagationformulequenched0}, equations \eqref{decompattendueent} and \eqref{decompattendueentquenched}, we obtain annealed and quenched expressions for the moments of the Wright-Fisher diffusion \eqref{levymodelsdesimp}: 
\begin{theo} \label{momdiff} 
For any $l \geq 1, m \geq l, x \in [0,1]$ and $T \geq 0$ we have
%\begin{align}
\[ \mathbb{E} \left [ (X(T))^l | X(0)=x \right ] = \mathbb{E}_m \left [ \sum_{A \in \mathcal{P}(V_{G_T})} F^l_{G_T}(A) x^{|A|} \right ] = \sum_{k=1}^{\infty} \sum_{j=1}^k R^{m,k}_T(l,j) x^{j}. \]
%\label{momentsf}
%\end{align}
For any fixed environment $\omega$, $l \geq 1, m \geq l, x \in [0,1]$ and $T \geq 0$ we have
%, if $\omega$ has a jump at time $t < T$, then we have 
%\begin{align}
\[ \mathbb{E}^{\omega} \left [ (X(\omega,T))^l \mid X(\omega,0)=x \right ] = \mathbb{E}^{\omega,T}_m \left [ \sum_{A \in \mathcal{P}(V_{G^{\omega,T}_0})} F^l_{G^{\omega,T}_0}(A) x^{|A|} \right ] = \sum_{k=1}^{\infty} \sum_{j=1}^k R^{m,k, \omega}_{0,T}(l,j) x^{j}. \]
%\label{quenchedmomentsf}
%\end{align}
\end{theo}

For deriving expressions for $h(x)$, it will be necessary to understand the behavior of the coefficients $R^{m,k}_t(i,j)$ as $t$ goes to infinity. What makes this difficult is that $F^i_{G_t}(\cdot)$ is not explicit (in fact, it is quite complicated). However, 
%as will be explained in Subsection \ref{hiitoabrw}, 
we can establish the existence of a renewal structure for the E-ASG and $(F^i_{G_{\beta}}(\cdot))_{\beta \geq 0}$ in Section \ref{behaviour02}, which allows to show that $(R^{\cdot,\cdot}_t(\cdot,\cdot))_{t \geq 0}$ forms a semigroup. 
%defined by a modification of a stochastic semigroup 

\begin{prop} [semigroup property] \label{semigroupprop}

For any $m,k\geq 1, i \in \{ 1,...,m \}, j \in \{ 1,...,k \}$, \\ $t \geq 0, r \geq 0$, the series $\sum_{\tilde k = 1}^{\infty} \sum_{\tilde j = 1}^{\tilde k} R^{m,\tilde k}_{t}(i, \tilde j) R^{\tilde k,k}_{r}(\tilde j,j)$ converges and we have 
\begin{eqnarray}
R^{m,k}_{t+r}(i,j) = \sum_{\tilde k = 1}^{\infty} \sum_{\tilde j = 1}^{\tilde k} R^{m,\tilde k}_{t}(i, \tilde j) R^{\tilde k,k}_{r}(\tilde j,j). \label{strongsgprop}
\end{eqnarray}
\end{prop}
%In order to determine $h^i_t(x)$, since we are only interested in the types of $i$ lines from time $\beta=0$, one may wonder why we do not just start the E-ASG with $i$ lines, instead of $m$ lines with $m \geq i$ (i.e. why we do not work under $\mathbb{P}_i$, instead of working under $\mathbb{P}_m$ for some $m \geq i$). The reason why it is important to allow the possibility to start the E-ASG with more lines than the number of lines we are interested in, and thus to write the coefficients $R^{m,k}_t(i,j)$, is that it allows to formulate a crucial branching property for $(F^{\cdot}_{G_{\beta}}(\cdot))_{\beta \geq 0}$ (see Lemma \ref{markpropf} of Subsection \ref{behaviour02}) and to have the above semigroup property. 
%O/N: EST-CE QUE C'EST LE BON ENDROIT POUR CETTE EXPLICATION? NON

\begin{remark} \label{generalizationofmomdual}
Proposition \ref{semigroupprop} allows to interpret Theorem \ref{momdiff} as a moment duality between \eqref{levymodelsdesimp} and a non-stochastic semigroup. In this sense, it is a generalization of the classical moment duality \eqref{mdsm} between the Wright-Fisher diffusion and the stochastic semigroup of the line counting process of the ASG (the later holds only in the case of one-sided selection). Let us note that $(R^{\cdot,\cdot}_t(\cdot,\cdot))_{t \geq 0}$ is indeed not a stochastic semigroup since coefficients can be negative. 
\end{remark}
The semigroup property (more precisely, the ingredients for its proof) and the study of the behavior of the coefficients $R^{m,k}_t(i,j)$ as $t$ goes to $0$ in Section \ref{behaviour01} allow, in Section \ref{diffeqsyst}, to establish a system of ODEs satisfied by the coefficients $R^{m,k}_t(i,j)$ that is analogous to the Kolmogorov forward equations. Let us first define some notations. 
\begin{defi} \label{defcoefedo}
For $j,k \geq 1$ we set $d_j := \lambda + j(j-1) + j\sigma$, $e_{k,j} := (k+1)k - j(j-1)$, $f_j := (j-1) \sigma$, $f_{k,j} := (k-1-j) \sigma$, and for $i,j \geq 1$ we set 
%\[ d_j := \lambda + \frac{j(j-1)}{2}, \ e_{k,j} := \frac{(k+1)k - j(j-1)}{2}, \ \tau(i,j):=\left\{\begin{array}{ll}
%            0 &\text{if $j < i-1$},\\
%            i(i-1)/2 &\text{if $j=i-1$},\\
%            \lambda \binom{i}{j-i} \mathbb{E} [ (1-S_1)^{2i-j} S_1^{j-i} ] &\text{if $i \leq j \leq 2i$},\\
%            0 &\textrm{if $j > 2i$}.
%            \end{array}\right. \]
%VOIR COMMENT CES COEFF SONT MODIFIES PAR LE FAIT QU'IL N'Y A PLUS DE $/2$ DANS LES TAUX. ON LES A TOUS CORRIGES -> OK
\[ \tau(i,j):=\left\{\begin{array}{ll}
            0 &\text{if $j < i-1$},\\
            i(i-1) &\text{if $j=i-1$},\\
%            \lambda \binom{i}{j-i} \mathbb{E} [ (1+J_1)^{2i-j} (-J_1)^{j-i} ] &\text{if $i \leq j \leq 2i$},\\
            \binom{i}{j-i} \int_{(-1,1)} (1+z)^{2i-j} (-z)^{j-i} \nu(dz) &\text{if $i \leq j \leq 2i$},\\
            0 &\textrm{if $j > 2i$}.
            \end{array}\right. \]
%where we recall that $\nu(\cdot)$ is defined in Subsection \ref{diffusion}. 
\end{defi}
The following result provides the system of ODEs satisfied by the coefficients $R^{m,k}_t(i,j)$. 
\begin{theo} \label{equadiffcoeffnew}
For any $m,k\geq 1, i \in \{ 1,...,m \}, j \in \{ 1,...,k \}$, $t \geq 0$, we have 
%\[ \frac{d}{dt} R^{m,k}_t(i,j) = \tau(j+1,j) R^{m,k+1}_t(i,j+1) + e_{k,j} R^{m,k+1}_t(i,j) -d_k R^{m,k}_t(i,j) + \mathds{1}_{\{k\text{ is even}\}} \sum_{l =1}^{j \wedge (k/2)} \tau(l,j) R^{m,k/2}_t(i,l). \]
\begin{align*}
\frac{d}{dt} R^{m,k}_t(i,j) &= \tau(j+1,j) R^{m,k+1}_t(i,j+1) + e_{k,j} R^{m,k+1}_t(i,j) + \mathds{1}_{k \geq 2} f_j R^{m,k-1}_t(i,j-1) \\
& + \mathds{1}_{j \leq k-1} f_{k,j} R^{m,k-1}_t(i,j) - d_k R^{m,k}_t(i,j) + \mathds{1}_{\{k\text{ is even}\}} \sum_{l =1}^{j \wedge (k/2)} \tau(l,j) R^{m,k/2}_t(i,l). 
\end{align*}
\end{theo}

The next step is to show the convergence of the coefficients $R^{m,k}_t(i,j)$ as $t$ goes to infinity. This is done in Section \ref{behaviourinfty} via the renewal structure of $(F^i_{G_{\beta}}(\cdot))_{\beta \geq 0}$ and coupling arguments. 
Combining this convergence with Theorem \ref{equadiffcoeffnew}, 
%from Subsection \ref{diffeqsyst} allows to 
we deduce linear relations satisfied by the limit coefficients. These results are gathered in the following theorem. 
\begin{theo} \label{cvcoeff}
For $m,k\geq 1, i \in \{ 1,...,m \}, j \in \{ 1,...,k \}$, 
\begin{eqnarray}
a^k_j := \underset{t \rightarrow \infty}{\lim} R^{m,k}_t(i,j) \label{cvcoeff1new}
\end{eqnarray}
exists and does not depend on $m,i$. Convergence is exponentially fast in $t$ and uniform in $m,i$. 
%Note that $R^{m,k}_{\infty}(i,j)=0$ if $j > k$. 
Moreover, $a^1_1=\pi(1)$ (where $\pi$ is defined in Section \ref{enlargedasglcp}) and for all $k \geq 1$ and $j \in \{ 1,...,k\}$, 
%\begin{align}
%\tau(j+1,j) R^{1,k+1}_{\infty}(1,j+1) + e_{k,j} R^{1,k+1}_{\infty}(1,j) & = d_k R^{1,k}_{\infty}(1,j) - \mathds{1}_{\{k\text{ is even}\}} \sum_{l =1}^{j \wedge (k/2)} \tau(l,j) R^{1,k/2}_{\infty}(1,l), \label{relreccoefnew} \\
%\sum_{l =1}^k R^{1,k}_{\infty}(1,l) & = \pi(k), \label{2emerelreccoefnew}
%\end{align}
\begin{align}
\tau(j+1,j) a^{k+1}_{j+1} + e_{k,j} a^{k+1}_{j} & = d_k a^{k}_{j} - \mathds{1}_{k \geq 2} f_j a^{k-1}_{j-1} - \mathds{1}_{j \leq k-1} f_{k,j} a^{k-1}_j - \mathds{1}_{\{k\text{ is even}\}} \sum_{l =1}^{j \wedge (k/2)} \tau(l,j) a^{k/2}_{l}. \label{relreccoefnew} 
%\sum_{l =1}^k a^{k}_{l} & = \pi(k), \label{2emerelreccoefnew}
\end{align}
%where we recall that $\pi$ is the stationary measure of the line counting process in the E-ASG. 
\end{theo}

\begin{remark} \label{extrarelation}
Let $k \geq 0$. Via \eqref{defcoeff} and \eqref{sommetotale=1}, we obtain $\sum_{l=1}^{k+1} R^{1,k+1}_t(1,l) = \mathbb{P}_{1} \left ( |V_{G_{t}}| = k+1 \right )$, which, as $t$ goes to infinity, yields the relation $\sum_{l =1}^{k+1} a^{k+1}_{l} = \pi(k+1)$. However, for $k \geq 1$, the latter turns out to be a linear combination of the $k$ equations given by \eqref{relreccoefnew} and therefore does not provide additional information. 
\end{remark}

\begin{remark}
Relations \eqref{relreccoefnew} can be stated in matrix form: 
for each $k \geq 1$, we have $$A_{k} . (a_1^{k+1}, a_2^{k+1}, \cdots, a_{k+1}^{k+1})^T = v_{k},$$
%\begin{eqnarray}
%A_{k} . \left( \begin{array}{c} a_1^{k+1} \\ a_2^{k+1} \\ \vdots \\ a_{k+1}^{k+1} \end{array} \right) = v_{k}, \label{systlincoef}
%\end{eqnarray}
where $A_{k}$ is a matrix of size $k \times(k+1)$ and $v_{k}$ is a vector of size $k$ that depends on the coefficients $(a^l_j)_{1 \leq l \leq k, 1 \leq j \leq l}$. 
\end{remark}

Finally, we are ready to state the announced series representation for $h(x)$. 
\begin{theo} \label{finalformula}
%\textbf{Def des $a^k_j$} and let 
%\[ P_k(X) := \sum_{j=1}^k a^k_j X^{j}. \]
Define the sequence of polynomials $(P_k(x))_{k \geq 1}$ via $P_k(x):= \sum_{j=1}^k a^k_j x^{j}$ (note that $P_k(0)=0$). 
%$P_k(X)$ has degree $k$ and null constant coefficient. 
$\mathbb{P}$-almost surely, $\lim_{t \rightarrow \infty} X(t)$ exists and belongs to $\{0,1\}$. Moreover, 
\begin{eqnarray}
h(x) = \sum_{k=1}^{\infty} P_k(x), \label{mainformula}
\end{eqnarray}
%NOTE : $h(x)$ EST MAINTENANT LA PROBA D'ETRE DE TYPE $0$ QUAND LA PROPORTION DE TYPE $0$ AU DEBUT EST $x$ (COMME DIT PLUS HAUT). 
where the series is normally convergent on $[0,1]$. Moreover, for any $m \geq 2$ and $x \in [0,1]$, 
%for any $m \geq 4 \vee 2^{7/2} (\sigma + \lambda)$ and $x \in [0,1]$, 
%\begin{align}
%\left | h(x) - \sum_{k=1}^{m-1} P_k(x) \right | \leq \frac{\sigma(\sigma + \lambda)+\lambda}{m-1} \exp \left ( (\lfloor \log_2(m) \rfloor -2) \left (\log(\sigma + \lambda) - \frac{\log(2)}{2} (\lfloor \log_2(m) \rfloor - 3) \right ) \right ). \label{boundfiniteapprox}
%\end{align}
\begin{align}
\left | h(x) - \sum_{k=1}^{m-1} P_k(x) \right | \leq \sum_{j \geq m} \pi(j) \leq C_1 e^{-C_2 (\log(m))^2}, \label{boundfiniteapprox}
\end{align}
where $C_1 = C_1(\sigma,\lambda)$ and $C_2 = C_2(\sigma,\lambda)$ are the explicit positive constants from Proposition \ref{recpilambda}. 
\end{theo}

%To conclude this subsection, let us mention that the system \eqref{systlincoef} can be solved explicitly, yielding an explicit recursion formula for the coefficients $a^k_j$. Indeed, $P_1(X),..., P_k(X)$ being defined we have 
%\begin{align*}
%a^{k+1}_{k+1} & = \frac{b_k(k+1) - \sum_{j=1}^k C_j^{k+1} \frac{(-1)^{j} (k+j-1)!}{j! (j-1)! (k-j+1)!} \sum_{i=j}^k \frac{(-1)^{i} i! (i-1)! (k-i+1)!}{(k+i-1)!} b_k(i)}{\sum_{j=1}^k C_j^{k+1} \frac{(-1)^{k+1-j} (k+1)! k! (k+j-1)!}{j! (j-1)! (2k)! (k-j+1)!}}, \\
%a^{k+1}_{j} & = \sum_{i=j}^k \frac{(-1)^{j} (k+j-1)!}{j! (j-1)! (k-j+1)!} \sum_{i=j}^k \frac{(-1)^{i} i! (i-1)! (k-i+1)!}{(k+i-1)!} b_k(i) \\ 
%& + \frac{(-1)^{k+1-j} (k+1)! k! (k+j-1)!}{j! (j-1)! (2k)! (k-j+1)!} a^{k+1}_{k+1}. 
%\end{align*}
%In the above expressions, $b_k(i)$ stands for the $i^{th}$ entry of the vector $b_k$. 
%
%PRENDRE EN COMPTE LE FAIT QUE TOUS LES COEFFICIENTS ONT CHANGE QUAND ON A CORRIGE LES ERREURS DANS LE SYSTEME LINEAIRE. DU COUP METTRE A JOUR LES EXPRESSION PLUS HAUT. 

%\subsection{Some numerical applications} \label{numappl}

Even if it does not seem possible to have a power series decomposition for $h(x)$ in the full generality of \eqref{levymodelsdesimp}, we can establish Taylor expansions of every order for $h(x)$ near $x=0$. Recall the coefficient $Q_t(i,j)$ defined in \eqref{coeff1indice}. Provided that it is well-defined, we see from \eqref{propagationformule1mix} that for any $n \geq 1$, 
\begin{align*}
%h^i_t(x) & = \sum_{j=1}^n \mathbb{E}_i \left [ \sum_{A \in \mathcal{P}(V_{G_t}) ; |A| = j} F^i_{G_t}(A) \right ] x^{j} + \mathbb{E}_i \left [ \sum_{j > n} \sum_{A \in \mathcal{P}(V_{G_t}) ; |A| = j} F^i_{G_t}(A) x^{j} \right ] \\
h^i_t(x) = \sum_{j=1}^n Q_t(i,j) x^{j} + \mathbb{E}_i \left [ \sum_{j > n} \sum_{A \in \mathcal{P}(V_{G_t}) ; |A| = j} F^i_{G_t}(A) x^{j} \right ]. 
\end{align*}
Thus, to obtain a Taylor expansion for $h(x)$, we need to consider the long-time behavior of the coefficients $Q_t(i,j)$ and of the remainder term. 
%Contrarily to the coefficients $R^{m,k}_t(i,j)$, we were not able to establish a semigroup property for the coefficients $Q_t(i,j)$. 
Similarly to the coefficients $R^{m,k}_t(i,j)$, the coefficients $Q_t(i,j)$ satisfy a system of ODEs analogous to the Kolmogorov forward equations:
\begin{theo} \label{equadiffcoeff1indice}

For any $i, j \geq 1$, $t \geq 0$, 
%the coefficient $Q_t(i,j)$ defined in \eqref{coeff1indice} is well-defined and 
we have 
\[ \frac{d}{dt} Q_t(i,j) = - d_j Q_t(i,j) +f_j Q_t(i,j-1) + \sum_{l =1}^{j+1} \tau(l,j) Q_t(i,l). \]
%IL FAUT PROBABLEMENT AUSSI METTRE LE CAS $l=j$ DU FAIT DES DEFINITIONS ACTUELLES DES COEFFICIENTS. OK
%ON POURRA RAJOUTER QUE $b_j$ EST LA SOMME DES $a^k_j$ EN UTILISANT LA CV UNIFORME DE LA SOMME DES $R^{i,k}_t(i,j)$ VERS $Q_t(i,j)$. OK 
%D'AILLEURS CA PERMET D'EXPRIMER TOUS LES $a^k_j$ EN FONCTION DES $b_j$ PAS VRAIMENT ET DONC LA SEULE INCONNUE DEVIENT $b_1$ NON. POUR L'INSTANT ON RESTE VAQGUE SUR CA PARCE QUE CA N'AIDE PAS TROP CONCRETEMENT A CALCULER LES $a^k_j$
\end{theo}
%As for the coefficients $R^{m,k}_t(i,j)$ we can determine the asymptotic behavior of the coefficients $Q_t(i,j)$ as $t$ goes to infinity and deduce the linear relations satisfied by the limit coefficients: 
The coefficients $Q_t(i,j)$ also converge as $t$ goes to infinity and their limits satisfy the following linear relations: 
\begin{theo} \label{cvcoeff1indice}

For $i, j \geq 1$, 
\begin{eqnarray}
b_j := \underset{t \rightarrow \infty}{\lim} Q_t(i,j) \label{cvcoeff1}
\end{eqnarray}
exists and does not depend on $i$. Moreover, $b_j = \sum_{k=1}^{\infty}a^k_j$. For all $j \geq 1$, we have 
\begin{eqnarray}
- d_j b_j + f_j b_{j-1} + \sum_{l=1}^{j+1} \tau(l,j) b_l = 0. \label{relreccoef}
\end{eqnarray}
%Moreover, 
%\begin{eqnarray}
%b_j = \sum_{k=1}^{\infty}a^k_j. \label{relcoefaetb}
%\end{eqnarray}

\end{theo}

Our result on the Taylor expansion of $h(x)$ then reads as follows. 

\begin{theo} \label{taylorexp}
For any $n \geq 1$, we have $h(x)=\sum_{k=1}^{n} b_k x^k+\underset{x \rightarrow 0}{o}(x^n)$. 
%there is $\epsilon_n > 0$ such that we have 
%\begin{eqnarray}
%\forall x \in (0,\epsilon_n), \ \left | h(x) - \sum_{k=1}^{n} b_k x^k \right | \leq x^{n+1/8}. \label{taylorexp1}
%\end{eqnarray}
\end{theo}

\begin{remark} \label{calcexprbk}
Relation~\eqref{relreccoef} allows to compute recursively explicit expressions for $b_k/b_1$. In particular, we get $b_2 = (\sigma-\int_{(-1,1)} z \nu(dz))b_1/2$ and 
\[ b_3 = \left (2\sigma-2\int_{(-1,1)} z \nu(dz)-\int_{(-1,1)} z^2 \nu(dz) \right ) \left (\sigma-\int_{(-1,1)} z \nu(dz) \right )b_1/12. \]
\end{remark}

The series representation (in Theorem \ref{finalformula}) and the estimates (in Theorem \ref{taylorexp}) for the fixation probability of a Wright-Fisher diffusion in L\'evy environment with jumps of both signs (a case where the classical moment duality does not hold) are, to the best of our knowledge, new. Novelty also lies in the strategy of encoding the combinatorics of the ASG into a function, leading to an extension of the classical moment duality \eqref{mdsm} to a moment duality between \eqref{levymodelsdesimp} and a non-stochastic semigroup (see Remark \ref{generalizationofmomdual}). Moreover, we establish analytical analogues to the usual probabilistic properties of the dual (Proposition \ref{semigroupprop} and Theorems \ref{equadiffcoeffnew}, \ref{cvcoeff}, \ref{equadiffcoeff1indice}, \ref{cvcoeff1indice}), even though that object is quite complex and not very explicit in the present case. 

%\subsection{Organization of the paper} \label{orgpap}

The content of the remaining sections has been described in Section \ref{orgpap1}. Now that our main results have been stated we can be more specific about Sections \ref{combinatorialfuct}, \ref{acotcfnew}, \ref{dlh(x)} and Appendix~\ref{append}. In Section \ref{combinatorialfuct} we prove Theorem \ref{propagationformule} and some estimates about the function $F^l_G(\cdot)$. In Section \ref{acotcfnew} we show that the coefficients $R^{m,k}_t(i,j)$ are well-defined, make appear a renewal structure in the E-ASG and in $(F^i_{G_{\beta}}(\cdot))_{\beta \geq 0}$, prove Proposition \ref{semigroupprop}, study the small-time behavior of $R^{m,k}_t(i,j)$ and prove Theorems \ref{equadiffcoeffnew} and \ref{cvcoeff} (for the sake of brevity, we prove Theorem \ref{equadiffcoeffnew} only in the subcase $\sigma=0$). Finally, we combine these results to prove Theorem~\ref{finalformula}. In Section \ref{dlh(x)} we prove Theorems \ref{equadiffcoeff1indice}--\ref{taylorexp}. Appendix~\ref{append} contains some technical proofs, including the proof of Proposition \ref{recpilambda}, and some technical lemmas. The section dependency is as follows. The definitions from Sections \ref{firstteps} and \ref{toolsmethres}, as well as Lemmas \ref{backwardtypedistribcoincide}, \ref{lemmemajointemporelle} and Proposition \ref{recpilambda}, are used all along the paper. Sections \ref{combinatorialfuct}, \ref{acotcfnew}, \ref{dlh(x)}, in this order, are built on each others and contain the proofs of the main results stated in Section \ref{mainresults}. Section \ref{exappgen} relies on the results proved in the other sections. Appendix \ref{A2} (resp. \ref{A3}, \ref{A1}, \ref{A4}) can be considered as being part of Section \ref{enlargedasglcp} (resp. \ref{behaviour01}, \ref{relwf-asgbis}, \ref{dlh(x)subsec1}). 

\section{Examples, applications, and generalizations} \label{exappgen}

\subsection{Examples and applications} \label{exapp}

\subsubsection{Case without random environment} \label{casenoenv}

As to illustrate the basic ideas of the method outlined in Section \ref{mainresults}, we here apply it in the classical case, where it simplifies, of no random environment, i.e. $L(t) := - \sigma t$ in~\eqref{levymodelsdesimp}. 
%Then \eqref{levymodelsdesimp} becomes $dX(s) = - \sigma X(s)(1-X(s)) ds + \sqrt{2X(s)(1-X(s))} dB(s)$. 
By definition of $F^l_{G}(A)$ in Section \ref{enlargedasgencodingfct}, we have in this case $F^i_{G_t}(A)\geq 0$ for any $t$. This and Theorem \ref{propagationformule} in particular imply that the heuristic reasoning just after Remark \ref{propagationformulequenched0} is now valid. Thus, $h^i_t(x) = \sum_{k\geq 1} Q_t(i,k) x^{k}$. Since the integrand in \eqref{coeff1indice} is null on $\{|V_{G_{t}}|<k\}$, using \eqref{sommetotale=1}, the non-negativity of coefficients $F^i_{G_t}(A)$, Lemma \ref{lemmemajointemporelle}, and Proposition \ref{recpilambda} we get $|Q_t(i,k)| \leq \mathbb{P}_i \left ( |V_{G_t}| \geq k \right ) \leq \sum_{j \geq k} \pi(j)/\pi(i) \leq C_1 e^{-C_2 (\log(k))^2}/\pi(i)$. Combining with \eqref{cvcoeff1} from Theorem \ref{cvcoeff1indice} and the convergence of $h^i_t(x)$ to $h(x)$ (from Proposition \ref{limht}), and applying dominated convergence, we get $h(x)=\sum_{k\geq 1} b_k x^k$, which improves Theorem \ref{taylorexp} in this case. Since in the present case the coefficients $a^k_j$ are non-negative, $h(x)=\sum_{k\geq 1} b_k x^k$ alternatively follows directly from \eqref{mainformula} of Theorem \ref{finalformula}, together with $b_j = \sum_{k=1}^{\infty}a^k_j$ (see Theorem \ref{cvcoeff1indice}). Then, \eqref{relreccoef} of Theorem \ref{cvcoeff1indice} (which is a consequence of Theorem \ref{equadiffcoeff1indice}) yields 
\begin{align}
b_2 = \frac{\sigma}{2} b_1 \ \text{and} \ \forall k \geq 2, b_{k+1} = \left ( \frac{k-1}{k+1} + \frac{\sigma}{k+1} \right ) b_k - \frac{(k-1) \sigma}{k(k+1)} b_{k-1}. \label{relrecbk}
\end{align}
Since the sequence $(b_1 \sigma^{k-1}/k!)_{k \geq 1}$ satisfies \eqref{relrecbk} and has same initial term as $(b_k)_{k \geq 1}$, we get $b_k = b_1 \sigma^{k-1}/k!$ for any $k \geq 1$. Therefore, $h(x)=b_1(e^{\sigma x}-1)/\sigma$ and, since $h(1)=1$, we get $h(x) = (e^{\sigma x}-1)/(e^{\sigma}-1)$. 
%\begin{align}
%h(x) = \frac{e^{\sigma x}-1}{e^{\sigma}-1}. \label{exprbk}
%\end{align}
We thus recover a classical result known in this case (see for example (7) of \cite{LKBW15} applied to $1-X$). 

\subsubsection{Case of one-sided selection} \label{exapponesidedenv}

Let us consider the case where the L\'evy process $L$ has only negative jumps, that is, $\nu((0,1))=0$. Then, type $0$ never expresses a selective advantage but type $1$ does. The definition of $F^l_{G}(A)$ in Section \ref{enlargedasgencodingfct} shows that in this case $F^i_{G_t}(A)\geq 0$ for any $t$. Reasoning as in Section \ref{casenoenv} yields the following improvement of Theorem \ref{taylorexp}. 
\begin{cor} \label{coroonesided}
If $\nu((0,1))=0$, then $h(x)=\sum_{k\geq 1} b_k x^k$, where the coefficients $b_k$ satisfy \eqref{relreccoef} and $\sum_{k\geq 1} b_k=1$. 
\end{cor}
This case is also covered by \cite[Cor.~2.7, applied to $1-X$]{cordvech}. That result yields in particular $h(x)=1-\sum_{k \geq 0} a_k (1-x)x^k = \sum_{k \geq 1} (a_{k-1}-a_k) x^k$, where the coefficients $(a_k)_{k\geq 0}$ are the ones from \cite[Th.~2.6, applied to $1-X$]{cordvech} (in the present case one needs to set $\theta=0$ in that theorem). Therefore, the coefficient $b_k$ from Corollary \ref{coroonesided} should equal $\tilde b_k:=a_{k-1}-a_k$. $(a_k)_{k\geq 0}$ satisfies a recursion relation given by (2.18) of \cite{cordvech}, which can be re-written as a recursion relation for $(\tilde b_k)_{k\geq 0}$. The later and \eqref{relreccoef} seem to indeed generate the same sequence. In particular, a straightforward calculation shows that the recursion relation for $(\tilde b_k)_{k\geq 0}$ provides the same values for $\tilde b_2/ \tilde b_1$ and $\tilde b_3/ \tilde b_1$ as those given by Remark \ref{calcexprbk} for $b_2/b_1$ and $b_3/b_1$. 
%Let us also mention that, in \cite{CSW19}, another type of one-sided fluctuating selection is studied. 

\subsubsection{Martingale case}

Recall from Theorem \ref{finalformula} that, almost surely, $\lim_{t \rightarrow \infty} X(t)$ exists and belongs to $\{0,1\}$. In the particular case where 
$\int_{(-1,1)} z \nu(dz)=\sigma$, $X$ is a bounded martingale so $\mathbb{E} [ \lim_{t \rightarrow \infty} X(t) | X(0)=x ]=x$. Therefore, $h(x)=x$ so that $b_1=1$ and $b_k=0$ for $k \geq 2$. This is in line with our results. Indeed, when $\int_{(-1,1)} z \nu(dz)=\sigma$, Remark \ref{calcexprbk} yields that $b_2=0$. Since $\tau(1,j)=0$ for $j\geq 3$ (see Definition \ref{defcoefedo}), using \eqref{relreccoef} we obtain by induction that also $b_k=0$ for all $k \geq 2$. 

\subsubsection{A numerical application} \label{numappl}

%The proof of Theorem \ref{taylorexp} allows to determine an explicit domain on which the Taylor expansion of order $n$ of $h(x)$ is valid. However, since the bounds in that proof are not optimal, this explicit domain has mostly a theoretical interest (of proving the existence of the Taylor expansion), and the actual domain of validity of the approximation by the Taylor expansion of $h(x)$ is larger. It seems from numerical computations that, for small enough parameter choices, the absolute values of coefficients $b_k$ decrease at good speed toward $0$, and that the domain of validity is the whole interval $[0,1]$. In other words it seems that in such cases we have $h(x)=\sum_{k\geq 1} b_k x^k$, as in the case with one-sided selection. 

Let us apply the formulas provided by Theorems \ref{cvcoeff1indice} and \ref{taylorexp} to some particular parameter choices. Assume $\sigma > 0$ and $\nu$ is of the form $\nu(dx) = \lambda \delta_a(dx)$ for some $\lambda \geq 0$ and $a \in [0,1)$. For a given choice of $(\sigma, \lambda, a)$, the numbers $b_k/b_1$ can be computed recursively thanks to \eqref{relreccoef} from Theorem \ref{cvcoeff1indice}. 
%In this case, $\sigma > 0$ models the constant selective advantage of type $1$ and the occasional jumps of size $a$ model the constant selective advantage of type $0$. 
For large $k$, those numbers explode for many parameter choices. This suggests that the Taylor expansion from Theorem~\ref{taylorexp} cannot be improved into a power series decomposition. However, it turns out that if the parameters are sufficiently small, $\lvert b_k\rvert $ decays rather fast towards $0$. In such cases, we expect $h(x)=\sum_{k\geq 1} b_k x^k$, similarly to the case with one-sided selection (see Section \ref{exapponesidedenv}). Assuming this is indeed true, $1/b_1$ could be approximated by summing the previously computed values of $b_k/b_1$. For $\sigma=0.8$, $\lambda=0.8$, and $a\in\{0,0.1,0.2,0.3\}$, this leads to the approximated values of coefficients $b_k$ given in Table \ref{tableofvalues} and to the graphical representations of $h(x)$ given in Figure \ref{represgraphich(x)}. Note that the case $a=0$ is the classical case without random environment, considered in Section~\ref{casenoenv}. 
%\begin{table}[h!]
%\centering
%\begin{tabular}{ |c|c|c|c|c|c|c|c|c| }
%\hline
% & $b_1 \approx$ & $b_2 \approx$ & $b_3 \approx$ & $b_4 \approx$ & $b_5 \approx$ & $b_6 \approx$ & $b_7 \approx$ & $b_8 \approx$ \\
%\hline
%$a=0.1$ & $0.6830193$ & $0.2458870$ & $0.0586850$ & $0.0106059$ & $0.0015752$ & $0.0002021$ & $0.0000229$ & $0.0000023$ \\
%\hline
%$a=0.2$ & $0.7145930$ & $0.2286698$ & $0.0475633$ & $0.0078140$ & $0.0011734$ & $0.0001641$ & $0.0000201$ & $0.0000020$ \\
%\hline
%$a=0.3$ & $0.7473968$ & $0.2092711$ & $0.0365527$ & $0.0056493$ & $0.0009582$ & $0.0001497$ & $0.0000206$ & $0.0000017$ \\
%\hline
%\end{tabular}
%\caption{Calculated approximations of coefficients $b_1$,..., $b_8$ when $\sigma=0.8$, $\lambda=0.8$, in the cases $a=0.1$, $a=0.2$, and $a=0.3$.}
%\label{tableofvalues}
%\end{table}
\begin{table}[h!]
\centering
\begin{tabular}{ |c|c|c|c|c|c|c|c| }
\hline
 & $b_1 \approx$ & $b_2 \approx$ & $b_3 \approx$ & $b_4 \approx$ & $b_5 \approx$ & $b_6 \approx$ & $b_7 \approx$ \\
\hline
$a=0\phantom{.1}$ & $0.6527730$ & $0.2611092$ & $0.0696291$ & $0.0139258$ & $0.0022281$ & $0.0002971$ & $0.0000340$ \\
\hline
$a=0.1$ & $0.6830193$ & $0.2458870$ & $0.0586850$ & $0.0106059$ & $0.0015752$ & $0.0002021$ & $0.0000229$ \\
\hline
$a=0.2$ & $0.7145930$ & $0.2286698$ & $0.0475633$ & $0.0078140$ & $0.0011734$ & $0.0001641$ & $0.0000201$ \\
\hline
$a=0.3$ & $0.7473968$ & $0.2092711$ & $0.0365527$ & $0.0056493$ & $0.0009582$ & $0.0001497$ & $0.0000206$ \\
\hline
\end{tabular}
\caption{Calculated approximations of coefficients $b_1$,..., $b_7$ when $\sigma=0.8$, $\lambda=0.8$, and $a\in\{0,0.1,0.2,0.3\}$.}
\label{tableofvalues}
\end{table}

\begin{figure}
\centering
\includegraphics[scale=0.40]{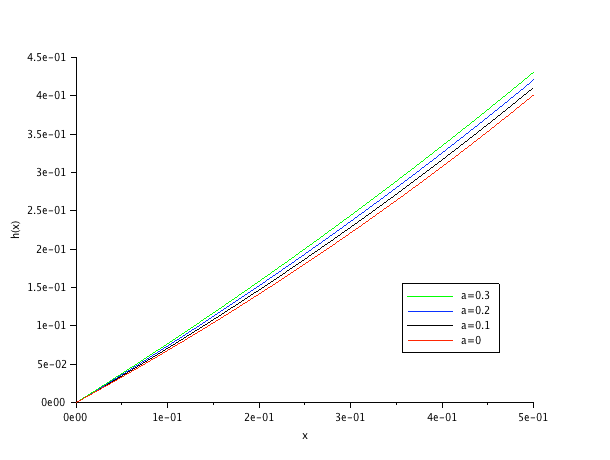} 
\caption{$h(x)$ as a function of $x$ when $\sigma=0.8$, $\lambda=0.8$, and $a\in\{0,0.1,0.2,0.3\}$. For $a=0$, the curve is generated from the classical formula $h(x) = (e^{\sigma x}-1)/(e^{\sigma}-1)$ (see Section \ref{casenoenv}).}
\label{represgraphich(x)}
\end{figure}

To estimate $h(x)$ for small $x$ and under parameter choices that make $b_k/b_1$ explode, one can determine an approximation of $b_1$ (for example via simulations of the E-ASG) and then use Theorems \ref{cvcoeff1indice} and \ref{taylorexp}. In order to compute suitable approximations of $h(x)$ on $[0,1]$, one can determine approximations of a reasonable number of coefficients $a^k_j$ and then use Theorem~\ref{finalformula}. Let us mention that, even though the upper bound $C_1 e^{-C_2 (\log(m))^2}$ in \eqref{boundfiniteapprox} has theoretical interest, the upper bound $\sum_{j \geq m} \pi(j)$ is better and can be calculated numerically with high precision (see the discussion after Proposition~\ref{recpilambda}). Therefore, the later seems more useful in practice in measuring how well $h(x)$ is approximated by $\sum_{k=1}^{m-1} P_k(x)$. 

\subsection{Relation with other models and generalizations} \label{othermodgene}

\subsubsection{Other types of fluctuating selection} \label{otherfluctsel}

Different types of fluctuating selection exist in the literature. There are models without jumps, but with selection coefficient $\sigma$ fluctuating over time. For example, these fluctuations can be random~\cite{GGP19, GPS19} or be a function of the type-frequency process~\cite{GS18, CHS19} (this last setting is called frequency-dependent selection). Another setup arises when the L\'evy process $L$ in \eqref{levymodelsdesimp} is replaced by a continuous L\'evy process, i.e. a drifted Brownian motion, see e.g.~\cite{BEK19}. The effects of jumps can also be different or more general, i.e. the term $X(s-)(1-X(s-)) dL(s)$ in \eqref{levymodelsdesimp} can be replaced by a different or more general function of $X(s-)$ and of the jump of the environment at time $s$, see e.g. \cite{BCM19,CSW19}. In the following subsection, we consider \eqref{levymodelsdesimp}, but in a fluctuating (or rather in-homogeneous) environment, which is yet another type of fluctuating selection. 

\subsubsection{Partial generalization to the case of an in-homogeneous environment} \label{inhomenv}

In the model \eqref{levymodelsdesimp}, the environmental influence on selection is given by the jumps of the L\'evy process $L$, that is, by a Poisson point process $(t_k,j_k)_{k \in I}$ on $[0,\infty) \times (-1,1)$ with intensity measure $dt \times \nu$ where $\nu((-1,1))<\infty$. 
%(in the most general form this last condition can be relaxed to $\int_{(-1,1)}|x|\nu(dx)<\infty$). 
In other words, the distribution of the environment is homogeneous. It can also make sense to consider a model with an in-homogeneous distribution of the environment, which allows to take into account some time-dependent tendencies; for example, increasing frequency of extreme events over time and/or change of the typical intensity of those events. In this case, since the distribution and rate of environmental events changes over time, we model the environment by a Poisson point process $(t_k,j_k)_{k \in I}$ on $[0,\infty) \times (-1,1)$ with intensity measure $\tilde \nu$ which satisfies $\tilde \nu (\{t\} \times (-1,1))=0$ and $\tilde \nu([0,t] \times (-1,1)) < \infty$ for any $t\geq 0$. Such an environment is quite general, allowing different kinds of situations to fit into this framework. 
%(in the most general form this last condition can be relaxed to $\int_{[0,t] \times (-1,1)}|x|\tilde \nu(ds,dx)<\infty$). 
Consider $L$ being defined via $L(s) := -\sigma s + \sum_{k:\,t_k \leq s} j_k$ as in the homogeneous case. Then, according to Remark 9.9 in \cite{KenIti1999}, $L$ is an \textit{additive process} in the sense of Definition 1.6 of \cite{KenIti1999} (but, in general, no longer a L\'evy process). Thus, replacing the L\'evy process $L$ by the just-defined additive process turns SDE~\eqref{levymodelsdesimp} into a model with in-homogeneous environment. 

%\textbf{Example:} Let $Z$ be a random variable on $\mathbb{R}$. A possible choice of $\tilde \nu$ is  $\tilde \nu(dt,dx):=\lambda_t dt \times \mathbb{P} (g_t(Z) \in dx)$, where $\lambda_t \geq 0$ and $g_t : \mathbb{R} \mapsto (-1,1)$ are Borel measurable. $\lambda_t$ is the instantaneous rate of jumps at time $t$ and $g_t$ determines the distribution of those jumps. Different kinds of situations fit into this framework. For example, to cover the situation where extreme events become increasingly frequent, while at the same time their magnitude decreases, one chooses $\lambda_t$ and $|g_t(z)|$ such that they respectively increase to infinity and decrease (as $t$ goes to infinity). 
%%and to determine the asymptotic properties of the model in terms of the asymptotic behaviors of $t \mapsto \lambda_t$ and $t \mapsto r_t$. 

The quenched ASG (resp. E-ASG) generalizes to the in-homogeneous case using Definition~\ref{defquenchedasg} (resp.~\ref{defquenchedeasg}) given a realization of the environment. The annealed ASG (resp. E-ASG) on $[0,T]$ is the branching-coalescing particle system $(A^{T}_{s})_{s \in [0,T]}$ (resp. $(G^{T}_{s})_{s \in [0,T]}$) arising from the quenched ASG (resp. E-ASG) by randomizing the environment. We note that the annealed ASG and E-ASG need to have the same (complicated) time-line as the quenched ASG and E-ASG because of the non-homogeneity of the distribution of the environment. The quenched coefficients $R^{m,k, \omega}_{r,t}(i,j)$ can be defined via~\eqref{defcoeffquenched}. The annealed coefficients $R^{m,k}_{r,t}(i,j)$ arise by integration of the quenched coefficients with respect to the random environment. \eqref{majonouveausgnew} in Lemma \ref{sgwelldef} of Section \ref{dawd} translates to the in-homogeneous case and guarantees the well-definedness of the coefficients.

Our methodology partially generalizes to the in-homogeneous case. All our quenched results remain of course true. Integrating the quenched result from Theorem \ref{momdiff} with respect to the law of the in-homogeneous environment yields: 
\begin{theo} \label{momdiffinhom} 
In the in-homogeneous case, for any $l \geq 1, m \geq l, x \in [0,1]$ and $T \geq 0$, we have
\[ \mathbb{E} \left [ (X(T))^l \mid X(0)=x \right ] = \mathbb{E}^{T}_m \left [ \sum_{A \in \mathcal{P}(V_{G^{T}_0})} F^l_{G^{T}_0}(A) x^{|A|} \right ] = \sum_{k=1}^{\infty} \sum_{j=1}^k R^{m,k}_{0,T}(l,j) x^{j}. \]
\end{theo}
In other words, the moment duality between \eqref{levymodelsdesimp} and a non-stochastic semigroup also holds in the in-homogeneous case. Imposing some additional conditions on the measure~$\tilde \nu$, it seems possible to prove convergence of $X(t)$ (almost surely) and of the coefficients $R^{m,k}_{0,t}(i,j)$ as $t$ goes to infinity, and to generalize the representation~\eqref{mainformula} of $h(x)$ with the limit coefficients.
%to a limit coefficient $\tilde a^k_j$ that does not depend on $m$ and $i$, 
However, getting linear relations such as \eqref{relreccoefnew} for the limit coefficients is more challenging. In the homogenous case, Theorem \ref{equadiffcoeffnew} strongly relies on the homogeneity of the environment. 
%: on one hand because it is an homogeneous system of ODEs, and on the other hand because it relies of the time-reversed timeline of the E-ASG. 
Therefore, even if one finds conditions on $\tilde \nu$ that allow to deduce that $\frac{d}{dt}R^{m,k}_{0,t}(l,j)$ exists, it seems hard to generalize Theorem \ref{equadiffcoeffnew} (and therefore \eqref{relreccoefnew}) to the in-homogeneous case. 

\subsubsection{Generalization to the case with mutations} \label{casemut}

In this subsection we discuss how the methods and results of the present paper extend to the model with mutations. For the sake of brevity, we do so without proofs. Here,
\begin{equation}\label{WFD}
 dX(s)=\left[\theta\nu_0(1-X(s))-\theta\nu_1 X(s) \right]ds+ X(s-)(1-X(s-)) dL(s) + \sqrt{2X(s)(1-X(s))} dB(s), 
\end{equation}
where $L$ and $B$ are as in \eqref{levymodelsdesimp}, $\theta > 0$ and $\nu_0,\nu_1\in(0,1)$ with $\nu_0+\nu_1=1$. The extra term (compared to~\eqref{levymodelsdesimp}) reflects that individuals mutate at positive rate with the resulting type being~$0$ (resp. $1$) with probability $\nu_0$ (resp. $\nu_1$). We are interested in the stationary distribution of the solution of \eqref{WFD}. For $l \geq 1$, let $M_l$ denote the moment of order $l$ of that stationary distribution. 
%Let $X(\infty)$ be a random variable following that stationary distribution. 

In this model the ASG is defined as in Definition \ref{defquenchedasg}, but with the additional rule that each line is decorated with a mutation to type $0$ (resp. $1$) with rate $\theta \nu_0$ (resp. $\theta \nu_1$). We define the E-ASG as in Definition \ref{defquenchedeasg}, but with the additional rule that each line is decorated with a mutation to type $0$ (resp. $1$) with rate $\theta \nu_0$ (resp. $\theta \nu_1$) and is subsequently terminated. For the annealed E-ASG on $[0,\infty)$, there is then almost surely a finite time $T_{ext}$ at which the E-ASG enters (and remains in) the state of $0$ lines. 

Extending the encoding function $F_G^l(\cdot)$ to the mutation case requires modifying some definitions of Section~\ref{enlargedasgdefnot}. First, we need to allow two extra types of generations for elements of $\mathbb{G}_m$: when each line in generation $n$ has exactly one son, but exactly one line has no son and is decorated with a mutation to type $0$ (resp. $1$). Then, there is one less line in generation $n+1$ than in generation $n$. We refer to the new generation as \textit{type $0$ (resp. $1$) mutation generation} of $G$. Next, we need to include the empty set of lines into $\mathcal{P}(V_G)$. Other definitions from Section \ref{enlargedasgdefnot} are unchanged. For $l \geq 1$, we recursively define the function $\tilde F_G^l(\cdot)$ by \eqref{defFjump}--\eqref{defFcoal} as in Section \ref{enlargedasgencodingfct}, but with the following additional rules: 
\begin{itemize}
\item If $n+1$ is a type $0$ mutation generation of $G$, let us denote by $L_M$ the line from generation $n$ that is decorated with a mutation. Then, for any $A \in \mathcal{P}(V_G)$, we set 
\[ \tilde F^l_G(A)=\tilde F^l_{\pi_n(G)}(P(A))+\tilde F^l_{\pi_n(G)}(P(A)\cup \{L_M\}). \]
\item If $n+1$ is a type $1$ mutation generation of $G$, then we set $\tilde F^l_G(A)=\tilde F^l_{\pi_n(G)}(P(A))$ for any $A \in \mathcal{P}(V_G)$. 
\end{itemize}
These additional rules can be interpreted as follows. Sets of lines that contain a line subject to a mutation are removed. If the mutation is to type $1$, the values of $\tilde F^l_G$ on other sets of lines are unchanged. If the mutation is to type $0$, the contributions of removed sets of lines are added to the remaining sets of lines. Under this definition of $\tilde F_G^l(\cdot)$, we can possibly have $\tilde F^l_{G}(\emptyset) \neq 0$. 

Under the above definition of $\tilde F_G^l(\cdot)$, we claim $\tilde F_G^l(\cdot)$ satisfies \eqref{entre0et10mix}, and \eqref{WFD} satisfies the following generalization of Theorem \ref{momdiff}: For any $l \geq 1, m \geq l, x \in [0,1]$ and $T \geq 0$, 
\begin{align}
\mathbb{E} \left [ (X(T))^l | X(0)=x \right ] = \mathbb{E}_m \left [ \sum_{A \in \mathcal{P}(V_{G_T})} \tilde F^l_{G_T}(A) x^{|A|} \right ]. \label{momdiffmuteq} 
\end{align}
%and the analogous identity holds in the quenched setting. 
%It is indeed possible to see that the proof of Theorem \ref{momdiff} extends without much efforts to a proof of \eqref{momdiffmuteq} and 

Similarly to Definition \ref{defcoefsg}, we define for any $i \geq 1,j\geq 0$ and $t \geq 0$, \\
\noindent $\tilde Q_t(i,j) := \mathbb{E}_i [ \sum_{A \in \mathcal{P}(V_{G_{t}}) ; |A|=j} \tilde F^i_{G_{t}}(A) ]$, 
%\begin{eqnarray}
%\tilde Q_t(i,j) := \mathbb{E}_i \left [ \sum_{A \in \mathcal{P}(V_{G_{t}}) ; |A|=j} \tilde F^i_{G_{t}}(A) \right ], \label{coeff1indicebis}
%\end{eqnarray}
and for any $j\geq 0$ and $t \geq 0$, $\tilde Q_t(0,j):=\mathds{1}_{j=0}$. 
%Note that the right-hand side of \eqref{momdiffmuteq}  does not depend on $m$ so, in this definition, $\mathbb{E}_i$ can be replaced by $ \mathbb{E}_m$ for any $m \geq i$. 

Note that, $\mathbb{P}_m$-almost surely, we have for $t \geq T_{ext}$, $\sum_{A \in \mathcal{P}(V_{G_t})} \tilde F^l_{G_t}(A) x^{|A|}=\tilde F^l_{G_t}(\emptyset)$. From~\eqref{momdiffmuteq} and dominated convergence we deduce that 
\begin{align}
M_l = \mathbb{E}_l \left [ \tilde F^l_{G_{T_{ext}}}(\emptyset) \right ] = \lim_{t \rightarrow \infty} \mathbb{E}_l [ \tilde F^l_{G_t}(\emptyset) ] = \lim_{t \rightarrow \infty} \tilde Q_t(l,0), \label{mlequalqtl0}
\end{align}
where the last equality follows from the definition of $\tilde Q_t(i,j)$. It is then reasonable to expect the methodology from Section \ref{acotcfnew} to be applicable. This allows to prove that the coefficients $\tilde Q_t(i,j)$ satisfy a system of ODEs in which each $\frac{d}{dt} \tilde Q_t(i,j)$ is a linear combination of coefficients $\tilde Q_t(k,j)$ for finitely many indices $k \geq 0$. Here, the system of ODEs is analogous to the Kolmogorov backward equations. Letting $j=0$ in those equations and combining with \eqref{mlequalqtl0} leads to linear relations satisfied by the moments $(M_l)_{l \geq 1}$. 

%As we said, an interest of the methodology developed in the present paper is its robustness and its possibility to be extended to other settings or more general models. We state without proofs the extentions of our results to more general models. 
%
%This shows that, compared with classical ASG methods where the adaptation to the case with mutations is relatively costly, the method developed in the present paper is easily extended to the case with mutations. 

\subsubsection{Generalization to the case with colonies and migrations} \label{casecolmig}

In this subsection we consider a population divided into $K$ colonies. Each of them has its own selection mechanism and there is migration between them. Fix $m(i,j) \geq 0$ (for $i \neq j$) such that $(m(i,j))_{1 \leq i,j \leq K}$ is an irreducible rate matrix. Let $(\rho_1,\dots,\rho_K)$ be the corresponding equilibrium probability distribution. If $(m(i,j))_{1 \leq i,j \leq K}$ represents the migration rates between colonies (i.e. on the infinitesimal time interval $(t,t+dt]$, a proportion $m(i,j) dt$ of individuals from colony $i$ have moved to colony $j$), then the sizes of the colonies (with respect to the total population size) are given at equilibrium by $\rho_1,\dots,\rho_K$. This model is then described by the system of SDEs 
\begin{align}
dX_i(t) & = X_i(t-)(1-X_i(t-)) dL_i(t) + \sqrt{\frac{2}{\rho_i}X_i(t)(1-X_i(t))} dB_i(t) \nonumber \\
& + \left ( \sum_{j=1 ; j \neq i}^{K} \tilde m(i,j) (X_j(t) - X_i(t)) \right ) dt, \label{sdecolonies}
\end{align}
where $i\in \{1,\ldots,K\} $, $\tilde m(i,j):= m(j,i) \rho_j/\rho_i$ are known as the \textit{backward migration rates}, $B_1,\dots,B_K$ are independent Brownian motions, $L=(L_1,\dots,L_K)$ is an independent L\'evy process built as the sum of a compound Poisson process in $\mathbb{R}^K$ with jumps in $(-1,1)^K$ and of a drift, where the drift vector $(-\sigma_1 t,\dots,-\sigma_K t)$ has non-positive components. $X_i(t)$ represents the proportion of type $0$ individuals in colony $i$ at time $t$. The last term in \eqref{sdecolonies} represents the effect of migrations. For the more classical version of this model without jumps, we refer to~\cite{10.1214/16-EJP3355}. 

The ancestral structure can be extended to this model. Here, lines of the ASG are partitioned into $K$ subsets that correspond to the $K$ colonies. Only lines in the same colony can coalesce (at rate $2/\rho_i$ for a pair in colony $i$). Single branchings favoring type $1$ occur in colony $i$ at rate $\sigma_i$. In the quenched setting, at each jump of the environment, each line from colony $i$ independently branches with probability $|J_i|$, where $J_i$ is the $i^{th}$ component of the jump. A line from the $i^{th}$ colony migrates to the $j^{th}$ colony at rate $\tilde m(i,j)$. For the classical model without jumps, this definition of the ASG reduces to the one in~\cite{10.1214/16-EJP3355}. 

The E-ASG arises from the ASG along the lines of Section \ref{enlargedasg}. Now, the E-ASG additionally has \textit{migration generations}. For $l=(l_1,\dots,l_K)$, we can define a function $\tilde F^l_{G}(\cdot)$ similarly as in Section \ref{enlargedasgencodingfct}. Though, here the mapping is from $\mathcal{P}(V^1_G) \times \dots \times \mathcal{P}(V^K_G)$ to $\mathbb{R}$, where $V^i_G$ is the set of lines in the last generation of $G$ that belong to colony $i$ and $\mathcal{P}(V^i_G)$ is the corresponding family of (possibly empty) subsets. Definitions \eqref{defFjump}--\eqref{defFcoal} straightforwardly generalize to this setting. In the case where the $n+1^{th}$ generation is a migration generation (where the migration is from colony $i$ to colony $j$ and $L_M$ denotes the migrated line), one should set
\begin{align*}
		&\tilde F_{G}(A_1,\dots,A_i,\dots,A_j,\dots,A_K)\\
		&=\begin{cases} \tilde F_{\pi_n(G)}(P(A_1),\dots,P(A_i) \cup \{L_M\},\dots,P(A_j \setminus \{L_M\}),\dots,P(A_K)), & \text{if} \ L_M \in A_j,  \\
		\tilde F_{\pi_n(G)}(P(A_1),\dots,P(A_i),\dots,P(A_j),\dots,P(A_K)),  & \text{if} \ L_M \notin A_j. \end{cases}  \end{align*}

%Then, one should be able to prove the following generalization of Theorem \ref{momdiff}: For any $l := (l_1,...,l_K) \in (\mathbb{Z}_+)^K \setminus \{(0,\dots,0)\}$, $x := (x_1,...,x_K) \in [0,1]^K$ and $T \geq 0$, 
%\begin{align*}
%\mathbb{E} \left [ \prod_{i=1}^K (X_i(T))^{l_i} \big | X(0)=x \right ] = \mathbb{E}_l \left [ \sum_{(A_1,...,A_K) \in \mathcal{P}(V^1_{G_T}) \times ... \times \mathcal{P}(V^K_{G_T})} F_{G_T}(A_1,...,A_K) \prod_{i=1}^K x_i^{|A_i|} \right ]. 
%%\label{dualitycolonies}
%\end{align*}
Theorem \ref{momdiff} should be generalizable to cover product moments $\mathbb{E} [ \prod_{i=1}^K (X_i(T))^{l_i}]$ (for initial conditions of the form $X_i(0)=x_i, i=1,...,K$) with this choice of $\tilde F^l_{\cdot}(\cdot)$. Moreover, it should be possible to extend the other results from Section \ref{mainresults} by generalizing the coefficients $R^{m,k}_t(i,j)$ from Definition \ref{defcoefsg} to multi-index coefficients. Eventually, this should lead to a representation for the fixation probability $h(x_1,\dots,x_K)$, and linear relations for the coefficients of the monomials. 

\section{A function that catches the combinatorics of the ASG} \label{combinatorialfuct}

\subsection{Expression of $h_T^l(x)$: Proof of Theorem \ref{propagationformule}} \label{dualrel}

The idea is to update the expression of $h^l_T(x)$ along events of the E-ASG. More precisely we take a subdivision $((kT/2^n,(k+1)T/2^n])_{k < 2^n}$ of $(0,T]$ so that, for each interval of the subdivision there is, with high probability, either $0$ or $1$ transition of the E-ASG. Then, by induction on $k$ we prove \eqref{propagationformule2}, which essentially says that a quantity related to $h^l_T(x)$ is expressed in terms of the coefficients $F^l_{G_{kT/2^n}}(A)$ and of the indicator functions of the events $E({kT/2^n},T,A,x)$ (introduced in Definition \ref{type0events}) for sets $A \in \mathcal{P}(V_{G_{kT/2^n}})$. To do so we compute, for $A \in \mathcal{P}(V_{G_{kT/2^n}})$, the indicator function of $E({kT/2^n},T,A,x)$ in terms of the indicator functions of events $E({(k+1)T/2^n},T,B,x)$ for $B \in \mathcal{P}(V_{G_{(k+1)T/2^n}})$, depending on the evolution of the E-ASG on $(kT/2^n,(k+1)T/2^n]$. We then plug this into the induction hypothesis of order $k$ and use the definition of $F^l_{\cdot}(\cdot)$ to recognize $F^l_{G_{(k+1)T/2^n}}(\cdot)$ in the resulting expression, which yields the induction hypothesis of order $k+1$. Then, \eqref{propagationformule1mix} will follow. 

We fix $m\geq 1, l \in \{ 1,...,m \}$, $x \in [0,1]$ and $T > 0$. Let us consider the annealed E-ASG on $[0,T]$, starting with $m$ ordered lines $L_1,..., L_m$ at time $\beta=0$, and apply the type assignment procedure on $[0,T]$ with initial condition $x$ (see Definition \ref{typeaspreasg}). 
%Clearly we have $h^l_T(x) = H(0,T,\{ L_1,..., L_l \},x) = \mathbb{E}_m [ \mathds{1}_{E(0,T,\{ L_1,..., L_l \},x)} ]$. 
According to the discussion after Definition \ref{type0events}, we have $h^l_T(x) = \mathbb{E}_m [ \mathds{1}_{E(0,T,\{ L_1,..., L_l \},x)} ]$. 
For $n \geq 1$ let us define 
%\[ E_{n,k} := \left \{ \forall i \in \{1,..., k\}, \ \sharp \{ \text{transitions of E-ASG between} \ (i-1)T/2^n \ \text{and} \ iT/2^n \} \leq 1 \right \}, \]
%\[ E'_{n,k} := \left \{ \forall i \in \{k+1, 2^n\}, \ \sharp \{ \text{transitions of E-ASG between} \ (i-1)T/2^n \ \text{and} \ iT/2^n \} \leq 1 \right \}, \]
\[ E_{n} := \left \{ \forall i \in \{1,..., 2^n\}, \ \sharp \{ \text{transitions of E-ASG on} \ ((i-1)T/2^n, iT/2^n] \} \leq 1 \right \}. \]
Note that $\mathds{1}_{E_{n}}$ $\mathbb{P}_m$-almost surely increases to $1$ as $n$ goes to infinity. 
%$E_{n,k} \in \mathcal{F}_{kT/2^n}$, that for any $k \in \{1,...,2^n \}$ we have $E_{n} = E_{n,k} \cap E'_{n,k}$, and that $\mathbb{P}_m (E_n) \underset{n \rightarrow \infty}{\longrightarrow} 1$. 
In particular we have 
\begin{eqnarray}
h^l_T(x) = \lim_{n \rightarrow \infty} \mathbb{E}_m [ \mathds{1}_{E_{n}} \mathds{1}_{E(0,T,\{ L_1,..., L_l \},x)} ]. \label{limitthejumps}
\end{eqnarray}
For $0 \leq a< b$, let us denote by respectively $I_0(a,b)$, $I_1(a,b)$, $I_2(a,b)$, and $I_3(a,b)$ the events where the E-ASG has, respectively, no transition on $(a,b]$, exactly one transition on $(a,b]$ that is a coalescence, exactly one transition on $(a,b]$ that is a multiple branching, and exactly one transition on $(a,b]$ that is a single branching. Let us fix $n \geq 1$ and prove by induction on $k \in \{0,1,...,2^n\}$ that 
%\begin{eqnarray}
%\mathbb{E}_m [ \mathds{1}_{E_{n}} \mathds{1}_{E(0,T,\{ \emptyset \},x)} ] = \mathbb{E}_m \left [ \mathds{1}_{E_{n,k}} \sum_{A \in \mathcal{P}(V_{G_{kT/2^n}})} F^1_{G_{kT/2^n}}(A) \mathbb{E}_m \left [ \mathds{1}_{E'_{n,k}} \mathds{1}_{E({kT/2^n},T,A,x)} | G_{kT/2^n} \right ] \right ], \label{propagationformule2}
%\end{eqnarray}
$\mathbb{P}_m$-almost surely, 
\begin{eqnarray}
\mathds{1}_{E_{n}} \sum_{A \in \mathcal{P}(V_{G_{kT/2^n}})} F^l_{G_{kT/2^n}}(A) \mathds{1}_{E({kT/2^n},T,A,x)} \in [0,1], \label{entre0et1}
\end{eqnarray}
and that 
\begin{eqnarray}
\mathbb{E}_m [ \mathds{1}_{E_{n}} \mathds{1}_{E(0,T,\{ L_1,..., L_l \},x)} ] = \mathbb{E}_m \left [ \mathds{1}_{E_{n}} \sum_{A \in \mathcal{P}(V_{G_{kT/2^n}})} F^l_{G_{kT/2^n}}(A) \mathds{1}_{E({kT/2^n},T,A,x)} \right ]. \label{propagationformule2}
\end{eqnarray}

Note that $V_{G_{0}} = \{ L_1,...,L_m \}$ and, by definition of $F^l_{\cdot}(\cdot)$ in Section \ref{enlargedasgencodingfct}, $F^l_{G_0}(A) = \mathds{1}_{A=\{ L_1,..., L_l \}}$. This shows that \eqref{entre0et1} and \eqref{propagationformule2} are true for $k=0$. Let us now assume that \eqref{entre0et1} and \eqref{propagationformule2} are true for some $k \in \{0,1,...,2^n-1\}$ and prove them for $k+1$. By the induction hypothesis $\mathbb{E}_m [ \mathds{1}_{E_{n}} \mathds{1}_{E(0,T,\{ L_1,..., L_l \},x)} ]$ equals 
%\begin{align}
%& \mathbb{E}_m \left [ \mathds{1}_{E_{n,k}} \sum_{A \in \mathcal{P}(V_{G_{kT/2^n}})} F^l_{G_{kT/2^n}}(A) \mathds{1}_{E'_{n,k}} \mathds{1}_{E({kT/2^n},T,A,x)} \right ] \nonumber \\
%= & \mathbb{E}_m \left [ \mathds{1}_{E_{n,k}} \sum_{A \in \mathcal{P}(V_{G_{kT/2^n}})} F^l_{G_{kT/2^n}}(A) \mathds{1}_{\{ \text{no transition on } [-(k+1)T/2^n,-kT/2^n] \}} \mathds{1}_{E'_{n,k+1}} \mathds{1}_{E({kT/2^n},T,A,x)} \right ] \label{notrans} \\
%+ & \mathbb{E}_m \left [ \mathds{1}_{E_{n,k}} \sum_{A \in \mathcal{P}(V_{G_{kT/2^n}})} F^l_{G_{kT/2^n}}(A) \mathds{1}_{\{ 1 \text{ coalescence on } [-(k+1)T/2^n,-kT/2^n] \}} \mathds{1}_{E'_{n,k+1}} \mathds{1}_{E({kT/2^n},T,A,x)} \right ] \label{1coales} \\
%+ & \mathbb{E}_m \left [ \mathds{1}_{E_{n,k}} \sum_{A \in \mathcal{P}(V_{G_{kT/2^n}})} F^l_{G_{kT/2^n}}(A) \mathds{1}_{\{ 1 \text{ jump on } [-(k+1)T/2^n,-kT/2^n] \}} \mathds{1}_{E'_{n,k+1}} \mathds{1}_{E({kT/2^n},T,A,x)} \right ] \label{1jump}. 
%\end{align}
\begin{align}
%& \mathbb{E}_m \left [ \mathds{1}_{E_{n}} \sum_{A \in \mathcal{P}(V_{G_{kT/2^n}})} F^l_{G_{kT/2^n}}(A) \mathds{1}_{E({kT/2^n},T,A,x)} \right ] \nonumber \\
\sum_{p=0}^3 \mathbb{E}_m \left [ \mathds{1}_{I_p(kT/2^n,(k+1)T/2^n)} \mathds{1}_{E_{n}} \sum_{A \in \mathcal{P}(V_{G_{kT/2^n}})} F^l_{G_{kT/2^n}}(A) \mathds{1}_{E({kT/2^n},T,A,x)} \right ] =: \sum_{p=0}^3 \mathbb{E}_m [H_p]. \label{notrans} 
%\\
%+ & \mathbb{E}_m \left [ \mathds{1}_{C(kT/2^n,(k+1)T/2^n)} \mathds{1}_{E_{n}} \sum_{A \in \mathcal{P}(V_{G_{kT/2^n}})} F^l_{G_{kT/2^n}}(A) \mathds{1}_{E({kT/2^n},T,A,x)} \right ] \label{1coales} \\
%+ & \mathbb{E}_m \left [ \mathds{1}_{MB(kT/2^n,(k+1)T/2^n) \}} \mathds{1}_{E_{n}} \sum_{A \in \mathcal{P}(V_{G_{kT/2^n}})} F^l_{G_{kT/2^n}}(A) \mathds{1}_{E({kT/2^n},T,A,x)} \right ] \label{1jump} \\
%+ & \mathbb{E}_m \left [ \mathds{1}_{SB(kT/2^n,(k+1)T/2^n) \}} \mathds{1}_{E_{n}} \sum_{A \in \mathcal{P}(V_{G_{kT/2^n}})} F^l_{G_{kT/2^n}}(A) \mathds{1}_{E({kT/2^n},T,A,x)} \right ] \label{1singbran}. 
\end{align}
On the event $I_0(kT/2^n,(k+1)T/2^n)$ we have $G_{kT/2^n} = G_{(k+1)T/2^n}$ and for any set \\ $A \in \mathcal{P}(V_{G_{kT/2^n}}) = \mathcal{P}(V_{G_{(k+1)T/2^n}})$, the lines of $A$ are, after assignment of types for the E-ASG on $[0,T]$ (see Definition \ref{typeaspreasg}), 
%(DIRE QUE LES LIGNES DE $A$, DANS CE CAS, SONT CONSIDEREES COMME REELLES ?) PLUS BESOIN
all of type $0$ at instant $\beta=kT/2^n$ if and only if they are at instant $\beta=(k+1)T/2^n$. Therefore the events $I_0(kT/2^n,(k+1)T/2^n) \cap E(t,T,A,x)$ are equal for $t=kT/2^n$ and $t=(k+1)T/2^n$. Therefore, $\mathbb{P}_m$-almost surely 
\begin{eqnarray}
H_0 = \mathds{1}_{I_0(kT/2^n,(k+1)T/2^n)} \mathds{1}_{E_{n}} \sum_{A \in \mathcal{P}(V_{G_{(k+1)T/2^n}})} F^l_{G_{(k+1)T/2^n}}(A) \mathds{1}_{E({(k+1)T/2^n},T,A,x)}. \label{notrans2}
\end{eqnarray}
By the induction hypothesis, $H_0 \in [0,1]$ so the right-hand side belongs $\mathbb{P}_m$-almost surely to $[0,1]$. 

On the event $I_1(kT/2^n,(k+1)T/2^n)$, $G_{(k+1)T/2^n}$ is obtained by adjunction of a coalescing generation to $G_{kT/2^n}$ (where the coalescence is made uniformly randomly among all pairs). Thus, there is exactly one pair of lines in $V_{G_{kT/2^n}}$ that share the same son in $V_{G_{(k+1)T/2^n}}$. For a set $A \in \mathcal{P}(V_{G_{kT/2^n}})$, recall that $D(A) \in \mathcal{P}(V_{G_{(k+1)T/2^n}})$ is the set of sons in $V_{G_{(k+1)T/2^n}}$ of lines of $A$. For any $A \in \mathcal{P}(V_{G_{kT/2^n}})$, $A$ contains either $0$, $1$, or $2$ of the two lines whose son is common. In either case we have that, after assignment of types for the E-ASG on $[0,T]$ (see Definition \ref{typeaspreasg}), all lines of $A$ are of type $0$ at instant $\beta=kT/2^n$ if and only if all lines of $D(A)$ are of type $0$ at instant $\beta=(k+1)T/2^n$. Therefore, 
\begin{align*}
H_1 = & \mathds{1}_{I_1(kT/2^n,(k+1)T/2^n)} \mathds{1}_{E_{n}} \sum_{A \in \mathcal{P}(V_{G_{kT/2^n}})} F^l_{G_{kT/2^n}}(A) \mathds{1}_{E({(k+1)T/2^n},T,D(A),x)} \\
= & \mathds{1}_{I_1(kT/2^n,(k+1)T/2^n)} \mathds{1}_{E_{n}} \sum_{B \in \mathcal{P}(V_{G_{(k+1)T/2^n}})} \left ( \sum_{A \in \mathcal{P}(V_{G_{kT/2^n}}) ; D(A) = B} F^l_{G_{kT/2^n}}(A) \right )  \mathds{1}_{E({(k+1)T/2^n},T,B,x)}. 
\end{align*}
In the above we have used that each $B \in \mathcal{P}(V_{G_{(k+1)T/2^n}})$ is equal to $D(A)$ for at least one set $A \in \mathcal{P}(V_{G_{kT/2^n}})$, since $D(P(B))=B$. 
We have by definition of $F^l_{\cdot}(\cdot)$ that, on the event $I_1(kT/2^n,(k+1)T/2^n)$, the expression inside the parenthesis equals $F^l_{G_{(k+1)T/2^n}}(B)$. We thus get 
\begin{align}
H_1 = \mathds{1}_{I_1(kT/2^n,(k+1)T/2^n)} \mathds{1}_{E_{n}} \sum_{B \in \mathcal{P}(V_{G_{(k+1)T/2^n}})} F^l_{G_{(k+1)T/2^n}}(B) \mathds{1}_{E({(k+1)T/2^n},T,B,x)}. \label{1coales20}
\end{align}
By the induction hypothesis, $H_1 \in [0,1]$ so the right-hand side belongs $\mathbb{P}_m$-almost surely to $[0,1]$. 
%And taking expectations of both quantities we get 
%\begin{align}
%\mathbb{E}_m [H_1] = \mathbb{E}_m \left [ \mathds{1}_{I_1(kT/2^n,(k+1)T/2^n)} \mathds{1}_{E_{n}} \sum_{B \in \mathcal{P}(V_{G_{(k+1)T/2^n}})} F^l_{G_{(k+1)T/2^n}}(B) \mathds{1}_{E({(k+1)T/2^n},T,B,x)} \right ]. \label{1coales2}
%\end{align}

On the event $I_2(kT/2^n,(k+1)T/2^n)$, $G_{(k+1)T/2^n}$ is obtained by adjunction of a multiple branching generation to $G_{kT/2^n}$. Let $W\in(-1,1)$ denote its weight, as in Section \ref{enlargedasgdefnot}. For a set $A \in \mathcal{P}(V_{G_{kT/2^n}})$, let us denote the lines of $A$ by $L_1^A,..., L_{|A|}^A$ and let $L_{i,1}^A$ and $L_{i,2}^A$ denote the two sons in $V_{G_{(k+1)T/2^n}}$ of the line $L_i^A$, where $L_{i,1}^A$ is the continuing brother and $L_{i,2}^A$ is the incoming brother. 
%(linked to $L_i^A$ by a star-shaped arrow). 
%Note that, on $MB(kT/2^n,(k+1)T/2^n)$, the line $L_i^A$ is subject to a branching on $[-(k+1)T/2^n,-kT/2^n]$, and 
We apply the type assignment procedure to the E-ASG on $[0,T]$ (see Definition \ref{typeaspreasg}) and denote by $l_i^A$ the label of the branching that occurs on the line $L_i^A$. 
%Note that if $l_{i}^A = \text{ real}$ then both $L_{i,1}^A$ and $L_{i,2}^A$ are real, and if $l_{i}^A = \text{ virtual}$ then $L_{i,1}^A$ is real while $L_{i,2}^A$ is virtual. 
By the type propagation rules: If $l_i^A =$ virtual then $L_i^A$ has the same type as $L_{i,1}^A$. If $l_i^A =$ real and $W > 0$, then $L_i^A$ has type $0$ if and only if either $L_{i,2}^A$ has type $0$ or $L_{i,2}^A$ and $L_{i,1}^A$ have respectively type $1$ and $0$. If $l_i^A =$ real and $W < 0$, then $L_i^A$ has type $0$ if and only if both $L_{i,1}^A$ and $L_{i,2}^A$ have type $0$. We thus get that $\mathds{1}_{E({kT/2^n},T, \{ L_i^A \},x)}$ equals 
\begin{align*}
& \mathds{1}_{W > 0} \left ( \mathds{1}_{\{ l_{i}^A = \text{ virtual} \}} \mathds{1}_{E({(k+1)T/2^n},T, \{ L_{i,1}^A \},x)} \right . \\
+ & \left . \mathds{1}_{\{ l_{i}^A = \text{ real} \}} (\mathds{1}_{E({(k+1)T/2^n},T, \{ L_{i,2}^A \},x)} + (1-\mathds{1}_{E({(k+1)T/2^n},T, \{ L_{i,2}^A \},x)}) \mathds{1}_{E({(k+1)T/2^n},T, \{ L_{i,1}^A \},x)}) \right ) \\
+ & \mathds{1}_{W < 0} \left ( \mathds{1}_{\{ l_{i}^A = \text{ virtual} \}} \mathds{1}_{E({(k+1)T/2^n},T, \{ L_{i,1}^A \},x)} + \mathds{1}_{\{ l_{i}^A = \text{ real} \}} \mathds{1}_{E({(k+1)T/2^n},T, \{ L_{i,1}^A \},x)} \mathds{1}_{E({(k+1)T/2^n},T, \{ L_{i,2}^A \},x)} \right ) \\ 
= & \left (1 - \mathds{1}_{W < 0} \mathds{1}_{\{ l_{i}^A = \text{ real} \}} \right ) \mathds{1}_{E({(k+1)T/2^n},T, \{ L_{i,1}^A \},x)} + \mathds{1}_{W > 0} \mathds{1}_{\{ l_{i}^A = \text{ real} \}} \mathds{1}_{E({(k+1)T/2^n},T, \{ L_{i,2}^A \},x)} \\ 
+ & \mathds{1}_{\{ l_{i}^A = \text{ real} \}} \left ( \mathds{1}_{W < 0} - \mathds{1}_{W > 0} \right ) \mathds{1}_{E({(k+1)T/2^n},T, \{ L_{i,1}^A \},x)} \mathds{1}_{E({(k+1)T/2^n},T, \{ L_{i,2}^A \},x)}. 
\end{align*}

For $B \in \mathcal{P}(V_{G_{(k+1)T/2^n}})$, recall that $P(B) \in \mathcal{P}(V_{G_{kT/2^n}})$ is the set of parents in $V_{G_{kT/2^n}}$ of lines of $B$. For our fixed $A \in \mathcal{P}(V_{G_{kT/2^n}})$, let $\mathcal{D}(A)$ denote the family of sets $B \in \mathcal{P}(V_{G_{(k+1)T/2^n}})$ that are such that $P(B) = A$. In other words, $B \in \mathcal{D}(A)$ if and only if each line of $A$ has at least one son in $B$ (possibly two) and each line of $B$ has its parent in $A$. Note that any $B \in \mathcal{D}(A)$ is included in $D(A)$. Since $\mathds{1}_{E({kT/2^n},T,A,x)} = \prod_{i=1}^{|A|} \mathds{1}_{E({kT/2^n},T, \{ L_i^A \},x)}$, we deduce from above that 
\begin{align} 
\mathds{1}_{E({kT/2^n},T,A,x)} & = \sum_{B \in \mathcal{D}(A)} \prod_{j=1}^{|B|} W(L_j^B) \mathds{1}_{E({(k+1)T/2^n},T, \{ L_{j}^B \},x)} = \sum_{B \in \mathcal{D}(A)} \mathds{1}_{E({(k+1)T/2^n},T,B,x)} \prod_{j=1}^{|B|} W(L_j^B), \label{devform}
\end{align}
where we denoted by $L_1^B,..., L_{|B|}^B$ the lines of any set $B \in \mathcal{D}(A)$ and where 
\[ W(L_j^B) := (1 - \mathds{1}_{W < 0} \mathds{1}_{\{ \text{branching from which} \ L_{j}^B \ \text{is issued has label real} \}} ) \] if $L_{j}^B$ is a continuing line whose associated incoming brother is in $B^c$, \[ W(L_j^B) := \mathds{1}_{W > 0} \mathds{1}_{\{ \text{branching from which} \ L_{j}^B \ \text{is issued has label real} \}} \] if $L_{j}^B$ is an incoming line whose associated continuing brother is in $B^c$,  \[ W(L_j^B) := \mathds{1}_{\{ \text{branching from which} \ L_{j}^B \ \text{is issued has label real} \}} ( \mathds{1}_{W < 0} - \mathds{1}_{W > 0} ) \] if $L_{j}^B$ is an incoming line whose associated continuing brother is in $B$, and $W(L_j^B) = 1$ if $L_{j}^B$ is a continuing line whose associated incoming brother is in $B$. Then, replacing $\mathds{1}_{E({kT/2^n},T,A,x)}$ by \eqref{devform} into the expression of $H_2$ (see \eqref{notrans}) we get 
\begin{align*}
H_2 = \mathds{1}_{I_2(kT/2^n,(k+1)T/2^n)} \mathds{1}_{E_{n}} \sum_{A \in \mathcal{P}(V_{G_{kT/2^n}})} \sum_{B \in \mathcal{D}(A)} F^l_{G_{kT/2^n}}(A) \left ( \prod_{j=1}^{|B|} W(L_j^B) \right ) \mathds{1}_{E({(k+1)T/2^n},T,B,x)}. 
\end{align*}
%By the induction hypothesis, $H_2 \in [0,1]$ so the right-hand side belongs $\mathbb{P}_m$-almost surely to $[0,1]$. 
$(G_{\beta})_{\beta \in [0,T]}$ being given, note from Definition \ref{typeaspreasg} that the types of lines in $V_{G_{t}}$ (for any $t \in [0,T]$) is a deterministic function of $(G_{\beta})_{\beta \in [0,T]}$, the labels on $[t,T]$, and the types of lines in $V_{G_{T}}$. 
%Recall that to build the process we are considering we first choose the random graph valued process $(G_t)_{t \in [0,T]}$, then we assign labels to the branchings and \textit{iid} types to the lines in $V_{G_{T}}$. 
%%, and then we decide independently for each splitting if the incoming line is real (with a probability equal to the size of the jump). 
%Then, the type of the lines of the graph is a deterministic function of these three processes. 
Also, recall that for any $t \in [0,T]$, $F^l_{G_{t}}(\cdot)$ is a deterministic function of $(G_{\beta})_{\beta \in [0,T]}$ and that it does not depend on the labels. 
Let us define the sigma-field 
\[ \mathcal{G}_k := \sigma ((G_{\beta})_{\beta \in [0,T]}, \ \text{labels on} \ [(k+1)T/2^n,T], \ \text{types of lines in} \ V_{G_{T}}). \]
Note that $\mathcal{G}_k$ contains all the information on the weights of the transitions of the E-ASG on $[0,T]$. 
%which contains all the information in $(G_t)_{t \in [0,T]}$, the assignment of types and the process that decides which lines are real on $[-T,-(k+1)T/2^n]$. 
Then, $\mathbb{E}_m[H_2]=\mathbb{E}_m[\mathbb{E}_m[H_2 | \mathcal{G}_k]]$. In the above expression, only the factors $W(L_j^B)$ are not measurable with respect to $\mathcal{G}_k$ so we get that $\mathbb{E}_m [H_2]$ equals 
\begin{align}
& \mathbb{E}_m \left [ \mathds{1}_{I_2(kT/2^n,(k+1)T/2^n)} \mathds{1}_{E_{n}}\sum_{A \in \mathcal{P}(V_{G_{kT/2^n}})} \sum_{B \in \mathcal{D}(A)} F^l_{G_{kT/2^n}}(A) \mathbb{E}_m \left [ \prod_{j=1}^{|B|} W(L_j^B) \big | \mathcal{G}_k \right ] \mathds{1}_{E({(k+1)T/2^n},T,B,x)} \right ]. \label{1jump0} 
\end{align}
Moreover, what is inside the above expectation $\mathbb{P}_m$-almost surely belongs to $[0,1]$ since it equals $\mathbb{E}_m[H_2 | \mathcal{G}_k]$ and, $\mathbb{P}_m$-almost surely, $H_2 \in [0,1]$. From Definition \ref{typeaspreasg}, the labels of the branchings occurring in $(kT/2^n,(k+1)T/2^n]$ are \textit{iid}, real with probability $|W|$ and virtual with probability $1-|W|$. Therefore, for any $B \in \mathcal{D}(A)$ we have 
\[ F^l_{G_{kT/2^n}}(A) \times \mathbb{E}_m \left [ \prod_{j=1}^{|B|} W(L_j^B) \big | \mathcal{G}_k \right ] = F^l_{G_{kT/2^n}}(A) \times (1+W \mathds{1}_{W < 0})^{\alpha(B)} \times (W \mathds{1}_{W > 0})^{\beta(B)} \times (-W)^{\gamma(B)}, \]
where $\alpha(B)$, $\beta(B)$ and $\gamma(B)$ are defined in Section \ref{enlargedasgencodingfct}. 
%$\alpha(B)$ is the number of continuing lines in $B$ whose associated incoming brothers are in $B^c$, $\beta(B)$ is the number of incoming lines in $B$ whose associated continuing brothers are in $B^c$, and $\gamma(B)$ is the number of pairs of brothers that are both in $B$. 
Recall that for any $B \in \mathcal{D}(A)$ we have $P(B)=A$. Therefore, the above equals
\[ F^l_{G_{kT/2^n}}(P(B)) \times (1+W \mathds{1}_{W < 0})^{\alpha(B)} \times (W \mathds{1}_{W > 0})^{\beta(B)} \times (-W)^{\gamma(B)}. \]
Using the definition of $F^l_{\cdot}(\cdot)$ in Section \ref{enlargedasgencodingfct}, we can identify the above with $F^l_{G_{(k+1)T/2^n}}(B)$. Plugging into \eqref{1jump0} we deduce 
\begin{align}
\mathbb{E}_m [H_2] = & \mathbb{E}_m \left [ \mathds{1}_{I_2(kT/2^n,(k+1)T/2^n)} \mathds{1}_{E_{n}} \sum_{A \in \mathcal{P}(V_{G_{kT/2^n}})} \sum_{B \in \mathcal{D}(A)} F^l_{G_{(k+1)T/2^n}}(B) \mathds{1}_{E({(k+1)T/2^n},T,B,x)} \right ] \label{1jump1} \\ 
= & \mathbb{E}_m \left [ \mathds{1}_{I_2(kT/2^n,(k+1)T/2^n)} \mathds{1}_{E_{n}} \sum_{B \in \mathcal{P}(V_{G_{(k+1)T/2^n}})} F^l_{G_{(k+1)T/2^n}}(B) \mathds{1}_{E({(k+1)T/2^n},T,B,x)} \right ]. \label{1jump2}
\end{align}
For the last equality we used that 
%since $V_{G_{(k+1)T/2^n}}$ is separated from $V_{G_{kT/2^n}}$ by a branching generation, 
for each $B \in \mathcal{P}(V_{G_{(k+1)T/2^n}})$ there is exactly one $A \in \mathcal{P}(V_{G_{kT/2^n}})$ such that $B \in \mathcal{D}(A)$ (this $A$ is precisely $P(B)$). Therefore, for each $B \in \mathcal{P}(V_{G_{(k+1)T/2^n}})$, $F^l_{G_{(k+1)T/2^n}}(B)$ appears exactly once in the sum in the expression \eqref{1jump1}. Moreover, since the quantities inside the expectations \eqref{1jump0}, \eqref{1jump1} and \eqref{1jump2} are actually $\mathbb{P}_m$-almost surely equal, they belong $\mathbb{P}_m$-almost surely to $[0,1]$. 

On the event $I_3(kT/2^n,(k+1)T/2^n)$, $G_{(k+1)T/2^n}$ is obtained by adjunction of a single branching generation to $G_{kT/2^n}$. Thus, there is exactly one line in $V_{G_{kT/2^n}}$ that has two sons in $V_{G_{(k+1)T/2^n}}$. Let us denote this line by $L_0$ and by $L_{0,1}$ and $L_{0,2}$ the two sons in $V_{G_{(k+1)T/2^n}}$ of $L_0$, where $L_{0,1}$ is the continuing brother and $L_{0,2}$ is the incoming brother. For a set $A \in \mathcal{P}(V_{G_{kT/2^n}})$, let us denote the lines of $A$ (others than $L_0$) by $L_1^A, L_{2}^A,...$ and let $L_{i,1}^A$ denote the son in $V_{G_{(k+1)T/2^n}}$ of the line $L_i^A$. Note that after assignment of types for the E-ASG on $[0,T]$ (see Definition \ref{typeaspreasg}), we have 
%\[ \mathds{1}_{E({kT/2^n},T, \{ L_0 \},x)} = \mathds{1}_{E({(k+1)T/2^n},T, \{ L_{0,2} \},x)} + (1-\mathds{1}_{E({(k+1)T/2^n},T, \{ L_{0,2} \},x)}) \mathds{1}_{E({(k+1)T/2^n},T, \{ L_{0,1} \},x)}) \]
$\mathds{1}_{E({kT/2^n},T, \{ L_0 \},x)} = \mathds{1}_{E({(k+1)T/2^n},T, \{ L_{0,1},L_{0,2} \},x)}$. 
Therefore $\mathds{1}_{E({kT/2^n},T,A,x)}$ equals 
%\begin{align*}
%& \mathds{1}_{L_0 \notin A} \prod_{i=1}^{|A|} \mathds{1}_{E({kT/2^n},T, \{ L_i^A \},x)} + \mathds{1}_{L_0 \in A} \mathds{1}_{E({kT/2^n},T, \{ L_0 \},x)} \prod_{i=1}^{|A|} \mathds{1}_{E({kT/2^n},T, \{ L_i^A \},x)} \\
%= & \mathds{1}_{L_0 \notin A} \prod_{i=1}^{|A|} \mathds{1}_{E({(k+1)T/2^n},T, \{ L_{i,1}^A \},x)} + \sum_{j=0}^1 \mathds{1}_{L_0 \in A} \mathds{1}_{E({(k+1)T/2^n},T, \{ L_{0,j} \},x)} \prod_{i=1}^{|A|} \mathds{1}_{E({(k+1)T/2^n},T, \{ L_{i,1}^A \},x)} \\
%- & \mathds{1}_{L_0 \in A} \mathds{1}_{E({(k+1)T/2^n},T, \{ L_{0,1},L_{0,2} \},x)} \prod_{i=1}^{|A|} \mathds{1}_{E({(k+1)T/2^n},T, \{ L_i^A \},x)} \\
%= & \sum_{B \in \mathcal{D}(A)} (\mathds{1}_{N(B) \in \{0,1\}} - \mathds{1}_{N(B) =2}) \mathds{1}_{E({(k+1)T/2^n},T, B,x)}. 
%\end{align*}
\begin{align*}
& \mathds{1}_{L_0 \notin A} \prod_{i=1}^{|A|} \mathds{1}_{E({kT/2^n},T, \{ L_i^A \},x)} + \mathds{1}_{L_0 \in A} \mathds{1}_{E({kT/2^n},T, \{ L_0 \},x)} \prod_{i=1}^{|A|-1} \mathds{1}_{E({kT/2^n},T, \{ L_i^A \},x)} \\
= & \mathds{1}_{L_0 \notin A} \prod_{i=1}^{|A|} \mathds{1}_{E({(k+1)T/2^n},T, \{ L_{i,1}^A \},x)} + \mathds{1}_{L_0 \in A} \mathds{1}_{E({(k+1)T/2^n},T, \{ L_{0,1},L_{0,2} \},x)} \prod_{i=1}^{|A|-1} \mathds{1}_{E({(k+1)T/2^n},T, \{ L_{i,1}^A \},x)} \\
= & \sum_{B \in \mathcal{D}(A)} \mathds{1}_{N(B) \in \{0,2\}} \mathds{1}_{E({(k+1)T/2^n},T, B,x)}, 
\end{align*}
where $N(B)$ is defined in Section \ref{enlargedasgencodingfct}. Recall that for any $B \in \mathcal{D}(A)$ we have $P(B)=A$. Therefore $F^l_{G_{kT/2^n}}(A) \mathds{1}_{E({kT/2^n},T,A,x)}$ equals 
\begin{align*}
\sum_{B \in \mathcal{D}(A)} F^l_{G_{kT/2^n}}(P(B)) \mathds{1}_{N(B) \in \{0,2\}} \mathds{1}_{E({(k+1)T/2^n},T, B,x)} = \sum_{B \in \mathcal{D}(A)} F^l_{G_{(k+1)T/2^n}}(B) \mathds{1}_{E({(k+1)T/2^n},T, B,x)}, 
\end{align*}
where the last equality comes from the definition of $F^l_{\cdot}(\cdot)$ in Section \ref{enlargedasgencodingfct}. Proceeding as in \eqref{1jump1}-\eqref{1jump2} we can conclude 
\begin{align}
H_3 = \mathds{1}_{I_3(kT/2^n,(k+1)T/2^n)} \mathds{1}_{E_{n}} \sum_{B \in \mathcal{P}(V_{G_{(k+1)T/2^n}})} F^l_{G_{(k+1)T/2^n}}(B) \mathds{1}_{E({(k+1)T/2^n},T,B,x)}. \label{1singbranch2}
\end{align}
By the induction hypothesis, $H_3 \in [0,1]$ so the right-hand side belongs $\mathbb{P}_m$-almost surely to $[0,1]$. 

Plugging \eqref{notrans2}, \eqref{1coales20}, \eqref{1jump2} and \eqref{1singbranch2} into \eqref{notrans} we get that 
%\begin{align*}
%& \mathbb{E}_m \left [ \mathds{1}_{E_{n}} \sum_{A \in \mathcal{P}(V_{G_{kT/2^n}})} F^l_{G_{kT/2^n}}(A) \mathds{1}_{E({kT/2^n},T,A,x)} \right ] \\
%= & \mathbb{E}_m \left [ \mathds{1}_{E_{n}} \sum_{A \in \mathcal{P}(V_{G_{(k+1)T/2^n}})} F^l_{G_{(k+1)T/2^n}}(A) \mathds{1}_{E({(k+1)T/2^n},T,A,x)} \right ], 
%\end{align*}
\eqref{propagationformule2} holds for $k+1$. Moreover, the terms in \eqref{notrans2}, \eqref{1coales20}, \eqref{1singbranch2}, as well as what are inside the expectations in \eqref{1jump2}, are $\mathbb{P}_m$-almost surely in $[0,1]$, and since at most one indicator $\mathds{1}_{I_p(kT/2^n,(k+1)T/2^n)}$ is not $0$, \eqref{entre0et1} holds for $k+1$. The induction is thus proved. Recall $(\mathcal{F}_t)_{t \geq 0}$ from Section \ref{enlargedasgdefnot}. We evaluate \eqref{propagationformule2} at $k=2^n$ and get that $\mathbb{E}_m [ \mathds{1}_{E_{n}} \mathds{1}_{E(0,T,\{ L_1,..., L_l \},x)} ]$ equals 
%\label{egsuren0}
\begin{align}
& \mathbb{E}_m \left [ \mathds{1}_{E_{n}} \sum_{A \in \mathcal{P}(V_{G_{T}})} F^l_{G_{T}}(A) \mathds{1}_{E({T},T,A,x)} \right ] = \mathbb{E}_m \left [ \mathbb{E}_m \left [ \mathds{1}_{E_{n}} \sum_{A \in \mathcal{P}(V_{G_{T}})} F^l_{G_{T}}(A) \mathds{1}_{E({T},T,A,x)} \big | \mathcal{F}_T \right ] \right ] \nonumber \\
= & \mathbb{E}_m \left [ \mathds{1}_{E_{n}} \sum_{A \in \mathcal{P}(V_{G_{T}})} F^l_{G_{T}}(A) \mathbb{E}_m \left [ \mathds{1}_{E({T},T,A,x)} \big | \mathcal{F}_T \right ] \right ] = \mathbb{E}_m \left [ \mathds{1}_{E_{n}} \sum_{A \in \mathcal{P}(V_{G_{T}})} F^l_{G_{T}}(A) x^{|A|} \right ]. \label{egsuren}
\end{align}
Moreover, by \eqref{entre0et1} at $k=2^n$ we get that all the terms that are inside the above expectations $\mathbb{P}_m$-almost surely belong to $[0,1]$ which shows that 
\begin{eqnarray}
\mathds{1}_{E_{n}} \sum_{A \in \mathcal{P}(V_{G_{T}})} F^l_{G_{T}}(A) \mathds{1}_{E(T,T,A,x)} \in [0,1], \ \ \ \mathds{1}_{E_{n}} \sum_{A \in \mathcal{P}(V_{G_{T}})} F^l_{G_{T}}(A) x^{|A|} \in [0,1]. \label{inegsuren}
\end{eqnarray}
Combining \eqref{egsuren} with \eqref{limitthejumps} we get 
\[ h^l_T(x)= \lim_{n \rightarrow \infty} \mathbb{E}_m \left [ \mathds{1}_{E_{n}} \sum_{A \in \mathcal{P}(V_{G_{T}})} F^l_{G_{T}}(A) \mathds{1}_{E({T},T,A,x)} \right ] = \lim_{n \rightarrow \infty} \mathbb{E}_m \left [ \mathds{1}_{E_{n}} \sum_{A \in \mathcal{P}(V_{G_T})} F^l_{G_T}(A) x^{|A|} \right ]. \]
%Clearly $\mathbb{P}_m(\cup_{n \geq 1} E_n) = 1$ so 
Recall that $\mathds{1}_{E_{n}}$ $\mathbb{P}_m$-almost surely increases to $1$ as $n$ goes to infinity, and thanks to \eqref{inegsuren} we can apply dominated convergence which yields 
%since $\mathbb{P}_m$-almost surely \\ $\mathds{1}_{E_{n}} \sum_{A \in \mathcal{P}(V_{G_{T}})} F^l_{G_{T}}(A) \mathds{1}_{E({T},T,A,x)} \in [0,1]$ and $\mathds{1}_{E_{n}} \sum_{A \in \mathcal{P}(V_{G_T})} F^l_{G_T}(A) x^{|A|} \in [0,1]$, 
\eqref{propagationformule1mix}. Combining \eqref{inegsuren} and the fact that $\mathds{1}_{E_{n}}$ increases to $1$ we get \eqref{entre0et10mix}. Finally, by choosing $x=1$ in \eqref{entre0et10mix} and \eqref{propagationformule1mix} we see that $\sum_{A \in \mathcal{P}(V_{G_{T}})} F^l_{G_{T}}(A)$ $\mathbb{P}_m$-almost surely belongs to $[0,1]$ and has expectation $1$, which yields \eqref{sommetotale=1}.

\subsection{An important estimate} \label{sectionlemmes}

The following lemma provides an estimate, based on \eqref{entre0et10mix} and on a combinatorial argument, that allows to bound deterministically the absolute value of the coefficients $F^l_{G_T}(B)$. This deterministic bound and its consequences (like Lemma \ref{sgwelldef}) will be used extensively all along the proofs in the paper. 

\begin{lemme} \label{lemmecomb}

Let us fix $T>0$ and $m \geq 1$. For any $l \in \{1,...,m\}$ and $B \in \mathcal{P}(V_{G_T})$ we have $\mathbb{P}_m$-almost surely 
\begin{eqnarray}
|F^l_{G_T}(B)| \leq |B|^{|B|}. \label{estimcomb}
\end{eqnarray}
Recall the measure $\pi$ from Section \ref{enlargedasglcp}. For all $M \geq 1$, $m \geq 1$, $k \geq 1$, and $T \geq 0$, 
\begin{eqnarray}
\mathbb{E}_m \left [ |V_{G_{T}}|^{k} \mathds{1}_{|V_{G_{T}}| \geq M} \right ] \leq \frac{1}{\pi(m)} \sum_{j \geq M} j^k \pi(j) < \infty. \label{bornemomentGtnew}
\end{eqnarray}

\end{lemme}

\begin{proof}

Let us fix a realization $G_T$ of the E-ASG on $[0,T]$. We apply the type assignment procedure on $[0,T]$ with initial condition, say, $1/2$ (see Definition \ref{typeaspreasg}). For $B \in \mathcal{P}(V_{G_T})$, under $\mathbb{P}_m$ and conditionally on the fixed realization $G_T$ of the E-ASG, the event where only the lines that belong to $B$ are of type $0$ at instant $\beta=T$ while all other lines in $V_{G_T}$ are of type $1$ has positive probability. On this event we have 
\[ \sum_{A \in \mathcal{P}(V_{G_{T}})} F^l_{G_{T}}(A) \mathds{1}_{E(T,T,A,x)} = \sum_{A \subseteq B} F^l_{G_{T}}(A), \]
%Indeed, $\mathds{1}_{E(T,T,A,x)}$ is equal to $1$ when $A \subseteq B$ and to $0$ for all other sets $A \in \mathcal{P}(V_{G_T})$. 
where $E(T,T,A,x)$ is as in Definition \ref{type0events}. Using \eqref{entre0et10mix} we deduce 
\begin{align}
\sum_{A \subseteq B} F^l_{G_{T}}(A) \in [0,1]. \label{lemmecomb1}
\end{align}

Note that \eqref{lemmecomb1} holds for any $B \in \mathcal{P}(V_{G_T})$. Let us now prove \eqref{estimcomb} by induction on the cardinality of $B \in \mathcal{P}(V_{G_T})$. If $|B|=1$, \eqref{lemmecomb1} yields $F^l_{G_{T}}(B) \in [0,1]$ so $B$ satisfies \eqref{estimcomb}. Now assume that for some $k \geq 1$, \eqref{estimcomb} is satisfied for all $B \in \mathcal{P}(V_{G_T})$ with cardinality $|B| \leq k$ and let us prove it for a $B \in \mathcal{P}(V_{G_T})$ with cardinality $|B|=k+1$. From \eqref{lemmecomb1} we get 
\[ |F^l_{G_{T}}(B)| \leq 1 + \sum_{A \subsetneq B} |F^l_{G_{T}}(A)| = 1 + \sum_{i=1}^{k} \sum_{A \subseteq B ; |A|=i} |F^l_{G_{T}}(A)|. \]
Using the induction hypothesis for sets $A \subseteq B$ with cardinality $i \leq k$ and using that there are exactly $\binom{k+1}{i}$ subsets $A \subseteq B$ with cardinality $i$, we get 
\[ |F^l_{G_{T}}(B)| \leq 1 + \sum_{i=1}^{k} \binom{k+1}{i} i^i \leq \sum_{i=0}^{k+1} \binom{k+1}{i} k^i = (k+1)^{k+1}. \]
Therefore, \eqref{estimcomb} holds for $B$ which proves the induction. Finally, \eqref{bornemomentGtnew} is an easy consequence of Lemma \ref{lemmemajointemporelle} and Proposition \ref{recpilambda}. 

\end{proof}

\section{Average coefficients of the encoding function and representation of $h(x)$} \label{acotcfnew}

\subsection{Well-definedness and some inequalities} \label{dawd}

The well-definedness of the coefficients $R^{m,k}_t(i,j)$ and $Q_t(i,j)$ from Definition \ref{defcoefsg} is guaranteed by the following lemma. Its last inequality will be useful in the rest of the paper to control coefficients $R^{m,k}_t(i,j)$. 

\begin{lemme} \label{sgwelldef}

For any $m \geq 1, i \in \{1,...,m\}, j \geq 1$, $t \geq 0$, we have $\mathbb{P}_m$-almost surely: 
\begin{align}
\sum_{A \in \mathcal{P}(V_{G_{t}}) ; |A|=j} |F^i_{G_{t}}(A)| & \leq \frac{j^j}{j!} |V_{G_{t}}|^{j}, \label{borneFps} \\
\mathbb{E}_m \left [ \sum_{A \in \mathcal{P}(V_{G_{t}}) ; |A|=j} |F^i_{G_{t}}(A)| \right ] & \leq \frac{j^j}{\pi(m) j!} \sum_{l \geq 1} l^j \pi(l) < \infty. \label{borneFesp} 
\end{align}
In particular the coefficients $R^{m,k}_t(i,j)$ and $Q_t(i,j)$ are well-defined and we have 
\begin{align}
|R^{m,k}_t(i,j)| & \leq \frac{(jk)^j}{j!} \mathbb{P}_m \left ( |V_{G_{t}}| = k \right ). \label{majonouveausgnew}
\end{align}

\end{lemme}

%IL SEMBLE QU'AVEC LA NOUVELLE DECOMPOSITION ET PEUT SE PASSER DE \eqref{borneFesp}, ET DU COUP PEUT-ETRE AUSSI DE Lemma \ref{bornemomentGt}, VOIR. ON A BESOIN DU Lemma \ref{bornemomentGt} PARCEQUE PARFOIS ON CONTROLE LE NB DE LIGNES EN $t+h$ MAIS PAS EN $t$. DU COUP GARDER Lemma \ref{bornemomentGt} MAIS VOIR QUELLES APPLICATIONS PEUVENT ETRE SIMPLIFIES(AUCUNE), ET VOIR SI ON PEUT SE PASSER DE \eqref{borneFesp}. OUI, ON PEUT SE PASSER DE \eqref{borneFesp}, DU COUP ON L'A MIS EN COMMENTAIRE (PAREIL POUR LA PREUVE) ET REMPLACE LA FOIS OU ON L'APPLIQUE. 

\begin{proof}

Using \eqref{estimcomb} and that there are exactly $\binom{|V_{G_{t}}|}{j} \leq |V_{G_{t}}|^{j}/j!$ sets $A \in \mathcal{P}(V_{G_{t}})$ with cardinality $j$ we get 
%\[ \sum_{A \in \mathcal{P}(V_{G_{t}}) ; |A|=j} |F^i_{G_{t}}(A)| \leq j^j \binom{|V_{G_{t}}|}{j} \leq \frac{j^j}{j!} |V_{G_{t}}|^{j}, \]
\eqref{borneFps}. Then, taking expectation and using \eqref{bornemomentGtnew} we obtain \eqref{borneFesp}. \eqref{majonouveausgnew} is a simple consequence of the definition of $R^{m,k}_t(i,j)$ in Definition \ref{defcoefsg} and of \eqref{borneFps}. 

\end{proof}

\subsection{Renewal structure, semigroup property, proof of Proposition \ref{semigroupprop}} \label{behaviour02}

\subsubsection{Some definitions: an extension of the encoding function} \label{behaviour021}

In order to establish a renewal property for $(F^{i}_{G_{\beta}}(\cdot))_{\beta \geq 0}$ we first define an extension of the function $F^l_{G}$. Let $G, \tilde G \in \mathbb{G}_m$ be such that $\pi_{\mathsf{depth}(G)}(\tilde G)=G$. In other words, $G$ is the restriction of $\tilde G$ to its first $\mathsf{depth}(G)+1$ generations (including generation $0$). Let $B \in \mathcal{P}(V_{G})$. We define the \textit{shifted graph} $\tilde G \setminus_B G \in \mathbb{G}_{|V_{G}|}$ as follows. First let $\{ L_1,...,L_{|V_{G}|} \}$ be the lines of $V_{G}$, ordered in such a way so that the first $|B|$ lines are the lines of $B$ (i.e. $\{ L_1,...,L_{|B|} \} = B$). Then we let $\tilde G \setminus_B G$ be the graph in $\mathbb{G}_{|V_{G}|}$ obtained from removing generations $0,1,...,\mathsf{depth}(G)-1$ of $\tilde G$. In other words, for any $i \geq 0$, generation $i$ of $\tilde G \setminus_B G$ is generation $\mathsf{depth}(G) + i$ of $\tilde G$. In particular generation $0$ of $\tilde G \setminus_B G$ is $\{ L_1,...,L_{|V_{G}|} \}$. Now, for $A \in \mathcal{P}(V_{\tilde G})$ let us define 
%$A \in D_{s}(B) \subset \mathcal{P}(V_{G_{s}})$
\begin{align}
f_{G, \tilde G}(B,A) := F_{\tilde G \setminus_B G}^{|B|}(A). \label{defpetitfg}
\end{align}
Note that there are several possible choices for $\tilde G \setminus_B G$ (as many as the number of orderings of $V_{G}$ such that the first $|B|$ lines are the lines of $B$) but that all these choices yields the same value for $f_{G, \tilde G}(B,A)$. Indeed, $F_{\tilde G \setminus_B G}^{|B|}(\cdot)$ depends on the ordering of lines of generation $0$ of $\tilde G \setminus_B G$ only via the choice of which are the first $|B|$ lines. $f_{G, \tilde G}(B,A)$ is thus uniquely defined. 

\begin{lemme} \label{lemmebehaviour021}
For any $t \geq 0$ and $B \in \mathcal{P}(V_{G_{t}})$, the process $(f_{G_t,G_{t+r}}(B,\cdot))_{r \geq 0}$ is well-defined. Moreover, for all $0 \leq r_2 \leq r_1$, $B \in \mathcal{P}(V_{G_{r_2}})$ and $A \in \mathcal{P}(V_{G_{r_1}})$, we have $\mathbb{P}_m$-a.s. 
\begin{align}
f_{G_{r_2}, G_{r_1}}(B,A) \leq |A|^{|A|}. \label{markpropf3}
\end{align}
\end{lemme}

\begin{proof}
For any $u \geq t \geq 0$ we have $\pi_{\mathsf{depth}(G_t)}(G_u)=G_t$. For $t \geq 0$ and $B \in \mathcal{P}(V_{G_{t}})$ we can thus define the $\mathbb{G}_{|V_{G_t}|}$-valued process $(G_{t+r} \setminus_B G_t)_{r \geq 0}$ that can be considered as the E-ASG shifted at time $t$ and $(f_{G_t,G_{t+r}}(B,\cdot))_{r \geq 0}$ is well-defined. \eqref{markpropf3} is a consequence of \eqref{defpetitfg} and \eqref{estimcomb}. 
\end{proof}
%More precisely, let $\{ L_1^t,...,L_{|V_{G_{t}}|}^t \}$ be the $|V_{G_{t}}|$ lines of $V_{G_{t}}$, ordered in such a way that the first $|B|$ lines are the lines of $B$ (i.e. $\{ L_1^t,...,L_{|B|}^t \} = B$). Let $N_t := \mathsf{depth}(G_t)$, where $\mathsf{depth}(\cdot)$ is defined in Subsection \ref{enlargedasg}. Then for any $s\geq t$ let $G^{t,B}_s$ be the graph in $\mathbb{G}_{|V_{G_{t}}|}$ obtained from taking $G_s$ and removing generations $0,1,...,N_t-1$. In other words, for any $i \geq 0$, generation $i$ of $G^{t,B}_s$ is generation $N_t + i$ of $G_s$, in particular generation $0$ of $G^{t,B}_s$ is $\{ L_1^t,...,L_{|V_{G_{t}}|}^t \}$. 
%Now, for $A \in \mathcal{P}(V_{G_{s}})$ let us define 
%%$A \in D_{s}(B) \subset \mathcal{P}(V_{G_{s}})$
%\begin{align}
%f_{t, s}(B,A) := F_{G^{t,B}_s}^{|B|}(A). \label{defpetitfgold}
%\end{align}
%Note that there are several possible choices for $(G^{t,B}_s)_{s \geq t}$ (as many as the number of orderings of $V_{G_{t}}$ such that the first $|B|$ lines are the lines of $B$) but that all these choices yields the same value for $f_{t, s}(B,A)$. Indeed, $F_{G^{t,B}_s}^{|B|}(\cdot)$ depends on the ordering of lines of generation $0$ of $G^{t,B}_s$ only via the choice of which are the first $|B|$ lines. $f_{t, s}(B,A)$ is thus uniquely defined. The fact that $f_{t, s}(B,A)$ is defined from a fixed realization $(G_r)_{r \geq 0}$ of the E-ASG is implicit in the notations. 

%SI BESOIN, DIRE QUE $f_{t, s}(B,A)$ EST DIFFERENT DE $0$ SEULEMENT SI $A$ DESCEND DE $B$. 
Let $m \geq 1, i \in \{1,...,m\}$, and $L_1,..., L_i$ denote the first $i$ lines of $G_0$. We note that we have $\mathbb{P}_m$-a.s. $f_{G_0, G_t}(\{ L_1,...,L_i \},\cdot) = F^{i}_{G_t}(\cdot)$. 

%\textbf{Id\'ee : $f_{t, s}(B,\cdot )$ est d\'efnie comme $F^i_{G_s}(\cdot)$ sauf qu'elle commence au temps $t$ (au lieu de $0$) avec initialement les lignes de $B$ dans le graphe (au lieu des lignes de $\{ L_1,...,L_i \}$). 
%%et on ignore les coalescences avec les lignes en dehors de $B$. 
%Le mieux est peut-\^etre de d\'efinir d'abord $f_{t, s}(B,\cdot )$ plus haut puis de d\'efinir $F^i_{G_t}(\cdot)$ comme le cas particulier o\`u $t=0$ et $B = \{ L_1,...,L_i\}$. 
%%Attention pour \eqref{petitfgliengdg} parce que \c ca veut dire qu'il faut prendre les premi\`eres lignes comme \'etant celles de $B$.
%} 

%We have 
%\begin{lemme}
%Let us fix $m \geq 1, i \in \{1,...,m\}, r_1 > r_2 \geq 0$, then $\mathbb{P}_m$-a.s. we have 
%\[ f_{0, r_1}(\{ L_1,...,L_i\},A) = \sum_{B \in \mathcal{P}(V_{G_{r_2}}) ; A \in D_{r_1}(B)} f_{0, r_2}(\{ L_1,...,L_i\},B) f_{r_2, r_1}(B,A). \]
%\end{lemme}

\subsubsection{Renewal structure} \label{behaviour022}

The following lemma makes appear the renewal structure of $(F^{i}_{G_{\beta}}(\cdot))_{\beta \geq 0}$ and is a key ingredient in proving the semigroup property (Proposition \ref{semigroupprop}) and in deriving the system of differential equations satisfied by coefficients $R^{m,k}_t(i,j)$ (Theorem \ref{equadiffcoeffnew}). It contains 1) a formula, involving $f_{\cdot,\cdot}(\cdot,\cdot)$, that can be seen as a branching property for $(F^{i}_{G_{\beta}}(\cdot))_{\beta \geq 0}$, and 2) a formula for conditional expectations involving $f_{\cdot,\cdot}(\cdot,\cdot)$, which makes the renewal appear. 

\begin{lemme} \label{markpropf}

Let us fix $m \geq 1, i \in \{1,...,m\}, r_1 > r_2 \geq 0$, then $\mathbb{P}_m$-a.s. we have 
%\begin{align}
%\forall A \in \mathcal{P}(V_{G_{r_1}}), \ F^i_{G_{r_1}}(A) = \sum_{B \in \mathcal{P}(V_{G_{r_2}}) ; A \in D_{r_1}(B)} F^i_{G_{r_2}}(B) f_{r_2, r_1}(B,A), \label{markpropf1} 
%\end{align}
\begin{align}
\forall A \in \mathcal{P}(V_{G_{r_1}}), \ F^i_{G_{r_1}}(A) = \sum_{B \in \mathcal{P}(V_{G_{r_2}})} F^i_{G_{r_2}}(B) f_{G_{r_2}, G_{r_1}}(B,A). \label{markpropf1} 
\end{align}
%Finally, for any $B \in \mathcal{P}(V_{G_{r_2}})$ and $A \in D_{r_1}(B)$ we have 
%\begin{align}
%\forall B \in \mathcal{P}(V_{G_{r_2}}), A \in \mathcal{P}(V_{G_{r_1}}), \ & f_{G_{r_2}, G_{r_1}}(B,A) \leq |A|^{|A|}. \label{markpropf3old}
%\end{align}
Recall the filtration $(\mathcal{F}_t)_{t \geq 0}$ defined in Section \ref{enlargedasgdefnot}. For any $B \in \mathcal{P}(V_{G_{r_2}})$ we have 
%\begin{align}
%\mathbb{E}_m \left [ \mathds{1}_{|V_{G_{r_1}}| = k} \sum_{A \in \mathcal{P}(V_{G_{r_1}}) ; A \in D_{r_1}(B) ; |A|=j} f_{r_2, r_1}(B,A) \big | \mathcal{F}_{r_2} \right ] = R^{|V_{G_{r_2}}|,k}_{r_1-r_2}(|B|,j), \label{markpropf2} 
%\end{align}
\begin{align}
\mathbb{E}_m \left [ \mathds{1}_{|V_{G_{r_1}}| = k} \sum_{A \in \mathcal{P}(V_{G_{r_1}}) ; |A|=j} f_{G_{r_2}, G_{r_1}}(B,A) \big | \mathcal{F}_{r_2} \right ] = R^{|V_{G_{r_2}}|,k}_{r_1-r_2}(|B|,j). \label{markpropf2} 
\end{align}

\end{lemme}

The idea is to prove \eqref{markpropf1} by induction on the number of transitions of the E-ASG on $(r_2,r_1]$. If \eqref{markpropf1} holds when there are $n$ transitions on $(r_2,r_1]$, then we deduce \eqref{multpropf2} below, which shows that \eqref{markpropf1} still holds when a supplementary transition is added into $(r_2,r_1]$, provided that some variants of \eqref{markpropf1} hold. 
%when there is exactly one transition on some interval. 
We then prove those variants of \eqref{markpropf1} (namely \eqref{multpropf3}, \eqref{multpropf4} and \eqref{multpropf5} below) by using the definitions of $f_{\cdot,\cdot}(\cdot,\cdot)$ and $F^{\cdot}_{\cdot}(\cdot)$ in all possible cases. \eqref{markpropf2} will be a consequence of the Markov property for the E-ASG. 

\begin{proof} [Proof of Lemma \ref{markpropf}]

Let us fix $A \in \mathcal{P}(V_{G_{r_1}})$. 
%We claim that 
%%according to the definition of $f_{\cdot,\cdot}(\cdot,\cdot)$ in \eqref{defpetitfg} we have 
%\begin{eqnarray}
%f_{0, r_1}(\{ L_1,...,L_i\},A) = \sum_{B \in \mathcal{P}(V_{G_{r_2}})} f_{0, r_2}(\{ L_1,...,L_i\},B) f_{r_2, r_1}(B,A). \label{multpropf}
%\end{eqnarray}
%If the claim is true, using the fact that, by \eqref{petitfgliengdg}, we have $f_{0, r_1}(\{ L_1,...,L_i\},A) = F^i_{G_{r_1}}(A)$ and $f_{0, r_2}(\{ L_1,...,L_i\},B) = F^i_{G_{r_2}}(B)$, we get that \eqref{markpropf1} follows. 
Let $n$ be the number of transitions of the E-ASG on $(r_2,r_1]$. We prove \eqref{markpropf1} by induction on $n$. If $n=0$ then $G_{r_1}=G_{r_2}$ so, for any $B \in \mathcal{P}(V_{G_{r_2}})$, $\mathsf{depth}(G_{r_1} \setminus_B G_{r_2})=0$ and we see, by the definition of $f_{G_{r_2}, G_{r_1}}(B,\cdot)$ in \eqref{defpetitfg} and by the definition of $F^{\cdot}_{\cdot}(\cdot)$ in Section \ref{enlargedasgencodingfct}, that for any $A \in \mathcal{P}(V_{G_{r_1}})$, $f_{G_{r_2}, G_{r_1}}(B,A) = \mathds{1}_{A=B}$. Therefore the right-hand side of \eqref{markpropf1} equals $F^{i}_{G_{r_2}}(A)$. Since $G_{r_1}=G_{r_2}$ we get that \eqref{markpropf1} holds if $n=0$. 

Let us now fix $n \geq 0$, assume that \eqref{markpropf1} holds for any interval $(r_2,r_1]$ containing $n$ transitions of the E-ASG, and prove that it holds for an interval $(r_2,r_1]$ containing $n+1$ transitions. Let us choose $r \in (r_2,r_1]$ that lies strictly between the $n^{th}$ and $n+1^{th}$ transition on $(r_2,r_1]$. By the induction hypothesis we can apply \eqref{markpropf1} with $r$ instead of $r_1$ and get, for $C \in \mathcal{P}(V_{G_{r}})$: 
\begin{eqnarray}
F^i_{G_{r}}(C) = \sum_{B \in \mathcal{P}(V_{G_{r_2}})} F^i_{G_{r_2}}(B) f_{G_{r_2}, G_r}(B,C). \label{multpropf1}
\end{eqnarray}
We now choose $A \in \mathcal{P}(V_{G_{r_1}})$, multiply both sides of \eqref{multpropf1} by $f_{G_r, G_{r_1}}(C,A)$ and sum over all $C \in \mathcal{P}(V_{G_{r}})$. We get 
%Applying the induction hypothesis to the left-hand side (POUR CA IL FAUT AUSSI PROUVER LE CAS $n=1$, PAS FORCEMENT, IL SUFFIT DE FAIRE A GAUCHE LA MEME CHOSE QU'A DROITE) 
%\begin{align*}
%f_{0, r_1}(\{ L_1,...,L_i\},A) = \sum_{C \in \mathcal{P}(V_{G_{r_2}})} f_{0, r_2}(\{ L_1,...,L_i\},C) \sum_{B \in \mathcal{P}(V_{G_{r}})} f_{r_2, r}(C,B) f_{r, r_1}(B,A). \label{multpropf1}
%\end{align*}
\begin{align}
\sum_{C \in \mathcal{P}(V_{G_{r}})} F^i_{G_{r}}(C) f_{G_r, G_{r_1}}(C,A) & = \sum_{C \in \mathcal{P}(V_{G_{r}})} f_{G_r, G_{r_1}}(C,A) \sum_{B \in \mathcal{P}(V_{G_{r_2}})} F^i_{G_{r_2}}(B) f_{G_{r_2}, G_r}(B,C) \nonumber \\
& = \sum_{B \in \mathcal{P}(V_{G_{r_2}})} F^i_{G_{r_2}}(B) \sum_{C \in \mathcal{P}(V_{G_{r}})} f_{G_{r_2}, G_r}(B,C) f_{G_r, G_{r_1}}(C,A). \label{multpropf2}
\end{align}
We distinguish three cases. If the $n+1^{th}$ transition of the E-ASG in $(r_2,r_1]$ is a coalescence, then we see from the definitions of $f_{\cdot,\cdot}(\cdot,\cdot)$ and $F^{\cdot}_{\cdot}(\cdot)$ (in \eqref{defpetitfg} and in Section \ref{enlargedasgencodingfct}) that for any $C \in \mathcal{P}(V_{G_{r}})$, $f_{G_r, G_{r_1}}(C,A) = \mathds{1}_{D(C)=A}$. Recall that for $C \in \mathcal{P}(V_{G_{r}})$, $D(C) \in \mathcal{P}(V_{G_{r_1}})$ is the set of sons, in $V_{G_{r_1}}$, of lines of $C$. We also have that $f_{G_{r_2}, G_r}(B,C)=F_{G_r \setminus_B G_{r_2}}^{|B|}(C)$ by definition of $f_{\cdot,\cdot}(\cdot,\cdot)$. We thus have 
\begin{align}
& \sum_{C \in \mathcal{P}(V_{G_{r}})} f_{G_{r_2}, G_r}(B,C) f_{G_r, G_{r_1}}(C,A) = \sum_{C \in \mathcal{P}(V_{G_{r}})} F_{G_r \setminus_B G_{r_2}}^{|B|}(C) \times \mathds{1}_{D(C)=A} \nonumber \\
= & \sum_{C \in \mathcal{P}(V_{G_{r}}) ; D(C) =A} F_{G_r \setminus_B G_{r_2}}^{|B|}(C) = F_{G_{r_1} \setminus_B G_{r_2}}^{|B|}(A) = f_{G_{r_2}, G_{r_1}}(B,A), \label{multpropf3}
\end{align}
where the last two equalities come from respectively \eqref{defFcoal} and \eqref{defpetitfg}. Proceeding similarly for the left-hand side of \eqref{multpropf2} we get 
\begin{eqnarray}
\sum_{C \in \mathcal{P}(V_{G_{r}})} F^i_{G_{r}}(C) f_{G_r, G_{r_1}}(C,A) = F^i_{G_{r_1}}(A). \label{multpropf4}
\end{eqnarray}
Plugging \eqref{multpropf3} and \eqref{multpropf4} into \eqref{multpropf2} we obtain that \eqref{markpropf1} holds for this interval $(r_2,r_1]$ with $n+1$ transitions of the E-ASG, when the last transition is a coalescence. We now assume that the $n+1^{th}$ transition in $(r_2,r_1]$ is a multiple branching with weight $W$. We then see from the definitions of $f_{\cdot,\cdot}(\cdot,\cdot)$ and $F^{\cdot}_{\cdot}(\cdot)$ (in \eqref{defpetitfg} and in Section \ref{enlargedasgencodingfct}) that for any $C \in \mathcal{P}(V_{G_{r}})$ we have $f_{G_r, G_{r_1}}(C,A) = \mathds{1}_{P(A)=C} \times (1+W \mathds{1}_{W < 0})^{\alpha(A)} \times (W \mathds{1}_{W > 0})^{\beta(A)} \times (-W)^{\gamma(A)}$. We recall that for $A \in \mathcal{P}(V_{G_{r_1}})$, $P(A) \in \mathcal{P}(V_{G_{r}})$ is the set of parents, in $V_{G_{r}}$, of lines of $A$. We also have from the definition of $f_{\cdot,\cdot}(\cdot,\cdot)$ that $f_{G_{r_2}, G_r}(B,C)=F_{G_r \setminus_B G_{r_2}}^{|B|}(C)$. We thus get 
\begin{align}
& \sum_{C \in \mathcal{P}(V_{G_{r}})} f_{G_{r_2}, G_r}(B,C) f_{G_r, G_{r_1}}(C,A) \nonumber \\
= & \sum_{C \in \mathcal{P}(V_{G_{r}})} F_{G_r \setminus_B G_{r_2}}^{|B|}(C) \times \mathds{1}_{P(A)=C} \times (1+W \mathds{1}_{W < 0})^{\alpha(A)} \times (W \mathds{1}_{W > 0})^{\beta(A)} \times (-W)^{\gamma(A)} \nonumber \\
= & F_{G_r \setminus_B G_{r_2}}^{|B|}(P(A)) \times (1+W \mathds{1}_{W < 0})^{\alpha(A)} \times (W \mathds{1}_{W > 0})^{\beta(A)} \times (-W)^{\gamma(A)} \nonumber \\
= & F_{G_{r_1} \setminus_B G_{r_2}}^{|B|}(A) = f_{G_{r_2}, G_{r_1}}(B,A), \label{multpropf5}
\end{align}
where the last two equalities come from respectively \eqref{defFcoal} and \eqref{defpetitfg}. Proceeding similarly for the left-hand side of \eqref{multpropf2} we get that \eqref{multpropf4} also holds when the $n+1^{th}$ transition of the E-ASG on $(r_2,r_1]$ is a multiple branching. Plugging \eqref{multpropf5} and \eqref{multpropf4} into \eqref{multpropf2} we obtain that \eqref{markpropf1} holds for this interval $(r_2,r_1]$ with $n+1$ transitions of the E-ASG, when the last transition is a multiple branching. The proof of \eqref{markpropf1} for this interval $(r_2,r_1]$ with $n+1$ transitions of the E-ASG, in the case where the last transition is a single branching, is done similarly. This concludes the proof by induction. 

%\eqref{markpropf3} is a consequence of the definition of $f_{G_{r_2}, G_{r_1}}(B,A)$ in \eqref{defpetitfg} together with Lemma \ref{lemmecomb}. 

We now justify \eqref{markpropf2}. First, we see from \eqref{markpropf3} that the conditional expectation in \eqref{markpropf2} is well-defined. 
%Let $(\tilde G_{t})_{t \geq 0} = (G_{r_2 + t})_{t \geq 0}$, then note that $\mathcal{P}(V_{G_{r_2}}) = \mathcal{P}(V_{\tilde G_{0}})$ (so $B \in \mathcal{P}(V_{G_{r_2}}) \Leftrightarrow B \in \mathcal{P}(V_{\tilde G_{0}})$) and that we have 
Using the definitions of $\cdot \setminus_B \cdot$ and $f_{\cdot,\cdot}(\cdot,\cdot)$ we get that, for any $B \in \mathcal{P}(V_{G_{r_2}})$, the term inside that conditional expectation equals 
%\begin{align*}
%& \mathds{1}_{|V_{G_{r_1}}| = k} \sum_{A \in \mathcal{P}(V_{G_{r_1}}) ; |A|=j} f_{r_2, r_1}(B,A) = \mathds{1}_{|V_{\tilde G_{r_1-r_2}}| = k} \sum_{A \in \mathcal{P}(V_{\tilde G_{r_1-r_2}}) ; |A|=j} \tilde f_{0,r_1-r_2}(B,A), 
%%= & \mathds{1}_{|V_{\tilde G_{r_1-r_2}}| = k} \sum_{A \in \mathcal{P}(V_{\tilde G_{r_1-r_2}}) ; |A|=j} F^{|B|}_{\tilde G_{r_1-r_2}}(A). 
%\end{align*} 
\begin{align*}
\mathds{1}_{|V_{G_{r_1} \setminus_B G_{r_2}}| = k} \sum_{A \in \mathcal{P}(V_{G_{r_1} \setminus_B G_{r_2}}) ; |A|=j} F_{G_{r_1} \setminus_B G_{r_2}}^{|B|}(A), 
%= & \mathds{1}_{|V_{\tilde G_{r_1-r_2}}| = k} \sum_{A \in \mathcal{P}(V_{\tilde G_{r_1-r_2}}) ; |A|=j} F^{|B|}_{\tilde G_{r_1-r_2}}(A). 
\end{align*}
which is only a function of the process $(G_{t+r_2} \setminus_B G_{r_2})_{t \geq 0}$. By the Markov property for $(G_{\beta})_{\beta \geq 0}$ we see that $(G_{t+r_2} \setminus_B G_{r_2})_{t \geq 0}$ under $\mathbb{P}_m (\cdot | \mathcal{F}_{r_2} )$ is equal in law to the E-ASG under $\mathbb{P}_{|V_{G_{r_2}}|}(\cdot)$. Therefore, using the definition of $R^{\cdot,\cdot}_t(\cdot,\cdot)$ in \eqref{defcoeff}, we get that the left-hand side of \eqref{markpropf2} equals $R^{|V_{G_{r_2}}|,k}_{r_1-r_2}(|B|,j)$, completing the proof. 
%\begin{align*}
%& \mathbb{E}_{|V_{G_{r_2}}|} \left [ \mathds{1}_{|V_{\hat G_{r_1-r_2}}| = k} \sum_{A \in \mathcal{P}(V_{\hat G_{r_1-r_2}}) ; |A|=j} f^{\hat G}_{0,r_1-r_2}(B,A) \right ] \\
%= & \mathbb{E}_{|V_{G_{r_2}}|} \left [ \mathds{1}_{|V_{\hat G_{r_1-r_2}}| = k} \sum_{A \in \mathcal{P}(V_{\hat G_{r_1-r_2}}) ; |A|=j} f^{\hat G}_{0,r_1-r_2}(\{ L_1,..., L_{|B|}\},A) \right ] \\
%= & \mathbb{E}_{|V_{G_{r_2}}|} \left [ \mathds{1}_{|V_{\hat G_{r_1-r_2}}| = k} \sum_{A \in \mathcal{P}(V_{\hat G_{r_1-r_2}}) ; |A|=j} F^{|B|}_{\hat G_{r_1-r_2}}(A) \right ] = R^{|V_{G_{r_2}}|,k}_{r_1-r_2}(|B|,j). 
%\end{align*}
%\begin{align*}
%& \mathbb{E}_{|V_{G_{r_2}}|} \left [ \mathds{1}_{|V_{\hat G_{r_1-r_2}}| = k} \sum_{A \in \mathcal{P}(V_{\hat G_{r_1-r_2}}) ; |A|=j} F^{|B|}_{\hat G_{r_1-r_2}}(A) \right ] = R^{|V_{G_{r_2}}|,k}_{r_1-r_2}(|B|,j), 
%\end{align*}
%where the last equality comes from the definition of $R^{\cdot,\cdot}_t(\cdot,\cdot)$ in \eqref{defcoeff}. 

%POUR \eqref{markpropf2} OK CAR PAR CONSTRUCTION $f_{G_t, G_{t+h}}(B,A)$ NE TIENT PAS COMPTE DES COALESCENCES QUI ONT LIEU EN DEHORS DE $B$, ENFIN DE SES DESCENDANTS. 

%\begin{align*}
%\forall B \in \mathcal{P}(V_{G_{r_2}}), \ & \mathbb{E}_m \left [ \mathds{1}_{|V_{G_{r_1}}| = k} \sum_{A \in \mathcal{P}(V_{G_{r_1}}) ; A \in D_{r_1}(B) ; |A|=j} f_{r_2, r_1}(B,A) \big | \mathcal{F}_{r_2} \right ] \nonumber \\
%= & \mathbb{E}_{|V_{G_{r_2}}|} \left [ \mathds{1}_{|V_{G_{r_1-r_2}}| = k} \sum_{A \in \mathcal{P}(V_{G_{r_1-r_2}}) ; A \in D_{r_1-r_2}(B) ; |A|=j} F^{|B|}_{G_{r_1-r_2}}(A) \right ], \label{markpropf2} 
%\end{align*}

\end{proof}

\begin{remark} \label{markpropfquenchedrmk}
Note that \eqref{markpropf1} actually holds for any $G, \tilde G \in \mathbb{G}_m$ such that $\pi_{\mathsf{depth}(G)}(\tilde G)=G$. In particular, \eqref{markpropf1} holds in the quenched setting: for any fixed environment $\omega$ and $t > 0$ we have $\mathbb{P}^{\omega,t}_m$-a.s. 
\begin{align}
\forall 0\leq r_1<r_2 \leq t, \ \forall A \in \mathcal{P}(V_{G^{\omega,t}_{r_1}}), \ F^i_{G^{\omega,t}_{r_1}}(A) = \sum_{B \in \mathcal{P}(V_{G^{\omega,t}_{r_2}})} F^i_{G^{\omega,t}_{r_2}}(B) f_{G^{\omega,t}_{r_2}, G^{\omega,t}_{r_1}}(B,A). \label{markpropfquenched} 
\end{align}
Also, proceeding as in the proof of \eqref{markpropf2} we can see that for $0\leq r_1<r_2\leq t$ and $B \in \mathcal{P}(V_{G^{\omega,t}_{r_2}})$ we have 
\begin{align}
\mathbb{E}^{\omega,t}_m \left [ \mathds{1}_{|V_{G^{\omega,t}_{r_1}}| = k} \sum_{A \in \mathcal{P}(V_{G^{\omega,t}_{r_1}}) ; |A|=j} f_{G^{\omega,t}_{r_2}, G^{\omega,t}_{r_1}}(B,A) \big | \mathcal{F}^{\omega,t}_{r_2} \right ] = R^{|V_{G^{\omega,t}_{r_2}}|,k,\omega}_{r_1,r_2}(|B|,j). \label{markpropfquenched2}  
\end{align}
\end{remark}

\subsubsection{semigroup property: Proof of Proposition \ref{semigroupprop}} \label{behaviour023}

The semigroup property can now be derived thanks to Lemma \ref{markpropf}: Using the definition of $R^{\cdot,\cdot}_{\cdot}(\cdot,\cdot)$ in \eqref{defcoeff}, dominated convergence, \eqref{markpropf1} and \eqref{markpropf2} with $r_1=t+r$ and $r_2=t$, and again the definition of $R^{\cdot,\cdot}_{\cdot}(\cdot,\cdot)$, we get that $R^{m,k}_{t+r}(i,j)$ equals 
\begin{align*}
%& \mathbb{E}_m \left [ \mathds{1}_{|V_{G_{t+r}}| = k} \sum_{A \in \mathcal{P}(V_{G_{t+r}}) ; |A|=j} F^i_{G_{t+r}}(A) \right ] \\
%& = \mathbb{E}_m \left [ \sum_{\tilde k = 1}^{\infty} \mathds{1}_{|V_{G_{t}}| = \tilde k} \mathds{1}_{|V_{G_{t+r}}| = k} \sum_{A \in \mathcal{P}(V_{G_{t+r}}) ; |A|=j} F^i_{G_{t+r}}(A) \right ] \\
& \sum_{\tilde k = 1}^{\infty} \mathbb{E}_m \left [ \mathds{1}_{|V_{G_{t}}| = \tilde k} \mathds{1}_{|V_{G_{t+r}}| = k} \sum_{A \in \mathcal{P}(V_{G_{t+r}}) ; |A|=j} F^i_{G_{t+r}}(A) \right ] \\
= & \sum_{\tilde k = 1}^{\infty} \mathbb{E}_m \left [ \mathds{1}_{|V_{G_{t}}| = \tilde k} \mathds{1}_{|V_{G_{t+r}}| = k} \sum_{A \in \mathcal{P}(V_{G_{t+r}}) ; |A|=j} \left ( \sum_{B \in \mathcal{P}(V_{G_{t}})} F^i_{G_{t}}(B) f_{G_t, G_{t+r}}(B,A) \right ) \right ] \\
= & \sum_{\tilde k = 1}^{\infty} \mathbb{E}_m \left [ \mathds{1}_{|V_{G_{t}}| = \tilde k} \sum_{B \in \mathcal{P}(V_{G_{t}})} F^i_{G_{t}}(B) \mathbb{E}_m \left [ \mathds{1}_{|V_{G_{t+r}}| = k} \sum_{A \in \mathcal{P}(V_{G_{t+r}}) ; |A|=j} f_{G_t, G_{t+r}}(B,A) \big | \mathcal{F}_{t} \right ] \right ] \\
= & \sum_{\tilde k = 1}^{\infty} \mathbb{E}_m \left [ \mathds{1}_{|V_{G_{t}}| = \tilde k} \sum_{B \in \mathcal{P}(V_{G_{t}})} F^i_{G_{t}}(B) R^{\tilde k,k}_{r}(|B|,j) \right ]
\end{align*}
\begin{align*}
= & \sum_{\tilde k = 1}^{\infty} \sum_{\tilde j = 1}^{\tilde k} \mathbb{E}_m \left [ \mathds{1}_{|V_{G_{t}}| = \tilde k} \sum_{B \in \mathcal{P}(V_{G_{t}}) ; |B| = \tilde j} F^i_{G_{t}}(B) \right ] R^{\tilde k,k}_{r}(\tilde j,j) = \sum_{\tilde k = 1}^{\infty} \sum_{\tilde j = 1}^{\tilde k} R^{m,\tilde k}_{t}(i, \tilde j) R^{\tilde k,k}_{r}(\tilde j,j). 
\end{align*}

\subsubsection{Another lemma} \label{behaviour024}

The following lemma will be used in the proof of Theorem \ref{equadiffcoeffnew} where we differentiate $R^{m,k}_t(i,j)$ with respect to $t$. For this we deal with finitely many relevant terms and some remainders consisting of sums of infinitely many terms. The following lemma will allow to see that those infinitely many terms can be non-zero only on an event with negligible probability, so that we only need to use a simple deterministic bound for the remainders in order to neglect their expectations, and eventually work with finitely many terms. 

\begin{lemme} \label{atleast2coal}
Let us fix $m \geq 1, r_1 > r_2 \geq 0$. $\mathbb{P}_m$-a.s. we have that, if there are strictly less than two coalescences in the E-ASG on $(r_2,r_1]$, then $f_{G_{r_2}, G_{r_1}}(B,A) = 0$ for any $A \in \mathcal{P}(V_{G_{r_1}})$ and $B \in \mathcal{P}(V_{G_{r_2}})$ such that $|B| \geq |A|+2$. 

\end{lemme}

\begin{proof}

%If there is no transition at all on $[r_2, r_1]$ then we see from the proof the case $n=0$ of \eqref{multpropf}, in the proof of Lemma \ref{markpropf}, that $f_{r_2, r_1}(B,A) = \mathds{1}_{A=B}$ so the claim is true. 
We fix $k \geq 1, j \in \{1,...,k\}$. Let $(\hat G_{\beta})_{\beta \geq 0}$ be a realization of the E-ASG under $\mathbb{P}_k$ and $\hat T_1<\hat T_2<...$ be its transition times. Let $M_{\beta} := \min \{ |A|, A \in \mathcal{P}(V_{\hat G_{\beta}}) \ \text{s.t.} \ F^j_{\hat G_{\beta}}(A) \neq 0 \}$. We see from the definition of $F^{j}_{\cdot}(\cdot)$ in Section \ref{enlargedasgencodingfct} that $M_0 = j$, and $M$ is constant between two transition times of $(\hat G_{\beta})_{\beta \geq 0}$. Moreover, at a transition time $\hat T_n$, if the transition is a single or a multiple branching then $M_{\hat T_n}\geq M_{\hat T_n-}$. If the transition at time $\hat T_n$ is a coalescence between two lines that both lie in a set $A \in \mathcal{P}(V_{\hat G_{\hat T_n-}})$ such that $|A|=M_{\hat T_n-}$, then $M_{\hat T_n}=M_{\hat T_n-}-1$. If the transition at time $\hat T_n$ is a coalescence, but there is no set $A \in \mathcal{P}(V_{\hat G_{\hat T_n-}})$ with $|A|=M_{\hat T_n-}$ that contains both lines involved in the coalescence, then $M_{\hat T_n}\geq M_{\hat T_n-}$. In any case we get that, if there are strictly less than two coalescences on $[0,t]$, then $M_t \geq j-1$. In this case we thus have $F^j_{\hat G_t}(A)=0$ for all $A \in \mathcal{P}(V_{\hat G_{t}})$ such that $|A| \leq j-2$. 

Now let us fix $m \geq 1, r_1 > r_2 \geq 0$, and let $(G_{\beta})_{\beta \geq 0}$ be a realization of the E-ASG under $\mathbb{P}_m$. We assume that $(G_{\beta})_{\beta \geq 0}$ has strictly less than two coalescences on $(r_2,r_1]$ and choose $A \in \mathcal{P}(V_{G_{r_1}})$ and $B \in \mathcal{P}(V_{G_{r_2}})$ such that $|B| \geq |A|+2$. By the definition of $f_{\cdot,\cdot}(\cdot,\cdot)$ in \eqref{defpetitfg} we have $f_{G_{r_2}, G_{r_1}}(B,A) = F_{G_{r_1} \setminus_B G_{r_2}}^{|B|}(A)$. Applying the above with $k := |V_{G_{r_2}}|, \hat G_{\beta} := G_{r_2 + \beta} \setminus_B G_{r_2}$ and $j := |B|$ we get $f_{G_{r_2}, G_{r_1}}(B,A)=0$. 

\end{proof}

\subsection{Small-time behavior} \label{behaviour01}

This subsection and the following are dedicated to the proof of Theorem \ref{equadiffcoeffnew}. The proof is a little heavy as it requires to consider separately several sub-cases. In order to make the proof lighter and more understandable, we make the choice to write the proof of Theorem \ref{equadiffcoeffnew} only in the case $\sigma=0$ (i.e. without single branchings). The general case with $\sigma > 0$ can be proved exactly in the same way and involves more sub-cases but it does not bring any additional technical issue. \textbf{In this Subsection and the following, we thus assume that $\sigma=0$}. 

As shown by Proposition \ref{semigroupprop}, $(R^{\cdot,\cdot}_t(\cdot,\cdot))_{t \geq 0}$ is a semigroup. 
%of infinite matrices. However, it is not clear whether the sums $\sum_{k \geq 1} \sum_{j = 1}^k |R^{m,k}_t(i,j)|$ are finite, which make it hard to work directly with the semigroup property. We are mostly interested in the system of differential equations satisfied by the coefficients $R^{m,k}_t(i,j)$ and it is possible to establish it without passing directly by the semigroup property. 
%INUTILE DE DIRE CA ICI CAR QU'ON UTILISE DIRECTEMENT LE SG OU PAS 1) ON A TOUJOURS BESOIN D'ETABLIR L'EQUA DIFF, 2) POUR CA ON A TOUJOURS BESOIN DES COMPORTEMENTS EN 0. DU COUP CA A PLUTOT SA PLACE AU MOMENT OU ON PROUVE L'EQUA DIFF -> LE METTRE LA 
To derive the system of differential equations satisfied by $(R^{\cdot,\cdot}_t(\cdot,\cdot))_{t \geq 0}$, we first need to establish the asymptotic behavior of $R^{m,k}_t(i,j)$ as $t$ goes to $0$, for all possible combinations of $m,k,i,j$. We do this in the following lemma. Recall the definitions of the coefficients $d_j, e_{k,j}$ and $\tau(i,j)$ in Definition \ref{defcoefedo} (but note that, since $\sigma=0$ here, we have $d_j=\lambda + j(j-1)$). We have 

\begin{lemme} \label{asymptcoeffnew}

\begin{align}
\forall k \ \text{even} \ \geq 2, j \in \{ 1,..., k \}, i \in \{ 1,..., j \wedge k/2 \}, \ & \lim_{\epsilon \rightarrow 0} \frac1{\epsilon} R^{k/2,k}_\epsilon(i,j) = \tau(i,j), \label{asymptcoeff1new} \\
\forall k \geq 1, j \in \{ 1,..., k \}, \ & \lim_{\epsilon \rightarrow 0} \frac1{\epsilon} R^{k+1,k}_\epsilon(j+1,j) = \tau(j+1,j), \label{asymptcoeff23new} \\
\forall k \geq 1, j \in \{ 1,..., k \}, \ & \lim_{\epsilon \rightarrow 0} \frac1{\epsilon} R^{k+1,k}_\epsilon(j,j) = e_{k,j}, \label{asymptcoeff27new} \\
\forall k \geq 1, j \in \{ 1,..., k \}, \ & \lim_{\epsilon \rightarrow 0} \frac1{\epsilon} (R^{k,k}_\epsilon(j,j)-1) = -d_k. \label{asymptcoeff26new}
\end{align}
For $k \geq 1, m \neq k/2, j \in \{ 1,..., k \}, i \in \{ 1, 2,..., (j-1) \wedge m \}$, and $\epsilon>0$ we have 
\begin{eqnarray}
\left | R^{m,k}_\epsilon(i,j) \right | \leq 16 m^4 \frac{(jk)^j}{j!} (1+\lambda)^2 \epsilon^2. \label{asymptcoeff13new}
\end{eqnarray}
For $k \geq 1, m \neq k+1, j \in \{ 1,..., k \wedge (m-1) \}$, and $\epsilon>0$ we have 
\begin{eqnarray}
\left | R^{m,k}_\epsilon(j+1,j) \right | \leq 16 m^4 \frac{(jk)^j}{j!} (1+\lambda)^2 \epsilon^2. \label{asymptcoeff14new}
\end{eqnarray}
For $k \geq 1, m \notin \{k/2, k, k+1\}, j \in \{ 1,..., k \wedge m \}$, and $\epsilon>0$ we have 
\begin{eqnarray}
\left | R^{m,k}_\epsilon(j,j) \right | \leq 16 m^4 \frac{(jk)^j}{j!} (1+\lambda)^2 \epsilon^2. \label{asymptcoeff15new}
\end{eqnarray}

\end{lemme}

To prove Lemma \ref{asymptcoeffnew} we study the behaviors of the expectations defining the coefficients $R^{m,k}_\epsilon(i,j)$, for all possible choices of $m,k,i,j$, on the events where the number of transitions of the E-ASG on $[0,\epsilon]$ is $0$, $1$, and larger or equal to $2$. We then deduce the estimates \eqref{asymptcoeff1new}-\eqref{asymptcoeff15new} from those behaviors. This involves both the combinatorics arising from the definition of $F^{\cdot}_{\cdot}(\cdot)$, and the small-time behavior of transition probabilities of the E-ASG. Some technical lemmas are required to cover some cases; we state and prove them in Appendix~\ref{A3}. 

\begin{proof} [Proof of Lemma \ref{asymptcoeffnew}]

Let $N_{mult}(\epsilon)$ and $N_{coal}(\epsilon)$ denote respectively the number of multiple branchings and coalescences of the E-ASG on $[0,\epsilon]$. $N_{mult}(\epsilon)+N_{coal}(\epsilon)$ is then the total number of transitions of the E-ASG on $[0,\epsilon]$ (since we assume that $\sigma=0$). For all $p \geq 1$ we set 
\begin{align*}
C_p(\epsilon) & := \{ N_{mult}(\epsilon)+N_{coal}(\epsilon)=p \} \ \text{and} \ C_{\geq p}(\epsilon) := \{ N_{mult}(\epsilon)+N_{coal}(\epsilon) \geq p \}. 
\end{align*}

Let us fix $k, m \geq 1$, $j \in \{ 1,..., k \}$ and $i \in \{ 1, 2,..., m \}$. We can write 
\begin{align}
R^{m,k}_\epsilon(i,j) & = \mathbb{E}_{m} \left [ \mathds{1}_{|V_{G_{\epsilon}}| = k} \sum_{A \in \mathcal{P}(V_{G_{\epsilon}}) ; |A|=j} F^i_{G_{\epsilon}}(A) \right ] \nonumber \\
& = \mathbb{E}_{m} \left [ \mathds{1}_{C_0(\epsilon)} \cdots \right ] + \mathbb{E}_{m} \left [ \mathds{1}_{C_1(\epsilon)} \cdots \right ] + \mathbb{E}_{m} \left [ \mathds{1}_{C_{\geq 2}(\epsilon)} \cdots \right ] \nonumber \\ 
& =: E_0({m},k,i,j,\epsilon) + E_1({m},k,i,j,\epsilon) + E_{\geq 2}({m},k,i,j,\epsilon). \label{asymptcoeff3new}
\end{align}

We first study $E_{\geq 2}(m,k,i,j,\epsilon)$. 
%This will allow, in the rest of the proof, to neglect the case where the E-ASG has more than $2$ transitions on $[0,h]$. 
%For the sake of generality note that the following reasoning is valid for any $k \in \{1,..., j+1 \}$ and not only for $k \in \{ 1, 2,..., j-1, j+1 \}$. 
Fix $m,k \geq 1, j \in \{ 1,..., k \}, i \in \{ 1, 2,..., m \}$. Using the definition of $E_{\geq 2}(m,k,i,j,\epsilon)$ in \eqref{asymptcoeff3new}, \eqref{borneFps} with $t=\epsilon$, and \eqref{esteasg>2} from Lemma \ref{esteasg}, we get 
\begin{eqnarray}
E_{\geq 2}(m,k,i,j,\epsilon) \leq \frac{(jk)^j}{j!} \mathbb{E}_m \left [ \mathds{1}_{C_{\geq 2}(\epsilon)} \right ] \leq 16 m^4 \frac{(jk)^j}{j!} (1+\lambda)^2 \epsilon^2. \label{asymptcoeff20new}
\end{eqnarray}
\eqref{asymptcoeff20new} implies in particular that $\epsilon^{-1} E_{\geq 2}(m,k,i,j,\epsilon)$ converges to $0$ as $\epsilon$ goes to $0$. 

In \eqref{asymptcoeff3new}, since $E_{\geq 2}(m,k,i,j,\epsilon)$ can be neglected, we now focus on cases where there is $0$ or $1$ transition on $[0,\epsilon]$. We first justify that $E_0(m,k,i,j,\epsilon)$ is null for some choices of $m,k,i,j$. Let $k, m \geq 1, j \in \{ 1,..., k \}, i \in \{ 1, 2,..., (j+1) \wedge m \}$ with $i \neq j$. 
%By construction of $(G_{\beta})_{\beta \geq 0}$ and of $F^i_{\cdot}(\cdot)$ 
We have, on $C_0(\epsilon)$, that $G_\epsilon = G_0$ so $F^i_{G_\epsilon}(\cdot) = F^i_{G_0}(\cdot)$ and, by definition of $F^i_{\cdot}(\cdot)$ in Section \ref{enlargedasgencodingfct}, $F^i_{G_0}(A) = \mathds{1}_{A=\{ L_1,..., L_i \}}$ (where we recall that $L_1,..., L_i$ denote the first $i$ lines of $G_0$). This shows that $\mathds{1}_{C_0(\epsilon)} \sum_{A \in \mathcal{P}(V_{G_{\epsilon}}) ; |A|=j} F^i_{G_{\epsilon}}(A) = 0$ if $j \neq i$. Taking expectation $\mathbb{E}_m[.]$ we get 
\begin{eqnarray}
\forall k, m \geq 1, j \in \{ 1,..., k \}, i \in \{ 1, 2,..., (j+1) \wedge m \} \ \text{with} \ i \neq j, \ E_0(m,k,i,j,\epsilon) = 0. \label{asymptcoeff16new}
\end{eqnarray}
Let $k \geq 1, m \neq k, j \in \{ 1,..., k \}, i \in \{ 1, 2,..., m \}$. On $C_0(\epsilon)$ there is no transition of the E-ASG on $[0,\epsilon]$ so $\mathds{1}_{C_0(\epsilon)} \mathds{1}_{|V_{G_{\epsilon}}| = k} = 0$, $\mathbb{P}_m$ almost surely. Taking expectation $\mathbb{E}_m[.]$ we thus get 
\begin{eqnarray}
\forall k \geq 1, m \neq k, j \in \{ 1,..., k \}, i \in \{ 1, 2,..., m \}, \ E_0(m,k,i,j,\epsilon) = 0. \label{asymptcoeff17new}
\end{eqnarray}

We now study $E_0(k,k,j,j,\epsilon)$ where $k \geq 1, j \in \{ 1,..., k \}$. Recall that, on $C_0(\epsilon)$, we have $G_\epsilon = G_0$ and $F^j_{G_\epsilon}(\cdot) = F^j_{G_0}(\cdot)$ and, by definition, $F^j_{G_0}(A) = \mathds{1}_{A=\{ L_1,..., L_j \}}$. Therefore 
\[ \mathds{1}_{C_0(\epsilon)} \mathds{1}_{|V_{G_{\epsilon}}| = k} \sum_{A \in \mathcal{P}(V_{G_{\epsilon}}) ; |A|=j} F^j_{G_{\epsilon}}(A) = \mathds{1}_{C_0(\epsilon)} \mathds{1}_{|V_{G_{0}}| = k}. \] 
Taking expectation $\mathbb{E}_k[.]$ and subtracting $1$ we get 
\[ E_0(k,k,j,j,\epsilon) - 1 = \mathbb{P}_k \left ( C_0(\epsilon) \right ) - 1 = e^{-\epsilon(\lambda + k(k-1))} - 1. \]
In the last equality we used that, under $\mathbb{P}_{k}$, the first transition of the E-ASG appears at rate $\lambda + k(k-1)$. We deduce that 
\begin{eqnarray}
\forall k \geq 1, j \in \{ 1,..., k \}, \ \lim_{\epsilon \rightarrow 0} \frac{E_{0}(k,k,j,j,\epsilon) - 1}{\epsilon} = - \left ( \lambda + k(k-1) \right ) = - d_k. \label{asymptcoeff30new}
\end{eqnarray}

We now study the terms $E_1(m,k,i,j,\epsilon)$. First, let $k \geq 1, m \neq k/2, j \in \{ 1,..., k \}, \\ i \in \{ 1, 2,..., (j-1) \wedge m \}$. On $C_1(\epsilon)$ there is exactly one transition on $[0,\epsilon]$. Under $\mathbb{P}_m$, if this transition is a multiple branching, then $|V_{G_{\epsilon}}| = 2m \neq k$ so $\mathds{1}_{|V_{G_{\epsilon}}| = k} \sum_{A \in \mathcal{P}(V_{G_{\epsilon}}) ; |A|=j} F^i_{G_{\epsilon}}(A) = 0$. 
Recall that, by definition, $F^i_{G_0}(A) = \mathds{1}_{A=\{ L_1,..., L_i \}}$, and that $D(\{ L_1,..., L_i \}) \in \mathcal{P}(V_{G_{\epsilon}})$ denotes the set of sons in $V_{G_{\epsilon}}$ of the lines $L_1,..., L_i$. 
If the transition on $[0,\epsilon]$ is a coalescence then $|D(\{ L_1,..., L_i \})| \in \{i-1,i\}$ so $|D(\{ L_1,..., L_i \})| < j$. By \eqref{defFcoal} and by $F^i_{G_0}(A) = \mathds{1}_{A=\{ L_1,..., L_i \}}$ we get that $F^i_{G_\epsilon}(A) = 0$ if $A \neq D(\{ L_1,..., L_i \})$. In particular $F^i_{G_\epsilon}(A) = 0$ for all $A \in \mathcal{P}(V_{G_{\epsilon}})$ such that $|A|=j$. Therefore $\mathds{1}_{|V_{G_{\epsilon}}| = k} \sum_{A \in \mathcal{P}(V_{G_{\epsilon}}) ; |A|=j} F^i_{G_{\epsilon}}(A) = 0$ in any case on $C_1(\epsilon)$ under $\mathbb{P}_m$. Taking expectation $\mathbb{E}_m[.]$ we get 
\begin{eqnarray}
\forall k \geq 1, m \neq k/2, j \in \{ 1,..., k \}, i \in \{ 1, 2,..., (j-1) \wedge m \}, \ E_1(m,k,i,j,\epsilon) = 0. \label{asymptcoeff18new}
\end{eqnarray}
\eqref{asymptcoeff13new} now follows from \eqref{asymptcoeff3new} together with \eqref{asymptcoeff16new}, \eqref{asymptcoeff18new}, and \eqref{asymptcoeff20new}. 

Now let $k \geq 1, m \neq k+1, j \in \{ 1,..., k \wedge (m-1) \}$. On $C_1(\epsilon)$ there is exactly one transition on $[0,\epsilon]$. Under $\mathbb{P}_m$, if this transition is a coalescence, then $|V_{G_{\epsilon}}| = m-1 \neq k$ so $\mathds{1}_{|V_{G_{\epsilon}}| = k} \sum_{A \in \mathcal{P}(V_{G_{\epsilon}}) ; |A|=j} F^{j+1}_{G_{\epsilon}}(A) = 0$. If this transition is a multiple branching then \eqref{defFjump} shows that, for any $A \in \mathcal{P}(V_{G_{\epsilon}})$, $F^{j+1}_{G_\epsilon}(A)$ is proportional to $F^{j+1}_{G_0}(P(A))$. 
%\[ F^{j+1}_{G_\epsilon}(A) = F^{j+1}_{G_0}(P(A)) \times (1+S_{1} \mathds{1}_{S_{1} < 0})^{\alpha(A)} \times (S_{1} \mathds{1}_{S_{1} > 0})^{\beta(A)} \times (-S_{1})^{\gamma(A)}. \]
Recall that $P(A) \in \mathcal{P}(V_{G_{0}})$ is the set of parents in $V_{G_0}$ of lines of $A$. 
%We also recall that $S_1$ is the weight of the first transition of $(G_{\beta})_{\beta \geq 0}$. 
By definition, $F^{j+1}_{G_0}(P(A)) = \mathds{1}_{P(A) = \{L_1,..., L_{j+1} \}}$, where $L_1,..., L_{j+1}$ denote the first $j+1$ lines of $G_0$. We thus get that $F^{j+1}_{G_\epsilon}(A) \neq 0$ implies $P(A) = \{L_1,..., L_{j+1} \}$ so, in particular, $|A| \geq {j+1}$. This shows that, when $|A|=j$, $F^{j+1}_{G_\epsilon}(A)=0$. Therefore $\mathds{1}_{|V_{G_{\epsilon}}| = k} \sum_{A \in \mathcal{P}(V_{G_{\epsilon}}) ; |A|=j} F^{j+1}_{G_{\epsilon}}(A) = 0$ in any case on $C_1(\epsilon)$ under $\mathbb{P}_m$. Taking expectation $\mathbb{E}_m[.]$ we get 
\begin{eqnarray}
\forall k \geq 1, m \neq k+1, j \in \{ 1,..., k \wedge (m-1) \}, \ E_1(m,k,j+1,j,\epsilon) = 0. \label{asymptcoeff19new}
\end{eqnarray}
\eqref{asymptcoeff14new} now follows from \eqref{asymptcoeff3new} together with \eqref{asymptcoeff16new}, \eqref{asymptcoeff19new}, and \eqref{asymptcoeff20new}. 

Now let $k \geq 1, m \notin \{k/2, k+1\}, j \in \{ 1,..., k \wedge m \}$. On $C_1(\epsilon)$ there is exactly one transition on $[0,\epsilon]$. Under $\mathbb{P}_m$, if this transition is a coalescence, then $|V_{G_{\epsilon}}| = m-1 \neq k$. If this transition is a multiple branching, then $|V_{G_{\epsilon}}| = 2m \neq k$. Therefore $\mathds{1}_{|V_{G_{\epsilon}}| = k} \sum_{A \in \mathcal{P}(V_{G_{\epsilon}}) ; |A|=j} F^{j}_{G_{\epsilon}}(A) = 0$ in any case on $C_1(\epsilon)$ under $\mathbb{P}_m$. Taking expectation $\mathbb{E}_m[.]$ we get 
\begin{eqnarray}
\forall k \geq 1, m \notin \{k/2, k+1\}, j \in \{ 1,..., k \wedge m \}, \ E_1(m,k,j,j,\epsilon) = 0. \label{asymptcoeff28new}
\end{eqnarray}
%EN FAIT, LE CAS OU ON A UN SAUT NE DOIS PAS ETRE ELIMINE: $R^{k/2}_t(j,j)$ INTERVIENT DANS L'EQUATION (AVANT, QUAND ON NE REGARDAIT PAS LES $k$, IL NE FAISANT QUE CHANGER LE COEFFICIENT DU TERME $R_t(j,j)$ MAIS MAINTENANT IL DONNE UN TERME EN PLUS (CE QUI NE CHANGE RIEN AU FAIT QUE LE SYSTEME PERMET DE CALCULER RECURSIVEMENT LES COEFFICIENTS LIMITES MAIS DU COUP IL FAIT METTRE A JOUR L'EQUA DIFF. 
\eqref{asymptcoeff15new} now follows from \eqref{asymptcoeff3new} together with \eqref{asymptcoeff17new}, \eqref{asymptcoeff28new}, and \eqref{asymptcoeff20new}. Also, \eqref{asymptcoeff26new} follows from \eqref{asymptcoeff3new} together with \eqref{asymptcoeff30new}, \eqref{asymptcoeff28new}, and \eqref{asymptcoeff20new}. 

%We now study $E_0({k/2},k,i,j,\epsilon)$. By construction of $(G_{\beta})_{\beta \geq 0}$ and of $F^i_{G_t}(\cdot)$ we have, on $C_0(\epsilon)$, that $G_\epsilon = G_0$ and $F^i_{G_\epsilon}(\cdot) = F^i_{G_0}(\cdot)$ and, still by construction of $F^i_{G_t}(\cdot)$, $F^i_{G_0}(A) = \mathds{1}_{A=\{ L_1,..., L_i \}}$. This shows that $\mathds{1}_{C_0(\epsilon)} \sum_{A \in \mathcal{P}(V_{G_{\epsilon}}) ; |A|=j} F^i_{G_{\epsilon}}(A) = 0$ if $j \neq i$. Taking expectation we get 
%\begin{eqnarray}
%E_0({k/2},k,i,j,\epsilon) = 0. \label{asymptcoeff4new}
%\end{eqnarray}

We now study $E_1({k/2},k,i,j,\epsilon)$ where $k \geq 2$ is even, $j \in \{ 1,..., k \}, i \in \{ 1,...,\lceil j/2 \rceil -1 \}$. On $C_1(\epsilon)$ there is exactly one transition on $[0,\epsilon]$. Under $\mathbb{P}_{k/2}$, if this transition is a coalescence, then $|V_{G_{\epsilon}}| = (k/2)-1 \neq k$ so $\mathds{1}_{|V_{G_{\epsilon}}| = k} \sum_{A \in \mathcal{P}(V_{G_{\epsilon}}) ; |A|=j} F^{i}_{G_{\epsilon}}(A) = 0$. If this transition is a multiple branching then \eqref{defFjump} shows that, for any $A \in \mathcal{P}(V_{G_{\epsilon}})$, $F^{i}_{G_\epsilon}(A)$ is proportional to $F^{i}_{G_0}(P(A))$. By definition of $F^{i}_{\cdot}(\cdot)$, $F^{i}_{G_0}(P(A)) = \mathds{1}_{P(A) = \{L_1,..., L_{i} \}}$. Therefore $F^{i}_{G_\epsilon}(A) \neq 0$ implies $P(A) = \{L_1,..., L_{i} \}$ so $|A| \leq 2i$, since $A$ can only contain lines in $V_{G_{\epsilon}}$ that are children of the lines $L_1,..., L_{i}$ and there are $2i$ such children in $V_{G_{\epsilon}}$. Here we have $i \in \{ 1,...,\lceil j/2 \rceil -1 \}$ so $2i<j$. In particular, $F^{i}_{G_\epsilon}(A) = 0$ whenever $|A|=j$. Therefore $\mathds{1}_{|V_{G_{\epsilon}}| = k} \sum_{A \in \mathcal{P}(V_{G_{\epsilon}}) ; |A|=j} F^{i}_{G_{\epsilon}}(A) = 0$ in any case on $C_1(\epsilon)$ under $\mathbb{P}_{k/2}$. Taking expectation $\mathbb{E}_{k/2}[.]$ we get that $E_1(k/2,k,i,j,\epsilon)=0$. When $i \in \{ 1,...,\lceil j/2 \rceil -1 \}$ we have $2i < j$ so $\tau(i,j)=0$ by Definition \ref{defcoefedo}. Therefore 
\begin{eqnarray}
\forall k \ \text{even} \ \geq 2, j \in \{ 1,..., k \}, i \in \{ 1,...,\lceil j/2 \rceil -1 \}, \ \lim_{\epsilon \rightarrow 0} \frac1{\epsilon} E_{1}({k/2},k,i,j,\epsilon) = \tau(i,j). \label{asymptcoeff31new}
\end{eqnarray}

We now study $E_1({k/2},k,i,j,\epsilon)$ where $k \geq 2$ is even, $j \in \{ 1,..., k \}, i \in \{ \lceil j/2 \rceil,..., j \wedge k/2 \}$. Note that we have 
\begin{align}
\mathds{1}_{C_1(\epsilon)} = \mathds{1}_{N_{mult}(\epsilon) = 0, N_{coal}(\epsilon) = 1} + \mathds{1}_{N_{mult}(\epsilon) = 1, N_{coal}(\epsilon) = 0}. \label{decompc1h}
\end{align}
Decomposing $\mathds{1}_{C_1(\epsilon)}$ in this way in the expectation defining $E_1({k/2},k,i,j,\epsilon)$ we get that $E_1({k/2},k,i,j,\epsilon)$ can be written as the sum of two expectations $E_{1,C}({k/2},k,i,j,\epsilon)$ (containing $\mathds{1}_{N_{mult}(\epsilon) = 0, N_{coal}(\epsilon) = 1}$) and $E_{1,J}({k/2},k,i,j,\epsilon)$ (containing $\mathds{1}_{N_{mult}(\epsilon) = 1, N_{coal}(\epsilon) = 0}$). On $C_1(\epsilon)$, if the transition of the E-ASG on $[0,\epsilon]$ is a coalescence, then $|V_{G_{\epsilon}}| = (k/2) - 1$ so the indicator $\mathds{1}_{|V_{G_{\epsilon}}| = k}$ appearing in the expectation $E_{1,C}({k/2},k,i,j,\epsilon)$ is null. Therefore 
%$|D(\{ L_1,..., L_i \})| \in \{i-1,i\}$ so $|D(\{ L_1,..., L_i \})| < j$. In particular the set $\{ A \in D(\{ L_1,..., L_i \}) ; |A|=j \}$ is empty. Note that by construction of $F^i_{G_t}(\cdot)$ we have $F^i_{G_t}(A) = 0$ if $A \notin D(\{ L_1,..., L_i \})$. Therefore $\sum_{A \in \mathcal{P}(V_{G_{\epsilon}}) ; |A|=j} F^i_{G_{\epsilon}}(A) = 0$ so 
\begin{eqnarray}
E_{1,C}({k/2},k,i,j,\epsilon) = 0. \label{asymptcoeff8new}
\end{eqnarray}
Recall that $(S_n)_{n \geq 1}$ denotes the sequence of weights associated with transitions of the E-ASG. On $C_1(\epsilon)$, if the transition of the E-ASG on $[0,\epsilon]$ is a multiple branching, then the indicator $\mathds{1}_{|V_{G_{\epsilon}}| = k}$ appearing in the expectation $E_{1,J}({k/2},k,i,j,\epsilon)$ equals $1$. Applying Lemma \ref{combjumpeasg} with our choice of $k,i,j$ and with $m=k/2$ we get 
\begin{align*}
%& \mathds{1}_{|V_{G_{\epsilon}}| = k, N_{mult}(\epsilon) = 1, N_{coal}(\epsilon) = 0} \times \sum_{A \in \mathcal{P}(V_{G_{\epsilon}}) ; |A|=j} F^i_{G_{\epsilon}}(A) \\
E_{1,J}({k/2},k,i,j,\epsilon) = \mathbb{E}_{k/2} \left [ \mathds{1}_{N_{mult}(\epsilon) = 1, N_{coal}(\epsilon) = 0} \binom{i}{j-i} (1+S_1)^{2i-j} \times (-S_1)^{j-i} \right ]. 
\end{align*}
Note from Section \ref{enlargedasgdefnot} and Definition \ref{defquenchedeasg} that, conditionally on the event in the indicator function, $S_1$ is distributed as $\nu(\cdot)/\lambda$. 
%The collection of the values of weights of branchings is independent from $((T^J_n)_{n \geq 1},(T^C_n)_{n \geq 1})$ where 
%Let $(T^J_n)_{n \geq 1}$ and $(T^C_n)_{n \geq 1}$ denote respectively the sequences of times of multiple branchings and coalescences of $(G_{\beta})_{\beta \geq 0}$. 
Combining with Definition \ref{defcoefedo} we get that the above equals $\mathbb{P}_{k/2} ( N_{mult}(\epsilon) = 1, N_{coal}(\epsilon) = 0 ) \lambda^{-1} \tau(i,j)$. 
%\begin{align*}
%%& \mathbb{E} \left [ \mathds{1}_{N_{mult}(h) = 1, N_{coal}(h) = 0} \sum_{A \in \mathcal{P}(V_{G_{h}}) ; |A|=j} F^i_{G_{h}}(A) \big | |V_{G_{0}}| = k \right ] \\
%E_{1,J}({k/2},k,i,j,h) = 
%%& \mathbb{P}_{k/2} \left ( N_{mult}(h) = 1, N_{coal}(h) = 0 \right ) \binom{i}{j-i} \mathbb{E} \left [ (1+S_1)^{2i-j} \times (-S_1)^{j-i} \right ] \\ 
%\mathbb{P}_{k/2} \left ( N_{mult}(h) = 1, N_{coal}(h) = 0 \right ) \lambda^{-1} \tau(i,j). 
%%\mathbb{P}_{k/2} \left ( T^J_1 \leq h < T^J_2, N_{coal}(0,T^J_1) = 0, N_{coal}(T^J_1,h) = 0 \right ) \lambda^{-1} \tau(i,j). 
%\end{align*}
Then, using \eqref{esteasg1br} from Lemma \ref{esteasg} we get 
\[ \lim_{\epsilon \rightarrow 0} \frac1{\epsilon} E_{1,J}(k/2,k,i,j,\epsilon) = \tau(i,j). \]
Combining with \eqref{asymptcoeff8new} we get 
\begin{eqnarray}
\forall k \ \text{even} \ \geq 2, j \in \{ 1,..., k \}, i \in \{ \lceil j/2 \rceil,..., j \wedge k/2 \}, \ \lim_{\epsilon \rightarrow 0} \frac1{\epsilon} E_{1}({k/2},k,i,j,\epsilon) = \tau(i,j). \label{asymptcoeff7new}
\end{eqnarray}
\eqref{asymptcoeff1new} now follows from \eqref{asymptcoeff3new} together with \eqref{asymptcoeff17new}, \eqref{asymptcoeff31new}, \eqref{asymptcoeff7new}, and \eqref{asymptcoeff20new}. 

We now study $E_1(k+1,k,j+1,j,\epsilon)$ where $k \geq 1, j \in \{ 1,..., k \}$. Using \eqref{decompc1h} in the expectation defining $E_1(k+1,k,j+1,j,\epsilon)$ we get that the latter is the sum of two expectations $E_{1,C}(k+1,k,j+1,j,\epsilon)$ (containing $\mathds{1}_{N_{mult}(\epsilon) = 0, N_{coal}(\epsilon) = 1}$) and $E_{1,J}(k+1,k,j+1,j,\epsilon)$ (containing $\mathds{1}_{N_{mult}(\epsilon) = 1, N_{coal}(\epsilon) = 0}$). Under $\mathbb{P}_{k+1}$, on $C_1(\epsilon)$, if the transition of the E-ASG on $[0,\epsilon]$ is a multiple branching, then $|V_{G_{\epsilon}}| = 2(k+1)$ so the indicator $\mathds{1}_{|V_{G_{\epsilon}}| = k}$ appearing in the expectation $E_{1,J}(k+1,k,j+1,j,\epsilon)$ is null. Therefore 
\begin{eqnarray}
E_{1,J}(k+1,k,j+1,j,\epsilon) = 0. \label{asymptcoeff22new}
\end{eqnarray}
%Assume now that $k = j+1$. On $C_1(\epsilon)$, if the transition of the E-ASG on $[0,\epsilon]$ is a multiple branching, then by \eqref{defFjump} we have that for any $A \in \mathcal{P}(V_{G_{\epsilon}})$, 
%\[ F^i_{G_\epsilon}(A) = F^i_{G_0}(P(A)) \times (1-S_{1} \mathds{1}_{S_{1} > 0})^{\alpha(A)} \times (-S_{1} \mathds{1}_{S_{1} < 0})^{\beta(A)} \times S_{1}^{\gamma(A)}, \]
%and recall that, by construction $F^i_{G_0}(P(A)) = 1$ if $P(A) = V_{G_{0}}$ and $F^i_{G_0}(P(A)) = 0$ otherwise. Therefore, if $F^i_{G_\epsilon}(A) \neq 0$ then $P(A) = V_{G_{0}}$ so, in particular, $|A| \geq k$. This shows that, when $|A|=j=k-1$, $F^i_{G_\epsilon}(A)=0$. We get that $\mathds{1}_{N_{mult}(\epsilon) = 1, N_{coal}(\epsilon) = 0} \sum_{A \in \mathcal{P}(V_{G_{\epsilon}}) ; |A|=j} F^i_{G_{\epsilon}}(A) = 0$ almost surely. 
%RAJOUTER LES $k$ ET ON POURRA SIMPLIFIER CET ARGUMENT (COMME DEJA FAIT PLUS HAUT). NB : ON A COPIE CET ARGUMENT POUR JUSTIFIER LA NULLITE D'UN TERME $E_{1}(m,k,i,j,\epsilon)$ PLUS HAUT. 
%Therefore, 
%\begin{eqnarray}
%E_{1,J}(m,k,i,j,\epsilon) = 0. \label{asymptcoeff9new}
%\end{eqnarray}
Under $\mathbb{P}_{k+1}$, on $C_1(\epsilon)$, if the transition of the E-ASG on $[0,\epsilon]$ is a coalescence, then the indicator $\mathds{1}_{|V_{G_{\epsilon}}| = k}$ appearing in the expectation $E_{1,J}(k+1,k,j+1,j,\epsilon)$ equals $1$. 
%$\{ A \in \mathcal{P}(V_{G_{\epsilon}}) ; |A|=j \} = \{ L_1,..., L_{j+1} \}$ so, 
Using \eqref{defFcoal} and that, by definition, $F^{j+1}_{G_0}(B) = \mathds{1}_{B=\{ L_1,..., L_{j+1} \}}$, we have 
\begin{align*}
\sum_{A \in \mathcal{P}(V_{G_{\epsilon}}) ; |A|=j} F^{j+1}_{G_\epsilon}(A) & = \sum_{B \in \mathcal{P}(V_{G_0}) ; |D(B)| = j} F^{j+1}_{G_0}(B) \\
& = \sum_{B \in \mathcal{P}(V_{G_0}) ; |D(B)| = j} \mathds{1}_{B=\{ L_1,..., L_{j+1} \}} = \mathds{1}_{|D(\{ L_1,..., L_{j+1} \})|=j}. 
\end{align*}
%and the fact that  $D(V_{G_0}) = V_{G_{\epsilon}}$ (so that one of the terms in the last sum is indeed equal to $F^i_{G_0}(V_{G_0})$). 
We thus get 
%that $\mathbb{P}_{k+1}$-almost surely, 
%\[ \mathds{1}_{|V_{G_{\epsilon}}| = k, N_{mult}(\epsilon) = 0, N_{coal}(\epsilon) = 1} \sum_{A \in \mathcal{P}(V_{G_{\epsilon}}) ; |A|=j} F^{j+1}_{G_{\epsilon}}(A) = \mathds{1}_{N_{mult}(\epsilon) = 0, N_{coal}(\epsilon) = 1, |D(\{ L_1,..., L_{j+1} \})|=j}. \] 
%Taking expectation $\mathbb{E}_{k+1}[.]$ we obtain 
\begin{align*}
E_{1,C}(k+1,k,j+1,j,\epsilon) & = \mathbb{P}_{k+1} \left ( N_{mult}(\epsilon) = 0, N_{coal}(\epsilon) = 1, |D(\{ L_1,..., L_{j+1} \})|=j \right ). 
%& = \mathbb{P}_{k+1} \left ( T^J_1 > \epsilon, N_{coal}(\epsilon) = 1, |D(\{ L_1,..., L_{j+1} \})|=j \right ). 
\end{align*}
$|D(\{ L_1,..., L_{j+1} \})|=j$ means that the two lines involved in the coalescence both belong to $\{ L_1,..., L_{j+1} \}$. Conditionally on $\{N_{mult}(\epsilon) = 0, N_{coal}(\epsilon) = 1\}$, the pair of lines involved is chosen uniformly among the $(k+1)k/2$ possible pairs of lines. In particular, the two lines involved belong to $\{ L_1,..., L_{j+1} \}$ with probability $((j+1)j/2)/((k+1)k/2)$. We thus get 
\[ E_{1,C}(k+1,k,j+1,j,\epsilon) = \mathbb{P}_{k+1} ( N_{mult}(\epsilon) = 0, N_{coal}(\epsilon) = 1) \times (j+1)j/(k+1)k. \]
Then, using \eqref{esteasg1coal} from Lemma \ref{esteasg} we get 
\[ \lim_{\epsilon \rightarrow 0} \frac1{\epsilon} E_{1,C}(k+1,k,j+1,j,\epsilon) = (j+1)j = \tau(j+1,j), \]
where the last equality is from Definition \ref{defcoefedo}. Combining with \eqref{asymptcoeff22new} we get 
\begin{eqnarray}
\forall k \geq 1, j \in \{ 1,..., k \}, \ \lim_{\epsilon \rightarrow 0} \frac1{\epsilon} E_{1}(k+1,k,j+1,j,\epsilon) = \tau(j+1,j). \label{asymptcoeff24new}
\end{eqnarray}
\eqref{asymptcoeff23new} now follows from \eqref{asymptcoeff3new} together with \eqref{asymptcoeff17new}, \eqref{asymptcoeff24new}, and \eqref{asymptcoeff20new}. 

Reasoning similarly as in the proof of \eqref{asymptcoeff24new} we can prove that 
\begin{eqnarray}
\forall k \geq 1, j \in \{ 1,..., k \}, \ \lim_{\epsilon \rightarrow 0} \frac1{\epsilon} E_{1}(k+1,k,j,j,\epsilon) = (k+1)k - j(j-1) = e_{k,j}. \label{asymptcoeff25new}
\end{eqnarray}
\eqref{asymptcoeff27new} now follows from \eqref{asymptcoeff3new} together with \eqref{asymptcoeff17new}, \eqref{asymptcoeff25new}, and \eqref{asymptcoeff20new}. 

\end{proof}

%CORRIGER L'EQUA DIFF ET LA PREUVE. MODIFIER LES NOTATIONS DU FAIT QU'ON A $4$ INDICES ET QU'ON A FAIT LA PREUVE AVEC DES NOTATIONS A $3$ INDICES. 

\subsection{A system of differential equations: proof of Theorem \ref{equadiffcoeffnew}} \label{diffeqsyst}

We can now derive rigorously the system of differential equations satisfied by coefficients $R^{m,k}_t(i,j)$. Unfortunately, we cannot use the semigroup property (Proposition \ref{semigroupprop}) to differentiate $R^{m,k}_t(i,j)$ with respect to $t$. This is due to the difficulty in working with infinitely many terms that are not all positive and to the lack of uniformity in the behavior of those terms. 
What we do is establishing a decomposition \eqref{decompladernew} of $\epsilon^{-1}(R^{m,k}_{t+\epsilon}(i,j)-R^{m,k}_t(i,j))$ into a sum of finitely many relevant terms and a remainder term. For the first ones, we proceed as in the proof of Proposition \ref{semigroupprop} to show, in \eqref{reldesgnew} below, that they can be can be written in terms of sums of coefficients $R^{\cdot,\cdot}_{\cdot}(\cdot,\cdot)$ that can be dealt with thanks to Lemmas \ref{asymptcoeffnew} and \ref{sgwelldef}. Lemma \ref{atleast2coal} shows that the integrand, inside the expectation defining the remainder term, is non-zero only on the event where there are at least two coalescences on $[t,t+\epsilon]$. Lemma \ref{esteasg} allows to neglect the probability of this event and, thanks to \eqref{exprL(t,t+h)new} below, the integrand can be re-written in terms of finitely many terms for which we have deterministic bounds (provided in particular by Lemma \ref{sgwelldef}). 
Recall that in this subsection and in the previous one \textbf{we assume that we are in the particular case $\sigma=0$} (i.e. without single branchings) in order to make the proof lighter. This does not change anything about the idea or the difficulty of the proof.

Let us fix $m,k\geq 1, i \in \{ 1,...,m \}, j \in \{ 1,...,k \}$, $t \geq 0$, and $\epsilon > 0$. By definition of $R^{\cdot,\cdot}_{\cdot}(\cdot,\cdot)$, $\frac{d}{dt} R^{m,k}_t(i,j)$ equals 
\[ \lim_{\epsilon \rightarrow 0} \frac1{\epsilon} \mathbb{E}_m \left [ \mathds{1}_{|V_{G_{t+\epsilon}}| = k} \sum_{A \in \mathcal{P}(V_{G_{t+\epsilon}}) ; |A|=j} F^i_{G_{t+\epsilon}}(A) - \mathds{1}_{|V_{G_{t}}| = k} \sum_{B \in \mathcal{P}(V_{G_{t}}) ; |B|=j} F^i_{G_{t}}(B) \right ], \]
provided the limit exists. By \eqref{markpropf1} applied with $r_1 = t+\epsilon, r_2 = t$, the term inside the expectation equals: 
\begin{align}
%& \mathds{1}_{|V_{G_{t+\epsilon}}| = k} \sum_{A \in \mathcal{P}(V_{G_{t+\epsilon}}) ; |A|=j} F^i_{G_{t+\epsilon}}(A) - \mathds{1}_{|V_{G_{t}}| = k} \sum_{B \in \mathcal{P}(V_{G_{t}}) ; |B|=j} F^i_{G_{t}}(B) \label{difft+h-tnew} \\ 
& \mathds{1}_{|V_{G_{t+\epsilon}}| = k} \sum_{A \in \mathcal{P}(V_{G_{t+\epsilon}}) ; |A|=j} \left ( \sum_{B \in \mathcal{P}(V_{G_{t}})} F^i_{G_{t}}(B) f_{G_t, G_{t+\epsilon}}(B,A) \right ) - \mathds{1}_{|V_{G_{t}}| = k} \sum_{B \in \mathcal{P}(V_{G_{t}}) ; |B|=j} F^i_{G_{t}}(B) \nonumber \\
= & \sum_{B \in \mathcal{P}(V_{G_{t}})} F^i_{G_{t}}(B) \left ( \mathds{1}_{|V_{G_{t+\epsilon}}| = k} \sum_{A \in \mathcal{P}(V_{G_{t+\epsilon}}) ; |A|=j} f_{G_t, G_{t+\epsilon}}(B,A) \right ) - \mathds{1}_{|V_{G_{t}}| = k} \sum_{B \in \mathcal{P}(V_{G_{t}}) ; |B|=j} F^i_{G_{t}}(B) \nonumber \\
= & \sum_{l=1}^{j+1} \sum_{B \in \mathcal{P}(V_{G_{t}}) ; |B|=l} F^i_{G_{t}}(B) \left ( - \mathds{1}_{l=j, |V_{G_{t}}| = k} + \mathds{1}_{|V_{G_{t+\epsilon}}| = k} \sum_{A \in \mathcal{P}(V_{G_{t+\epsilon}}) ; |A|=j} f_{G_t, G_{t+\epsilon}}(B,A) \right ) \nonumber \\
+ & \sum_{B \in \mathcal{P}(V_{G_{t}}) ; |B|\geq j+2} F^i_{G_{t}}(B) \left ( \mathds{1}_{|V_{G_{t+\epsilon}}| = k} \sum_{A \in \mathcal{P}(V_{G_{t+\epsilon}}) ; |A|=j} f_{G_t, G_{t+\epsilon}}(B,A) \right ) \nonumber \\
=: & \left ( \sum_{l=1}^{j+1} M(l,t,t+\epsilon) \right ) + L(t,t+\epsilon). \label{defmltt+h}
\end{align}
Let us also define 
\[ \tilde M(l,t,t+\epsilon) := \sum_{B \in \mathcal{P}(V_{G_{t}}) ; |B|=l} F^i_{G_{t}}(B) \left ( \mathds{1}_{|V_{G_{t+\epsilon}}| = k} \sum_{A \in \mathcal{P}(V_{G_{t+\epsilon}}) ; |A|=j} f_{G_t, G_{t+\epsilon}}(B,A) \right ), \]
and note that $\tilde M(l,t,t+\epsilon) = M(l,t,t+\epsilon)$ for $l \neq j$. The same reasoning as above yields 
\begin{eqnarray}
\mathds{1}_{|V_{G_{t+\epsilon}}| = k} \sum_{A \in \mathcal{P}(V_{G_{t+\epsilon}}) ; |A|=j} F^i_{G_{t+\epsilon}}(A) = \left ( \sum_{l=1}^{j+1} \tilde M(l,t,t+\epsilon) \right ) + L(t,t+\epsilon). \label{exprL(t,t+h)new}
\end{eqnarray}
We now justify that the expectations of the terms $M(\cdot,t,t+\epsilon)$ and $L(t,t+\epsilon)$ are well-defined. 
By \eqref{markpropf3} we have almost surely $|f_{G_t, G_{t+\epsilon}}(B,A)| \leq |A|^{|A|}$. Moreover, on $\{|V_{G_{t+\epsilon}}| = k\}$, there are $\binom{k}{j}$ sets $A \in \mathcal{P}(V_{G_{t+\epsilon}})$ with cardinality $j$. Using also \eqref{borneFps} to bound $|\sum_{B \in \mathcal{P}(V_{G_{t}}) ; |B|=l} F^i_{G_{t}}(B)|$ we get that for all $l \in \{ 1,..., j+1 \}$ we have: 
\begin{align}
| M(l,t,t+\epsilon) | \leq \frac{l^l}{l!} \left ( 1 + \binom{k}{j} j^j \right ) |V_{G_{t}}|^{l} \ \ \ \text{and} \ \ \ | \tilde M(l,t,t+\epsilon) | & \leq \frac{l^l j^j}{l!} \binom{k}{j} |V_{G_{t}}|^{l}. \label{majonbfinitermesnew} 
\end{align}
%\leq 2 \frac{l^l j^j}{l! j!} 2^{(l+j) N_{mult}(t+\epsilon)}
%where $N_{mult}(t+\epsilon)$ is the number of jumps of the L\'evy environment on $[0,t+\epsilon]$. 
%Note that similarly 
%\begin{align}
%| \tilde M(l,t,t+\epsilon) | & \leq \frac{l^l j^j}{l!} \binom{k}{j} |V_{G_{t}}|^{l}. \label{majonbfinitermestildenew} 
%%& \leq \frac{l^l j^j}{l! j!} 2^{(l+j) N_{mult}(t+\epsilon)}. \nonumber 
%\end{align}
$\mathbb{E}_m[| M(l,t,t+\epsilon) | ] < \infty$ and $\mathbb{E}_m[| \tilde M(l,t,t+\epsilon) | ] < \infty$ now follow from \eqref{majonbfinitermesnew} together with \eqref{bornemomentGtnew}. The absolute value of the left-hand side of \eqref{exprL(t,t+h)new} is deterministically bounded by $(jk)^j/j!$, according to \eqref{borneFps}, and the expectation of $|\sum_{l=1}^{j+1} \tilde M(l,t,t+\epsilon)|$ is well-defined as we have just shown. We thus get from \eqref{exprL(t,t+h)new} that $\mathbb{E}_m[| L(t,t+\epsilon) | ] < \infty$. Therefore, 
\begin{eqnarray}
\frac{d}{dt} R^{m,k}_t(i,j) = \sum_{l=1}^{j+1} \lim_{\epsilon \rightarrow 0} \frac1{\epsilon} \mathbb{E}_m \left [ M(l,t,t+\epsilon) \right ] + \lim_{\epsilon \rightarrow 0} \frac1{\epsilon} \mathbb{E}_m \left [ L(t,t+\epsilon) \right ], \label{decompladernew}
\end{eqnarray}
provided the limits exist. We are thus left to prove the following four statements: 
\begin{align}
\lim_{\epsilon \rightarrow 0} \frac1{\epsilon} \mathbb{E}_m \left [ L(t,t+\epsilon) \right ] =& 0, \label{decomplader2new} \\
\forall l \in \{ 1,..., j-1 \}, \ \lim_{\epsilon \rightarrow 0} \frac1{\epsilon} \mathbb{E}_m \left [ M(l,t,t+\epsilon) \right ] =& \mathds{1}_{k \ \text{even}, l \leq k/2} \tau(l,j) R^{m,k/2}_t(i,l), \label{decomplader1new} \\
\lim_{\epsilon \rightarrow 0} \frac1{\epsilon} \mathbb{E}_m \left [ M(j+1,t,t+\epsilon) \right ] =& \tau(j+1,j) R^{m,k+1}_t(i,j+1), \label{decomplader4new} \\
\lim_{\epsilon \rightarrow 0} \frac1{\epsilon} \mathbb{E}_m \left [ M(j,t,t+\epsilon) \right ] = -d_k R^{m,k}_t(i,j) + \mathds{1}_{k \ \text{even}, j \leq k/2} & \tau(j,j) R^{m,k/2}_t(i,j) + e_{k,j} R^{m,k+1}_t(i,j). \label{decomplader3new}
\end{align}
Once \eqref{decomplader2new}, \eqref{decomplader1new}, \eqref{decomplader4new}, and \eqref{decomplader3new} are proved Theorem \ref{equadiffcoeffnew} will follow. 

Let us first justify \eqref{decomplader2new}. We let $C(t,t+\epsilon)$ denote the event where there are at least two coalescences of the E-ASG on $(t,t+\epsilon]$. The set $\{ B \in \mathcal{P}(V_{G_{t}}) ; |B|\geq j+2 \}$ is non-empty if and only if $|V_{G_{t}}| \geq j+2$. In this case, for $B \in \mathcal{P}(V_{G_{t}})$ with $|B|\geq j+2$, we see from Lemma \ref{atleast2coal} that the set $\{ A \in \mathcal{P}(V_{G_{t+\epsilon}}) ; |A|=j, f_{G_t, G_{t+\epsilon}}(B,A) \neq 0 \}$ is non-empty only on $C(t,t+\epsilon)$. Therefore $L(t,t+\epsilon) = \mathds{1}_{C(t,t+\epsilon)} L(t,t+\epsilon)$ and, using \eqref{exprL(t,t+h)new}, we get 
\[ L(t,t+\epsilon) = \mathds{1}_{C(t,t+\epsilon)} \left ( \mathds{1}_{|V_{G_{t+\epsilon}}| = k} \sum_{A \in \mathcal{P}(V_{G_{t+\epsilon}}) ; |A|=j} F^i_{G_{t+\epsilon}}(A) - \sum_{l=1}^{j+1} \tilde M(l,t,t+\epsilon) \right ). \]
Then, by \eqref{borneFps} and \eqref{majonbfinitermesnew} we have 
\begin{align*}
|L(t,t+\epsilon)| &\leq \mathds{1}_{C(t,t+\epsilon)} \left ( \frac{(jk)^j}{j!} + \sum_{l=1}^{j+1} \frac{l^l j^j}{l!} \binom{k}{j} |V_{G_{t}}|^{l} \right ) 
%& \leq \mathds{1}_{|V_{G_{t}}| \geq j+2} \mathds{1}_{C(t,t+\epsilon)} \frac{(j+2)(j+1)^{j+1}(jk)^{j}}{j!} |V_{G_{t}}|^{j+1} \\
\leq C_{k,j} \mathds{1}_{C(t,t+\epsilon)} |V_{G_{t}}|^{j+1}, 
\end{align*}
where we have set $C_{k,j} := (j+2)(j+1)^{j+1}(jk)^{j}/j!$. Taking expectation: 
\begin{align*}
\mathbb{E}_m [|L(t,t+\epsilon)| ] & \leq C_{k,j} \mathbb{E}_m \left [ \mathds{1}_{C(t,t+\epsilon)} |V_{G_{t}}|^{j+1} \right ] \nonumber \\
& = C_{k,j} \sum_{r \geq 1} r^{j+1} \mathbb{P}_m \left ( C(t,t+\epsilon) \big | |V_{G_{t}}| = r \right ) \times \mathbb{P}_m \left ( |V_{G_{t}}| = r \right ). 
%\label{majoL(t,t+h)1new} 
\end{align*}
By the Markov property, $\mathbb{P}_m ( C(t,t+\epsilon) | |V_{G_{t}}| = r)=\mathbb{P}_r ( C(0,\epsilon) )$ and, by \eqref{esteasg>2} from Lemma \ref{esteasg}, the later is smaller than $16 r^4 (1+\lambda)^2 \epsilon^2$. Therefore, $\mathbb{E}_m [|L(t,t+\epsilon)| ]$ is smaller than 
\begin{align*}
\epsilon^2 \tilde C_{k,j} \sum_{r \geq 1} r^{j+5} \mathbb{P}_m \left ( |V_{G_{t}}| = r \right ) = \epsilon^2 \tilde C_{k,j} \mathbb{E}_m \left [ |V_{G_{t}}|^{j+5} \right ] \leq \frac{\epsilon^2 \tilde C_{k,j}}{\pi(m)} \sum_{r \geq 1} r^{j+5} \pi(r), 
%\leq \epsilon^2 m^{j+5} C_{k,j} e^{15 \lambda \epsilon + (32 \times 2^j - 1)\lambda t}, 
\end{align*}
where we have set $\tilde C_{k,j} := 16 (1+\lambda)^2 C_{k,j}$ and where the last inequality comes from \eqref{bornemomentGtnew}. This concludes the proof of \eqref{decomplader2new}. 

%Let us now prove \eqref{decomplader1new}, \eqref{decomplader4new} and \eqref{decomplader3new}. 
For $l \in \{ 1,..., j+1 \}$, starting from the definition of $M(l,t,t+\epsilon)$ in \eqref{defmltt+h} and proceeding as in the proof of Proposition \ref{semigroupprop} (in Section \ref{behaviour023}), we can show that 
\begin{eqnarray}
\frac1{\epsilon} \mathbb{E}_m [ M(l,t,t+\epsilon) ] = \sum_{r=l}^{\infty} \frac1{\epsilon} (R^{r,k}_\epsilon(l,j) - \mathds{1}_{l=j, r=k}) R^{m,r}_t(i,l). \label{reldesgnew}
\end{eqnarray}

Let us fix $l \in \{ 1,..., j-1 \}$. From \eqref{reldesgnew} we get 
%\begin{eqnarray}
\[ \frac1{\epsilon} \mathbb{E}_m [ M(l,t,t+\epsilon) ] = \mathds{1}_{k \ \text{even}, l \leq k/2} \frac1{\epsilon} R^{k/2,k}_\epsilon(l,j) R^{m,k/2}_t(i,l) + \sum_{r \geq l ; r \neq k/2} \frac1{\epsilon} R^{r,k}_\epsilon(l,j) R^{m,r}_t(i,l). \]
%\label{reldesg1new}
%\end{eqnarray}
%ATTENTION, \eqref{asymptcoeff1new} A ETE COUPE EN DEUX DONC IL FAUT AUSSI COUPER CE CAS LA EN $2$ ET UTILISER \eqref{asymptcoeff1new} ET LE TRUC (A RAJOUTER) QUI TRAITE LES CAS QUI NE SONT PLUS DANS \eqref{asymptcoeff1new}. OK, LA PREUVE DE \eqref{asymptcoeff1new} A ETE COUPEE EN DEUX ET \eqref{asymptcoeff1new} COUVRE A NOUVEAU TOUT. 
By \eqref{asymptcoeff1new}, the first term in the right-hand side converges to $\mathds{1}_{k \ \text{even}, l \leq k/2} \tau(l,j) R^{m,k/2}_t(i,l)$ when $\epsilon$ goes to $0$. According to \eqref{asymptcoeff13new}, \eqref{majonouveausgnew} and then to \eqref{bornemomentGtnew} we have 
\begin{align*}
\left | \sum_{r \geq l ; r \neq k/2} \frac1{\epsilon} R^{r,k}_\epsilon(l,j) R^{m,r}_t(i,l) \right | & \leq \epsilon \sum_{r \geq l ; r \neq k/2} 16 r^4 (1+\lambda)^2 \frac{(jk)^j}{j!} \frac{(rl)^l}{l!} \mathbb{P}_m \left ( |V_{G_{t}}| = r \right ) \\
%& = 16 \epsilon (1+\lambda)^2 \frac{(jk)^j}{j!} \frac{l^l}{l!} \sum_{r \geq l ; r \neq k/2} r^{l+4} \mathbb{P}_m \left ( |V_{G_{t}}| = r \right ) \\
& \leq 16 \epsilon (1+\lambda)^2 \frac{(jk)^j}{j!} \frac{l^l}{l!} \mathbb{E}_m \left [ |V_{G_{t}}|^{l+4} \right ] \\
& \leq 16 \epsilon (1+\lambda)^2 \frac{(jk)^j}{j!} \frac{l^l}{l!} \frac{1}{\pi(m)} \sum_{r \geq 1} r^{l+4} \pi(r) \underset{\epsilon \rightarrow 0}{\longrightarrow} 0. 
\end{align*}
%ATTENTION IL SEMBLE QUE LES $k^4 + 4 \lambda k^2$ SOIENT EN FAIT $r^4 + 4 \lambda r^2$ ET CA MODIFIE DONC UN PEU LE RESTE. OK
Therefore \eqref{decomplader1new} follows. Let us now apply \eqref{reldesgnew} with $l=j+1$. We get 
%\begin{eqnarray}
\[ \frac1{\epsilon} \mathbb{E}_m [ M(j+1,t,t+\epsilon) ] = \frac1{\epsilon} R^{k+1,k}_\epsilon(j+1,j) R^{m,k+1}_t(i,j+1) + \sum_{r \geq j+1 ; r \neq k+1} \frac1{\epsilon} R^{r,k}_\epsilon(j+1,j) R^{m,r}_t(i,j+1). \]
By \eqref{asymptcoeff23new}, the first term in the right-hand side converges to $\tau(j+1,j) R^{m,k+1}_t(i,j+1)$ when $\epsilon$ goes to $0$. The convergence of the second term to $0$ is proved exactly as in the proof of \eqref{decomplader1new} above, using this time \eqref{asymptcoeff14new} instead of \eqref{asymptcoeff13new}. \eqref{decomplader4new} follows. 

Let us now apply \eqref{reldesgnew} with $l=j$. We get 
\begin{align*}
\frac1{\epsilon} \mathbb{E}_m [ M(j,t,t+\epsilon) ] & = \frac1{\epsilon} \left ( R^{k,k}_\epsilon(j,j) - 1 \right ) R^{m,k}_t(i,j) + \mathds{1}_{k \ \text{even}, j \leq k/2} \frac1{\epsilon} R^{k/2,k}_\epsilon(j,j) R^{m,k/2}_t(i,j) \\
& + \frac1{\epsilon} R^{k+1,k}_\epsilon(j,j) R^{m,k+1}_t(i,j) + \sum_{r \geq j ; r \notin \{k/2,k,k+1\}} \frac1{\epsilon} R^{r,k}_\epsilon(j,j) R^{m,r}_t(i,j). 
\end{align*}
According to \eqref{asymptcoeff26new}, \eqref{asymptcoeff1new}, and \eqref{asymptcoeff27new}, the first, second, and third terms in the right-hand side converge to respectively $-d_k R^{m,k}_t(i,j)$, $\mathds{1}_{k \ \text{even}, j \leq k/2} \tau(j,j) R^{m,k/2}_t(i,j)$, and $e_{k,j} R^{m,k+1}_t(i,j)$ when $\epsilon$ goes to $0$. The convergence of the fourth term to $0$ is proved exactly as in the proof of \eqref{decomplader1new} above, using this time \eqref{asymptcoeff15new} instead of \eqref{asymptcoeff13new}. \eqref{decomplader3new} follows, which concludes the proof. 

%EN FAIT IL Y A UN TERME EN PLUS QU'ON AVAIT PAS VU : IL Y A DEUX TERMES D'ORDRE $k+1$, CELUI QUI CORRESPOND A UNE COALESCENCE D'UNE DES $j+1$ LIGNES CONSIDEREES ET CELUI QUI CORRESPOND A UNE COALESCENCE D'UNE DES $k=j-1$ AUTRES LIGNES. ENSUITE CA MARCHE QUAND MEME CAR LES $k$ ETATS DE L'ORDRE $k$ GENERENT $k$ EQUA TIONS, LE FAIT QUE LA SOMME VAUT $d_{k+1}$ EST UNE $k+1^{eme}$ EQUATION DONC L'ORDRE $k+1$ EST BIEN DETERMINE MEME SI C'EST UN PEU PLUS COMPLIQUE QUE CE QU'ON PENSAIT AU DEBUT. 

%\end{proof}

\subsection{Behavior at infinity: proof of Theorem \ref{cvcoeff}} \label{behaviourinfty}

%\begin{proof} of Theorem \ref{cvcoeff}

We first prove the existence of the limit in \eqref{cvcoeff1new} using a coupling argument. The idea is to take two independent realizations of the E-ASG and glue them together after a common bottleneck. It is shown in \eqref{majoespcolle} below that the probability that this common bottleneck hasn't occurred within time $t$ decays exponentially fast. Using the branching property from Lemma \ref{markpropf} we obtain, in \eqref{couplagedesfnew} below, that the functions $F^{\cdot}_{\cdot}(\cdot)$ are the same for the two realizations of the E-ASG after their common bottleneck. Combining this with \eqref{majoespcolle} and estimates from Lemma \ref{sgwelldef} we obtain \eqref{lesgestcauchynew}, which shows that $(R^{m,k}_t(i,j))_{t \geq 0}$ satisfies the Cauchy property as $t$ goes to infinity and is therefore convergent. 

Let $(G_{\beta})_{\beta \geq 0}$ start with $m \geq 1$ lines at $\beta=0$ and let $i \in \{1,...,m\}$. 
%From the dynamic of $(G_{\beta})_{\beta \geq 0}$ described in Section \ref{enlargedasg}, we see that 
Recall from Section \ref{enlargedasglcp} that the line counting process of $(G_{\beta})_{\beta \geq 0}$ is a Markov process for which the state $1$ is recurrent, which means that $(G_{\beta})_{\beta \geq 0}$ experiences infinitely many bottlenecks. Moreover the state $1$ can be reached with positive probability in any time $t>0$ from any starting position so the event $\{ |V_{G_{t}}|=1 \}$ has positive probability for any $t > 0$. On this event we denote by $L^G_t$ the single line of $G_t$. By \eqref{sommetotale=1} we have $\sum_{A \in \mathcal{P}(V_{G_{t}})} F^{i}_{G_{t}}(A) = 1$ almost surely so, in particular, this also holds on $\{ |V_{G_{t}}|=1 \}$ which induces 
\begin{eqnarray}
1 = \sum_{A \in \mathcal{P}(V_{G_{t}})} F^{i}_{G_{t}}(A) = F^{i}_{G_{t}}(\{L^G_t\}). \label{bottleneckeffectnew}
\end{eqnarray}
Thanks to this we are able, in the following, to build a coupling of two realizations of $(F^{i}_{G_{\beta}}(\cdot))_{{\beta} \geq 0}$ that remain identical after a common bottleneck. 

We fix $m_1, m_2 \geq 1, i_1 \in \{1,...,m_1\}, i_2 \in \{1,...,m_2\}, r > 0$. Let $(G^1_{\beta})_{{\beta} \geq 0}$, $(G^2_{\beta})_{{\beta} \geq 0}$, and $(\tilde G_{\beta})_{{\beta} \geq 0}$ be three independent realizations of the E-ASG with distribution $\mathbb{P}_{m_1}$, $\mathbb{P}_{m_2}$ and $\mathbb{P}_{1}$ respectively. 
%realizations of $(G_{\beta})_{\beta \geq 0}$ with $|V_{G^1_{0}}| = m_1, |V_{G^2_{0}}| = m_2$ and $|V_{\tilde G_{0}}| = 1$. 
We denote by $\mathbb{P}_{m_1, m_2,1}(\cdot)$ the law of the joint process $(G^1_{\beta}, G^2_{\beta}, \tilde G_{\beta})_{{\beta} \geq 0}$ and we set $\mathcal{T} := \inf \{ t \geq 0, \ |V_{G^1_{t}}| = |V_{G^2_{r+t}}| = 1 \}$. A couple of two independent realizations of the line counting process of the E-ASG is an irreducible Markov chain with stationary distribution $\pi \otimes \pi$ so it is recurrent. In particular the state $(1,1)$ is recurrent for the Markov chain $(V_{G^1_t}, V_{G^2_{r+t}})_{t \geq 0}$ so $\mathcal{T}$ is $\mathbb{P}_{m_1, m_2,1}$-almost surely finite. As a by-product of the proof of Lemma 5.2 of \cite{cordvech} (with $\theta=0$ and $\mu=\lambda \delta_1$), the expectation of the hitting time of a given state by the line counting process of the E-ASG is bounded independently from the initial state. In other words, for any $k \geq 1$ there is a finite constant $C_k > 0$ such that 
\[ \forall m \geq 1, \ \mathbb{E}_m [ \inf \{ t \geq 0, \ |V_{G_{t}}| = k \} ] < C_k. \]
It is not difficult to deduce that the same holds for a couple of two independent realizations of the line counting process of the E-ASG. In particular the expectation of the hitting time, by the Markov chain $(V_{G^1_t}, V_{G^2_{r+t}})_{t \geq 0}$, of the state $(1,1)$, is bounded by a constant $C$ that does not depend on the initial position. By successive applications of the Markov property at times $2C, 4C, 6C,...$ and Markov inequality we get $\forall n \geq 1, \mathbb{P}_{m_1,m_2,1} ( \mathcal{T} \geq 2nC ) \leq 2^{-n}$. Therefore, if we set $c := \log (2)/2C > 0$, we get 
\begin{eqnarray}
\forall m_1, m_2 \geq 1, \forall r,t \geq 0, \ \mathbb{P}_{m_1,m_2,1} (\mathcal{T} > t) \leq 2 e^{-c t}. \label{majoespcolle}
\end{eqnarray}

We define $(\hat G^1_{\beta})_{\beta \geq 0}$ as follows: $\hat G^1_{\beta}:=G^1_{\beta}$ if ${\beta} \leq \mathcal{T}$ and, for ${\beta} > \mathcal{T}$, $\hat G^1_{\beta}$ is obtained by \textit{gluing} $\tilde G_{{\beta}-\mathcal{T}}$ to $G^1_{\mathcal{T}}$. Let us define what we mean by gluing. 
%ON PEUT DIRE CE QUE CA VEUT DIRE EN TERME DE SHIFT ET DE PROJECTION SUR LES PREMIERES GENERATIONS? NON CAR ON DIT PLUS BAS CE QUE CA VEUT DIRE EN TERME DE SHIFT ET CA NE SERT A RIEN DE DIRE EN TERME DE PROJECTION. 
Let $N^1_{\mathcal{T}} := \mathsf{depth}(G^1_{\mathcal{T}})$, as defined in Section \ref{enlargedasgdefnot}. For $\beta > \mathcal{T}$ and $i \geq 0$, generation $i$ of $\hat G^1_{\beta}$ is generation $i$ of $G^1_{\mathcal{T}}$ if $i \leq N^1_{\mathcal{T}}$, and it is generation $i-N^1_{\mathcal{T}}$ of $\tilde G_{{\beta}-\mathcal{T}}$ if $i \geq N^1_{\mathcal{T}}$. 
%In other words we have $\pi_{N^1_{\mathcal{T}}}(\hat G^1_t)=G^1_{\mathcal{T}}$. 
Note that generation $N^1_{\mathcal{T}}$ of $G^1_{\mathcal{T}}$ (which contains only one line) is identified with generation $0$ of $\tilde G_{{\beta}-\mathcal{T}}$ (which also contains only one line). $(\hat G^2_{\beta})_{\beta \geq 0}$ is defined as follows: $\hat G^2_{\beta}:=G^2_{\beta}$ if ${\beta} \leq r+\mathcal{T}$ and, for ${\beta} > r+\mathcal{T}$, $\hat G^2_{\beta}$ is obtained by gluing $\tilde G_{{\beta}-(r+\mathcal{T})}$ to $G^2_{r+\mathcal{T}}$ in the way described above. 
%\[ \hat G^1_t:=\left\{\begin{array}{ll}
%            G^1_t &\text{if $t \leq T$},\\
%            \tilde G_{t-T} &\text{if $t > T$},\\
%            \end{array}\right. 
%            \ \ \ \hat G^2_t:=\left\{\begin{array}{ll}
%            G^2_t &\text{if $t \leq s+T$},\\
%            \tilde G_{t-(s+T)} &\text{if $t > s+T$}.\\
%            \end{array}\right. \]
%In other words, we identify the single line of $G^1_T$ with the single line of $\tilde G_0$ so that $(\hat G^1_t)_{t \geq 0}$ is $(G^1_t)_{t \geq 0}$ before instant $T$ and $(\hat G^1_t)_{t \geq 0}$ after instant $T$. Similarly we identify the single line of $G^2_{s+T}$ with the single line of $\tilde G_0$ so that $(\hat G^2_t)_{t \geq 0}$ is $(G^2_t)_{t \geq 0}$ before instant $s+T$ and $(\hat G^2_t)_{t \geq 0}$ after instant $s+T$. 
We see from the strong Markov property for the E-ASG, applied at times respectively $\mathcal{T}$ and $r+\mathcal{T}$, that $(\hat G^1_{\beta})_{\beta \geq 0}$ and $(\hat G^2_{\beta})_{\beta \geq 0}$ 
%are a realizations of $(G_{\beta})_{\beta \geq 0}$ starting with respectively $m_1$ and $m_2$ lines 
have distribution respectively $\mathbb{P}_{m_1}$ and $\mathbb{P}_{m_2}$. Moreover we see from the definitions of $(\hat G^1_{\beta})_{{\beta} \geq 0}$ and $(\hat G^2_{\beta})_{{\beta} \geq 0}$, and from the definition of the shifting in Section \ref{behaviour021}, that for any $t \geq \mathcal{T}$, 
%$(\hat G^1)^{T,\{ L^{\hat G^1}_T \}}_t = \tilde G_{t-T}$, and for any $t > s+T$, $(\hat G^2)^{s+T,\{ L^{\hat G^2}_{s+T} \}}_t = \tilde G_{t-(s+T)}$. 
\begin{eqnarray}
\hat G^1_t \setminus_{\{ L^{\hat G^1}_{\mathcal{T}} \}} \hat G^1_{\mathcal{T}} = \tilde G_{t-\mathcal{T}} = \hat G^2_{r+t} \setminus_{\{ L^{\hat G^2}_{r+\mathcal{T}} \}} \hat G^2_{r+\mathcal{T}}. \label{identaftercut}
\end{eqnarray}
Note also that $|V_{\hat G^1_{\mathcal{T}}}|=|V_{\hat G^2_{r+\mathcal{T}}}|=1$. 
%ATTENTION, EN FAIT CE NE SONT PAS LES GRAPHS QUI SONT EGAUX, CE SONT LES GRAPHES COUPES. OK
We consider $(\hat G^1_{\beta}, \hat G^2_{\beta})_{{\beta} \geq 0}$ and $u \geq \mathcal{T}$. Applying \eqref{markpropf1} to $F^{i_1}_{\hat G^1_{\cdot}}(\cdot)$ with $r_1 = u$ and $r_2 = \mathcal{T}$ we have 
\[ \forall A \in \mathcal{P}(V_{\hat G^1_{u}}) (=\mathcal{P}(V_{\tilde G_{u-\mathcal{T}}})), \ F^{i_1}_{\hat G^1_{u}}(A) = \sum_{B \in \mathcal{P}(V_{\hat G^1_{\mathcal{T}}})} F^{i_1}_{\hat G^1_{\mathcal{T}}}(B) f_{\hat G^1_{\mathcal{T}}, \hat G^1_u}(B,A). \]
Since $|V_{\hat G^1_{\mathcal{T}}}|=1$ we have $\mathcal{P}(V_{\hat G^1_{\mathcal{T}}}) = \{ \{ L^{\hat G^1}_{\mathcal{T}} \} \}$ (where $L^{\hat G^1}_{\mathcal{T}}$ is the single line in $V_{\hat G^1_{\mathcal{T}}}$, and is identified with $L^{\tilde G}_0$, the single line in $V_{\tilde G_0}$). 
%and any $A \in \mathcal{P}(V_{\hat G^1_{r}})$ of course belongs to $D_{r}(\{ L^{\hat G^1}_T \})$. 
We thus get 
\[ \forall A \in \mathcal{P}(V_{\hat G^1_{u}}) (=\mathcal{P}(V_{\tilde G_{u-\mathcal{T}}})), \ F^{i_1}_{\hat G^1_{u}}(A) = F^{i_1}_{\hat G^1_{\mathcal{T}}}(\{ L^{\hat G^1}_{\mathcal{T}} \}) f_{\hat G^1_{\mathcal{T}}, \hat G^1_u}(\{ L^{\hat G^1}_{\mathcal{T}} \},A). \]
By \eqref{bottleneckeffectnew}, \eqref{defpetitfg}, and \eqref{identaftercut}, 
%since $\hat G^1_{T+{\cdot}} = \tilde G_{\cdot}$ we have $f^{\hat G^1}_{T, r}(\{ L^{\hat G^1}_T \},A) = f^{\tilde G}_{0, r-T}(\{ L^{\tilde G}_0 \},A)$. 
we get $F^{i_1}_{\hat G^1_{u}}(A) = F^1_{\tilde G_{u-\mathcal{T}}}(A)$ for all $A \in \mathcal{P}(V_{\hat G^1_{u}}) (= \mathcal{P}(V_{\tilde G_{u-\mathcal{T}}}))$. We similarly get $F^{i_2}_{\hat G^2_{r+u}}(A) = F^1_{\tilde G_{u-\mathcal{T}}}(A)$ for all $A \in \mathcal{P}(V_{\hat G^2_{r+u}}) (= \mathcal{P}(V_{\tilde G_{u-\mathcal{T}}}))$. We deduce that for any $k \geq 1$ and $j \in \{ 1,..., k \}$ we have $\mathbb{P}_{m_1, m_2,1}$-almost surely 
\[ \mathds{1}_{|V_{\hat G^1_{u}}| = k} \sum_{A \in \mathcal{P}(V_{\hat G^1_{u}}) ; |A|=j} F^{i_1}_{\hat G^1_{u}}(A) = \mathds{1}_{|V_{\hat G^2_{r+u}}| = k} \sum_{A \in \mathcal{P}(V_{\hat G^2_{r+u}}) ; |A|=j} F^{i_2}_{\hat G^2_{r+u}}(A). \]
In particular we get that we have $\mathbb{P}_{m_1, m_2,1}$-almost surely for all $t \geq 0$, 
\begin{align}
\mathds{1}_{t > \mathcal{T}} \left ( \mathds{1}_{|V_{\hat G^1_{t}}| = k} \sum_{A \in \mathcal{P}(V_{\hat G^1_{t}}) ; |A|=j} F^{i_1}_{\hat G^1_{t}}(A) - \mathds{1}_{|V_{\hat G^2_{r+t}}| = k} \sum_{A \in \mathcal{P}(V_{\hat G^2_{r+t}}) ; |A|=j} F^{i_2}_{\hat G^2_{r+t}}(A) \right ) = 0. \label{couplagedesfnew}
\end{align}
Since $(\hat G^1_{\beta})_{{\beta} \geq 0}$ and $(\hat G^2_{\beta})_{{\beta} \geq 0}$ have distribution respectively $\mathbb{P}_{m_1}$ and $\mathbb{P}_{m_2}$, we have by definition of $R^{\cdot,\cdot}_{\cdot}(\cdot,\cdot)$ that $\mathbb{E}[\mathds{1}_{|V_{\hat G^1_{t}}| = k} \sum_{A \in \mathcal{P}(V_{\hat G^1_{t}}) ; |A|=j} F^{i_1}_{\hat G^1_{t}}(A)]=R^{m_1,k}_t(i_1,j)$ and \\
\noindent $\mathbb{E}[\mathds{1}_{|V_{\hat G^2_{r+t}}| = k} \sum_{A \in \mathcal{P}(V_{\hat G^2_{r+t}}) ; |A|=j} F^{i_2}_{\hat G^2_{r+t}}(A)]=R^{m_2,k}_{r+t}(i_2,j)$. Combining with \eqref{couplagedesfnew} we get that $R^{m_1,k}_t(i_1,j) - R^{m_2,k}_{r+t}(i_2,j)$ equals 
\begin{align*}
%& \mathbb{E} \left [ \mathds{1}_{t > \mathcal{T}} \left ( \mathds{1}_{|V_{\hat G^1_{t}}| = k} \sum_{A \in \mathcal{P}(V_{\hat G^1_{t}}) ; |A|=j} F^{i_1}_{\hat G^1_{t}}(A) - \mathds{1}_{|V_{\hat G^2_{r+t}}| = k} \sum_{A \in \mathcal{P}(V_{\hat G^2_{r+t}}) ; |A|=j} F^{i_2}_{\hat G^2_{r+t}}(A) \right ) \right ] \\ 
\mathbb{E}_{m_1,m_2,1} \left [ \mathds{1}_{\mathcal{T} \geq t} \left ( \mathds{1}_{|V_{\hat G^1_{t}}| = k} \sum_{A \in \mathcal{P}(V_{\hat G^1_{t}}) ; |A|=j} F^{i_1}_{\hat G^1_{t}}(A) - \mathds{1}_{|V_{\hat G^2_{r+t}}| = k} \sum_{A \in \mathcal{P}(V_{\hat G^2_{r+t}}) ; |A|=j} F^{i_2}_{\hat G^2_{r+t}}(A) \right ) \right ]. 
\end{align*}
%The first term is null according to \eqref{couplagedesfnew} and for the second term, 
By \eqref{borneFps}, the absolute value of the integrand in the above expectation is smaller than $\frac{2(jk)^j}{j!} \mathds{1}_{\mathcal{T} \geq t}$. 
%\[ \mathds{1}_{\mathcal{T} \geq t} \left | \mathds{1}_{|V_{\hat G^1_{t}}| = k} \sum_{A \in \mathcal{P}(V_{\hat G^1_{t}}) ; |A|=j} F^{i_1}_{\hat G^1_{t}}(A) - \mathds{1}_{|V_{\hat G^2_{r+t}}| = k} \sum_{A \in \mathcal{P}(V_{\hat G^2_{r+t}}) ; |A|=j} F^{i_2}_{\hat G^2_{r+t}}(A) \right | \leq \frac{2(jk)^j}{j!} \mathds{1}_{\mathcal{T} \geq t}. \]
Combining with \eqref{majoespcolle} we get 
%\begin{align}
%|R^{m_1,k}_t(i_1,j) - R^{m_2,k}_{r+t}(i_2,j)| & \leq \frac{2 (jk)^j}{j!} \mathbb{P} \left ( \mathcal{T} \geq t \right ) \leq 2 c_1 \frac{(jk)^j}{j!} e^{-c_2 t}. \label{lesgestcauchynew}
%\end{align}
\begin{align}
|R^{m_1,k}_t(i_1,j) - R^{m_2,k}_{r+t}(i_2,j)| & \leq \frac{2 (jk)^j}{j!} \mathbb{P}_{m_1,m_2,1} \left ( \mathcal{T} \geq t \right ) \leq \frac{4 (jk)^j}{j!} e^{-c t}. \label{lesgestcauchynew}
\end{align}

Finally, for any $m,k\geq 1, i \in \{ 1,...,m \}, j \in \{ 1,...,k \}$, \eqref{lesgestcauchynew} applied with $m_1=m_2=m$ and $i_1=i_2=i$ shows that $(R^{m,k}_t(i,j))_{t \geq 0}$ satisfies the Cauchy property as $t$ goes to $\infty$ and is therefore convergent. It moreover shows that the convergence is exponentially fast in $t$ and uniform in $m,i$. Then, \eqref{lesgestcauchynew} applied with $r=0$ shows that $\lim_{t \rightarrow \infty} R^{m,k}_t(i,j)$ does not depend on $m$ and $i$. This proves the first claim of Theorem \ref{cvcoeff}. 
%ON A BESOIN D'UNE BORNE INDEPENDANTE DE $t$ POUR $\mathbb{E} \left [ |V_{G_{t}}|^{k} \right ]$. EST-CE QUE C'EST POSSIBLE VIA LE FAIT QUE C'EST RECURRENT ? ON A DES TEMPS DE RENOUVELLEMENT ET L'ESPERANCE DU TRUC AU TEMPS $t$ EST EGALE A SOMME DES ESPERANCES DU TRUC A TELLE ETAPE ENTRE $2$ RENOUVELLEMENTS FOIS INDICATRICE QUE $t$ CORRESPOND A CETTE ETAPE... OU SINON PAR DUALITE OU UN ARGUMENT SIMILAIRE A CELUI DE LA PREUVE PRECEDENTE QUAND ON CONTROLAIT LES PETITS SAUTS, BREF UN TRUC QUI UTILISE QUE POSITIF RECCURENT DONC $\mathbb{E} \left [ |V_{G_{t}}|^{k} \right ]$ EST BORNEE EN $t$. 
Then, \eqref{relreccoefnew} follows combining Theorem \ref{equadiffcoeffnew} and the convergence \eqref{cvcoeff1new}. The claim $a^1_1=\pi(1)$ comes from Remark \ref{extrarelation}. 

%To justify \eqref{2emerelreccoefnew}, note that for any $k \geq 1, t \geq 0$, $\sum_{l=1}^k R^{1,k}_t(1,l)$ equals 
%\begin{align*}
%\sum_{l=1}^k \mathbb{E}_{1} \left [ \mathds{1}_{|V_{G_{t}}| = k} \sum_{A \in \mathcal{P}(V_{G_{t}}) ; |A|=l} F^1_{G_{t}}(A) \right ] = \mathbb{E}_{1} \left [ \mathds{1}_{|V_{G_{t}}| = k} \sum_{A \in \mathcal{P}(V_{G_{t}})} F^1_{G_{t}}(A) \right ] = \mathbb{P}_{1} \left ( |V_{G_{t}}| = k \right ). 
%\end{align*}
%We have used \eqref{sommetotale=1} for the last equality. Now letting $t$ go to infinity in $\sum_{l=1}^k R^{1,k}_t(1,l) = \mathbb{P}_{1} \left ( |V_{G_{t}}| = k \right )$ we obtain \eqref{2emerelreccoefnew}. 

%\section{Representation of $h(x)$} \label{prmainres}

\subsection{Rigorous relation between Wright-Fisher diffusion and ASG} \label{relwf-asgbis}

We have so far worked on quantities defined from the ASG or the E-ASG. An essential step is to relate them to \eqref{levymodelsdesimp}. The following proposition establishes rigorously the relation, announced in Section \ref{relwf-asg}, between the diffusion \eqref{levymodelsdesimp} and the ASG. 
\begin{prop} \label{h(x)ht(x)} 

For any $x \in [0,1]$, $l \geq 1$ and $T \geq 0$, we have 
\begin{eqnarray}
h^l_T(x) = \mathbb{E} \left [ (X(T))^l | X(0)=x \right ]. \label{momentstoht}
\end{eqnarray}
For any fixed environment $\omega$, $x \in [0,1]$, $l \geq 1$ and $T \geq 0$, we have 
%, if $\omega$ has a jump at time $t < T$, then we have 
\begin{eqnarray}
h^{l, \omega}_{0,T}(x) = \mathbb{E}^{\omega} \left [ (X(\omega,T))^l \mid X(\omega,0)=x \right ]. \label{quencheddual0}
\end{eqnarray}

\end{prop}

The proof of Proposition \ref{h(x)ht(x)} is technical and thus we shift it to Appendix~\ref{A1}. 

\subsection{Representation of $h(x)$: Conclusion of the proof of Theorem \ref{finalformula}} \label{ccl}

We now have almost all the ingredients required to prove Theorem \ref{finalformula}. Lemma \ref{cvn} below provides some simple estimates that will allow 1) to justify the well-definedness of the series in \eqref{mainformula}, 2) to apply term by term the convergence from Theorem \ref{cvcoeff} in the series representation \eqref{decompattendueent}, which will be done in Proposition \ref{limht} to prove \eqref{mainformula}, 3) to apply Proposition \ref{recpilambda} in order to justify the bound \eqref{boundfiniteapprox}. This will yield Theorem \ref{finalformula}. 
\begin{lemme} \label{cvn}
For any $m,k\geq 1, i \in \{ 1,...,m \}$, $y \in [0,1]$, $t \geq 0$, 
\[ P^{m,k,i}_t (y) := \sum_{j=1}^k R^{m,k}_t(i,j) y^j \in [0, \mathbb{P}_m \left ( |V_{G_{t}}| = k \right )] \subset [0, \pi(k)/\pi(m)]. \] 
%and 
%%\[ \sup_{y \in [0,1]} \left ( \sum_{j=1}^k R^{m,k}_t(i,j) y^j \right ) = \sum_{j=1}^k R^{m,k}_t(i,j) = \mathbb{P}_m \left ( |V_{G_{t}}| = k \right ). \]
%\[ \sup_{y \in [0,1]} P^{m,k,i}_t (y) = P^{m,k,i}_t (1) = \mathbb{P}_m \left ( |V_{G_{t}}| = k \right ). \]
For any $k \geq 1$, $y \in [0,1]$, $P_k(y) \in [0, \pi(k)]$. 
%\[ P_k(y) \in [0, \pi(k)], \ \text{and} \ \sup_{y \in [0,1]} P_k(y) = P_k(1) = \pi(k). \]
%In particular the series $\sum_{k=1}^{\infty} (\sum_{j=1}^k R^{m,k}_t(i,j) y^j)$ and $\sum_{k=1}^{\infty} P_k(y)$ are normally convergent on $[0,1]$. 
In particular, for any $m\geq 1, i \in \{ 1,...,m \}$, $t \geq 0$, the series $\sum_{k=1}^{\infty} P^{m,k,i}_t (\cdot)$ and $\sum_{k=1}^{\infty} P_k(\cdot)$ are normally convergent on $[0,1]$. 
\end{lemme}

\begin{proof}

We fix $m,k\geq 1$, $i \in \{ 1,...,m \}$, $y \in [0,1]$ and $t \geq 0$. Using the definitions of $P^{m,k,i}_t (y)$ and $R^{m,k}_t(i,j)$ (see \eqref{defcoeff}), \eqref{entre0et10mix}, and \eqref{majointemporelle}, we get 
\begin{align*}
P^{m,k,i}_t (y) & = \sum_{j=1}^k R^{m,k}_t(i,j) y^j = \sum_{j=1}^k \mathbb{E}_m \left [ \mathds{1}_{|V_{G_{t}}| = k} \sum_{A \in \mathcal{P}(V_{G_{t}}) ; |A|=j} F^i_{G_{t}}(A) \right ] y^j \\
& = \mathbb{E}_m \left [ \mathds{1}_{|V_{G_{t}}| = k} \sum_{A \in \mathcal{P}(V_{G_{t}})} F^i_{G_{t}}(A) y^{|A|} \right ] \in \left [ 0, \mathbb{P}_m \left ( |V_{G_{t}}| = k \right ) \right ] \subset [0, \pi(k)/\pi(m)]. 
\end{align*}
%Using \eqref{defcoeff} and \eqref{sommetotale=1} and proceeding similarly as above we get that $P^{m,k,i}_t (1) = \mathbb{P}_m ( |V_{G_{t}}| = k )$, concluding the proof of 
This proves the first part. By \eqref{cvcoeff1new}, $P^{m,k,i}_t (y)$ converges to $P_k(y)$ as $t$ goes to infinity. Moreover, by the classical result about converge to the equilibrium distribution for continuous-time irreducible positive recurrent Markov chains, $\mathbb{P}_m ( |V_{G_{t}}| = k )$ converges to $\pi(k)$ as $t$ goes to infinity. We thus get the second part of the proposition, letting $t$ go to infinity in the first part. The last statement about normal convergences is an easy consequence of the bounds in the first two parts and of $\sum_{k \geq 1} \pi(k) = 1 < \infty$. 

\end{proof}

\begin{prop} \label{limht}
$\mathbb{P}$-almost surely, $\lim_{t \rightarrow \infty} X(t)$ exists and belongs to $\{0,1\}$. Moreover, for any $i \geq 1$ and $x \in [0,1]$ we have 
\begin{eqnarray}
h(x) = \lim_{t \rightarrow \infty} h^i_t(x) = \sum_{k=1}^{\infty} P_k(x). \label{httoh}
\end{eqnarray}
\end{prop}

\begin{proof}
Let us fix $i \geq 1$, $x \in [0,1]$, and $m\geq i$. Recall from \eqref{decompattendueent} and the definition of $P^{m,k,i}_t (x)$ that we have $h^i_t(x) = \sum_{k=1}^{\infty} P^{m,k,i}_t (x)$. 
%\[ h^i_t(x) = \sum_{k=1}^{\infty} \left ( \sum_{j=1}^k R^{m,k}_t(i,j) x^{j} \right ) = \sum_{k=1}^{\infty} P^{m,k,i}_t (x). \]
From \eqref{cvcoeff1new} we see that for every $k \geq 1$ we have $P^{m,k,i}_t (x) \longrightarrow_{t \rightarrow \infty} P_k(x)$. The bound $|P^{m,k,i}_t (x)| \leq \pi(k)/\pi(m)$ from Lemma \ref{cvn} allows to apply dominated convergence so we get the second equality in \eqref{httoh}. 
%\begin{eqnarray}
%\lim_{t \rightarrow \infty} h^i_t(x) = \sum_{k=1}^{\infty} P_k(x). \label{limhtl}
%\end{eqnarray}

By \eqref{levymodelsdesimp}, $X$ is a sub-martingale if $\mathbb{E}[L(1)] \geq 0$, and a super-martingale if $\mathbb{E}[L(1)] \leq 0$. Moreover $X$ is bounded. Therefore, in any case, $X(t)$ converges almost surely to a limit as $t$ goes to infinity. By dominated convergence we get that for any $i \geq 1$, 
\begin{eqnarray}
\lim_{t \rightarrow \infty} \mathbb{E} \left [ X(t)^i \big | X(0)=x \right ] = \mathbb{E} \left [ \left ( \lim_{t \rightarrow \infty} X(t) \right )^i \big | X(0)=x \right ]. \label{limmmt}
\end{eqnarray}
Combining \eqref{limmmt}, the second equality in \eqref{httoh}, and Proposition \ref{h(x)ht(x)}, we get that for any $i \geq 1$, 
\begin{eqnarray}
\mathbb{E} \left [ \left ( \lim_{t \rightarrow \infty} X(t) \right )^i \big | X(0)=x \right ] = \sum_{k=1}^{\infty} P_k(x). \label{h=mmt}
\end{eqnarray}
\eqref{h=mmt} shows in particular that all moments of positive order of $\lim_{t \rightarrow \infty} X(t)$ are equal so this random variable is supported on $\{0,1\}$. 
%\textbf{Question: How can we deduce that $X$ is almost surely absorbed in finite time, and not only at $t=\infty$?}
Combining with \eqref{defh(x)} we get 
\[ h(x) = \mathbb{P} \left ( \lim_{t \rightarrow \infty} X(t) = 1 \big | X(0)=x \right ) = \mathbb{E} \left [ \lim_{t \rightarrow \infty} X(t) \big | X(0)=x \right ] = \sum_{k=1}^{\infty} P_k(x), \]
where we have used \eqref{h=mmt} for the last equality. This concludes the proof. 

\end{proof}

We can now conclude the proof of Theorem \ref{finalformula}. \eqref{mainformula} and the claims about $\lim_{t \rightarrow \infty} X(t)$ are established in Proposition \ref{limht}. The normal convergence of $\sum_{k=1}^{\infty} P_k(y)$ has been proved in Lemma \ref{cvn}. 
%The derivation of the linear relation \eqref{systlincoef} is immediate from Theorem \ref{cvcoeff}, applying \eqref{relreccoefnew} with $j=1,...,k$. 
From \eqref{mainformula} and the bound $|P_k(x)| \leq \pi(k)$ from Lemma \ref{cvn} we get 
%\[ \left | h(x) - \sum_{k=1}^{m-1} P_k(x) \right | \leq \sum_{j \geq m} \pi(j). \]
\eqref{boundfiniteapprox}, where the second inequality in \eqref{boundfiniteapprox} is only \eqref{majoqueuepiknew} from Proposition \ref{recpilambda}. 

\begin{remark}
We have proved in Proposition \ref{limht} that $X$ is almost surely eventually absorbed at $\{0,1\}$ but more can actually be proved. Indeed, since it can be seen that the line counting process of the E-ASG comes down from infinity, one can use classical arguments (see for example \cite[Prop. 2.19]{CHS19}) to show that $X$ is actually almost surely absorbed in $\{0,1\}$ in finite time. 
\end{remark}

%Letting $t \rightarrow \infty$ in \eqref{decompattendueent} we get 
%\[ 1-h(x) = \sum_{k=1}^{\infty} \sum_{j=1}^k a^k_j (1-x)^{j} = \sum_{k=1}^{\infty} P_k(1-x), \]
%where 
%\[ P_k(X) := \sum_{j=1}^k a^k_j X^{j}. \]

%\subsection{Solving the system}

%We now prove that each matrix $A_{k}$ is invertible. For this we only need to prove that for any $k \geq 1$, 
%\[ \sum_{j=1}^k C_j^{k+1} \frac{(-1)^{k+1-j} (k+1)! k! (k+j-1)!}{j! (j-1)! (2k)! (k-j+1)!} \neq 0. \]
%
%POUR CE QUI PRECEDE, PRENDRE EN COMPTE LE FAIT QUE TOUS LES COEFFICIENTS ONT CHANGE QUAND ON A CORRIGE LES ERREURS DANS LE SYSTEME LINEAIRE. OK POUR CE QUI SUIT. 
%
%We here solve the system \eqref{systlincoef}. By doing Gauss reduction to cancel the first $k$ coefficients in the last row we get that \eqref{systlincoef} is equivalent to 
%\[ \tilde A_{k+1} . x = \tilde b_{k+1} \]
%where $\tilde A_{k+1}$ is the same as $A_{k+1}$, except that the last row is $(0,0,...,0,Z_{n+1}^{n+1})$ (instead of $(C_1^{k+1},C_2^{k+1},...,C_k^{k+1},C_{k+1}^{k+1})$ for $A_{k+1}$), and $\tilde b_{k+1}$ is the same as $b_{k+1}$, except that $\tilde b_{k+1}(k+1) = b_{k+1}(k+1) - \sum_{j=1}^{k} Z_j^{n+1} b_{k+1}(j)$, and where the coefficients $(Z_j^{k+1}, 1 \leq j \leq k+1)$ are defined by $Z_1^{k+1} = C_1^{k+1} = 1/k(k+1)$ and the recursion relation $Z_{j+1}^{k+1} = C_{j+1}^{k+1} - Z_j^{k+1} \frac{(j+1)j}{(k+j)(k+1-j)}$. It is straightforward to prove by induction that $Z_{j}^{k+1} = ...$

\section{Taylor expansion of $h(x)$ near $x=0$} \label{dlh(x)}

\subsection{System of differential equations and asymptotics of coefficients: Proofs of Theorems \ref{equadiffcoeff1indice} and \ref{cvcoeff1indice}} \label{dlh(x)subsec1}

%VOIR DANS LA VERSION DU 02/05/1019 CE DONT ON A BESOIN POUR QUE WELL-DEFINED ET, POUR LA DERIVATION DE L'EQUA-DIFF, CE QUI EST DIFFERENT DE LA PREUVE DU THEOREME \ref{equadiffcoeffnew}
%
%ON A IMPORTE : 
%
%We have 
%\begin{lemme} \label{asymptcoeff1indice}
%
%\begin{align}
%\forall j \geq 1, k \in \{ 1, 2,..., j-1, j+1 \}, \ & \lim_{h \rightarrow 0} \frac1{h} Q_h(k,j) = \tau(k,j), \label{asymptcoeff11indice} \\
%\forall j \geq 1, \ & \lim_{h \rightarrow 0} \frac1{h} (Q_h(j,j)-1) = -d_j. \label{asymptcoeff21indice}
%\end{align}
%
%\end{lemme}

We can proceed similarly as in the proofs of Theorems \ref{equadiffcoeffnew} and \ref{cvcoeff} to prove Theorems \ref{equadiffcoeff1indice} and \ref{cvcoeff1indice}. However, we rather choose to show that the latter two can be derived from Theorems \ref{equadiffcoeffnew} and \ref{cvcoeff} via two lemmas. Lemma \ref{reconstitutioncoeff} below relates the coefficient $Q_t(i,j)$ to a series of coefficients $R^{m,k}_t(i,j)$, and Lemma \ref{reconstitutiondercoeff} below allows to differentiate this series term by term. Theorems \ref{equadiffcoeff1indice} and \ref{cvcoeff1indice} will then easily follow. 

\begin{lemme} \label{reconstitutioncoeff}
For any $t \geq 0$ and $i,j \geq 1$, $Q_t(i,j) = \sum_{k=1}^{\infty} R^{i,k}_t(i,j)$ and the convergence of the series holds uniformly in $t \in [0,\infty)$. 
\end{lemme}

\begin{proof}
Using the expressions of the coefficients $R^{m,k}_t(i,j)$ and $Q_t(i,j)$ in Definition \ref{defcoefsg}, \eqref{borneFps} from Lemma \ref{sgwelldef}, and \eqref{bornemomentGtnew} from Lemma \ref{lemmecomb}, we get for $n \geq 1$, 
\begin{align*}
\left | Q_t(i,j) - \sum_{k=1}^{n} R^{i,k}_t(i,j) \right | & = \left | \mathbb{E}_i \left [ \mathds{1}_{|V_{G_{t}}| > n} \sum_{A \in \mathcal{P}(V_{G_{t}}) ; |A|=j} F^i_{G_{t}}(A) \right ] \right | \\
& \leq \frac{j^j}{j!} \mathbb{E}_i \left [ \mathds{1}_{|V_{G_{t}}| > n} |V_{G_{t}}|^{j} \right ] \leq \frac{j^j}{j! \pi(i)} \sum_{l > n} l^{j} \pi(l). 
%& = \frac{j^j}{j!} \sum_{l > n} l^j \mathbb{P}_i \left ( |V_{G_{t}}| = l \right ) \leq \frac{j^j}{j! \pi(i)} \sum_{l > n} l^{j} \pi(l). 
\end{align*}
By \eqref{majoqueuepiknew}, the right-hand side converges to $0$ as $n$ goes to infinity and the lemma follows. 
\end{proof}

\begin{lemme} \label{reconstitutiondercoeff}
For any $t \geq 0$ and $i,j \geq 1$, 
\[ \sum_{k=1}^{\infty} \frac{d}{dt} R^{i,k}_t(i,j) = - d_j Q_t(i,j) +f_j Q_t(i,j-1) + \sum_{l =1}^{j+1} \tau(l,j) Q_t(i,l), \]
and the convergence of the series in the left-hand side holds uniformly in $t \in [0,\infty)$. 
\end{lemme}

The proof of Lemma \ref{reconstitutiondercoeff} is rather straightforward (modulo the use of Theorem \ref{equadiffcoeffnew}) but quite computational so we shift it to Appendix~\ref{A4}. Theorem \ref{equadiffcoeff1indice} now follows easily from Lemmas \ref{reconstitutioncoeff} and \ref{reconstitutiondercoeff}. 
%\subsection{Asymptotics of coefficients: Proof of Theorem \ref{cvcoeff1indice}}
We conclude the proof of Theorem \ref{cvcoeff1indice} as follows. 
The existence of the limit in \eqref{cvcoeff1} and the identity $b_j = \sum_{k=1}^{\infty}a^k_j$ follow from the combination of \eqref{cvcoeff1new} with the uniform convergence from Lemma \ref{reconstitutioncoeff}. $b_j = \sum_{k=1}^{\infty}a^k_j$ shows that $b_j$ does not depend on $i$. Then, \eqref{relreccoef} follows combining Theorem \ref{equadiffcoeff1indice} and the convergence \eqref{cvcoeff1}. 

\subsection{Taylor expansion: Proof of Theorem \ref{taylorexp}} \label{proofoftaylorexp}

Let us fix $n \geq 1$, $T > 0$, and $x \in [0,1]$. Using \eqref{propagationformule1mix} and decomposing on the cardinality of $A \in \mathcal{P}(V_{G_{T}})$, we get that $h^1_T(x)$ can be decomposed as
\begin{align*}
\mathbb{E}_1 \left [\sum_{j=1}^{2n} \sum_{A \in \mathcal{P}(V_{G_{T}}) ; |A|=j} F^1_{G_{T}}(A) \mathds{1}_{E(T,T,A,x)} \right ] + \mathbb{E}_1 \left [\sum_{j>2n} \sum_{A \in \mathcal{P}(V_{G_{T}}) ; |A|=j} F^1_{G_{T}}(A) \mathds{1}_{E(T,T,A,x)} \right ], 
\end{align*}
where $E(T,T,A,x)$ is as in Definition \ref{type0events}. The finiteness of the first expectation follows from \eqref{borneFesp} applied for each $j$ from $1$ to $2n$. The finiteness of the second expectation follows from \eqref{entre0et10mix} together with the finiteness of the first expectation. Let us denote by $S_T(n,x)$ the second expectation. Proceeding as in \eqref{egsuren} we see that the first term equals $\mathbb{E}_1 [\sum_{j=1}^{2n} \sum_{A \in \mathcal{P}(V_{G_{T}}) ; |A|=j} F^1_{G_{T}}(A) x^j ]$ and clearly the later equals $\sum_{j=1}^{2n} Q_T(1,j) x^{j}$. We thus get 
\begin{align}
h^1_T(x) = \sum_{j=1}^{2n} Q_T(1,j) x^{j} + S_T(n,x). \label{taylorexp3}
\end{align}
We now bound the term $S_T(n,x)$ when $x$ is small, uniformly in $T$. Once this is done, it will be possible to justify Theorem \ref{taylorexp} by letting $T$ go to infinity in \eqref{taylorexp3}. 
\begin{prop} \label{taylorexp4}

For any $n \geq 1$ there exists $\epsilon_n > 0$ such that 
\begin{eqnarray}
\forall T > 0, \forall x \in (0,\epsilon_n), \ |S_T(n,x)| \leq x^{n+1/4}. \label{bornereste0}
\end{eqnarray}

\end{prop}

The idea is that the integrand, inside the expectation defining $S_T(n,x)$, is non-zero only if, after applying the type assignment procedure from Definition \ref{typeaspreasg}, the E-ASG contains at least $2n+1$ lines of type $0$ at time $\beta=T$. The probability of this event is controlled using our upper bounds for the tail distribution of the line counting process of the E-ASG and, thanks to \eqref{entre0et10mix}, the integrand is re-written in terms of finitely many terms for which we have suitable bounds. 

\begin{proof} [Proof of Proposition \ref{taylorexp4}]
On $\{ |V_{G_{T}}| \leq 2n \}$, the sum inside the expectation defining $S_T(n,x)$ is empty and therefore null. We apply the type assignment procedure on $[0,T]$ with initial condition $x$ (see Definition \ref{typeaspreasg}) and denote by $N_T$ the number of lines that are assigned type $0$ in $V_{G_{T}}$. Note that for any $k \geq 1$, conditionally on $\{ |V_{G_{T}}| = k \}$, we have $N_T \sim \mathcal{B}(k,x)$, where $\mathcal{B}(k,x)$ denotes the binomial distribution with parameters $k$ and $x$. Note moreover that for any $A \in \mathcal{P}(V_{G_{T}})$ such that $|A| > 2n$ we have $\mathds{1}_{E(T,T,A,x)} = 0$ on $\{ N_T \leq 2n \}$. Therefore, the sum inside the expectation defining $S_T(n,x)$ is null outside the event $\{ |V_{G_{T}}| > 2n, N_T > 2n \}$. We thus have 
\begin{align*}
S_T(n,x) & = \mathbb{E}_1 \left [ \mathds{1}_{|V_{G_{T}}| > 2n, N_T > 2n} \sum_{j>2n} \sum_{A \in \mathcal{P}(V_{G_{T}}) ; |A|=j} F^1_{G_{T}}(A) \mathds{1}_{E(T,T,A,x)} \right ] \\
& = \mathbb{E}_1 \left [ \mathds{1}_{|V_{G_{T}}| > 2n, N_T > 2n} \sum_{A \in \mathcal{P}(V_{G_{T}})} F^1_{G_{T}}(A) \mathds{1}_{E(T,T,A,x)} \right ] \\
& - \sum_{j=1}^{2n} \mathbb{E}_1 \left [ \mathds{1}_{|V_{G_{T}}| > 2n, N_T > 2n} \sum_{A \in \mathcal{P}(V_{G_{T}}) ; |A|=j} F^1_{G_{T}}(A) \mathds{1}_{E(T,T,A,x)} \right ]. 
\end{align*}
Therefore 
\begin{align*}
|S_T(n,x)| & \leq \mathbb{E}_1 \left [ \mathds{1}_{|V_{G_{T}}| > 2n, N_T > 2n} \left | \sum_{A \in \mathcal{P}(V_{G_{T}})} F^1_{G_{T}}(A) \mathds{1}_{E(T,T,A,x)} \right | \right ] \\
& + \sum_{j=1}^{2n} \mathbb{E}_1 \left [ \mathds{1}_{|V_{G_{T}}| > 2n, N_T > 2n} \sum_{A \in \mathcal{P}(V_{G_{T}}) ; |A|=j} \left | F^1_{G_{T}}(A) \right | \right ]. 
\end{align*}
Using \eqref{entre0et10mix} for the first term and \eqref{borneFps} for each of the $2n$ other terms we get 
\begin{align*}
|S_T(n,x)| & \leq \mathbb{E}_1 \left [ \mathds{1}_{|V_{G_{T}}| > 2n, N_T > 2n} \right ] + \sum_{j=1}^{2n} \mathbb{E}_1 \left [ \frac{j^j}{j!} |V_{G_{T}}|^{j} \mathds{1}_{|V_{G_{T}}| > 2n, N_T > 2n} \right ] \\
& \leq (2n+1)^{2n+1} \mathbb{E}_1 \left [ |V_{G_{T}}|^{2n} \mathds{1}_{|V_{G_{T}}| > 2n, N_T > 2n} \right ]. 
\end{align*}
Recall that, conditionally on $|V_{G_{T}}|$, we have $N_T \sim \mathcal{B}(|V_{G_{T}}|,x)$. Therefore the above expectation can be re-written as $\mathbb{E}_1 [ |V_{G_{T}}|^{2n} \mathds{1}_{|V_{G_{T}}| \geq \sum_{i=1}^{2n+1} Y_i} ]$, where $Y_1,Y_2,...,Y_{2n+1}$ are \textit{iid} geometric random variable with parameter $x$, independent of $|V_{G_{T}}|$. We set $J := \sharp \{ i \in \{1,...,2n+1\}, Y_i \leq x^{-1/2} \}$. Note that $J \sim \mathcal{B}(2n+1,\mathbb{P}(Y_1 \leq x^{-1/2}))$ and that $J$ is independent of $|V_{G_{T}}|$. We have 
\begin{align*}
|S_T(n,x)| & \leq (2n+1)^{2n+1} \mathbb{E}_1 \left [ |V_{G_{T}}|^{2n} \mathds{1}_{|V_{G_{T}}| \geq \sum_{i=1}^{2n+1} Y_i} \right ] \\
& \leq (2n+1)^{2n+1} \left ( \mathbb{E}_1 \left [ |V_{G_{T}}|^{2n} \right ] \mathbb{P} \left [ J = 2n+1 \right ] + \mathbb{E}_1 \left [ |V_{G_{T}}|^{2n} \mathds{1}_{|V_{G_{T}}| \geq x^{-1/2}} \right ] \right ) \\
%& = (2n+1)^{2n+1} \mathbb{E}_1 \left [ |V_{G_{T}}|^{2n} \right ] \left ( \sum_{j=2n+2}^{2n+1} \binom{2n+1}{j} \mathbb{P}(Y_1 \leq x^{-1/2})^{j} \mathbb{P}(Y_1 > x^{-1/2})^{2n+1-j} \right ) \\
& = (2n+1)^{2n+1} \left ( \mathbb{E}_1 \left [ |V_{G_{T}}|^{2n} \right ] \mathbb{P}(Y_1 \leq x^{-1/2})^{2n+1} + \mathbb{E}_1 \left [ |V_{G_{T}}|^{2n} \mathds{1}_{|V_{G_{T}}| \geq x^{-1/2}} \right ] \right ) \\
%& \leq (2n+1)^{2n+1} \left ( \mathbb{E}_1 \left [ |V_{G_{T}}|^{2n} \right ] 2^{2n+1} \mathbb{P}(Y_1 \leq x^{-1/2})^{2n+2} + \mathbb{E}_1 \left [ |V_{G_{T}}|^{2n} \mathds{1}_{|V_{G_{T}}| \geq 2n x^{-1/2}} \right ] \right ) \\
%& = (2n+1)^{2n+1} \left ( \left ( \frac1{\pi(1)} \sum_{k \geq 1} k^{2n} \pi(k) \right ) \mathbb{P}(Y_1 \leq x^{-1/2})^{2n+1} + \sum_{k \geq x^{-1/2}} k^{2n} \mathbb{P}_1 \left ( |V_{G_{T}}|=k \right ) \right ) \\
& \leq \frac{(2n+1)^{2n+1}}{\pi(1)} \left ( \left ( \sum_{k \geq 1} k^{2n} \pi(k) \right ) \mathbb{P}(Y_1 \leq x^{-1/2})^{2n+1} + \sum_{k \geq x^{-1/2}} k^{2n} \pi(k) \right ), 
\end{align*}
where we have used two times \eqref{bornemomentGtnew} for the last inequality. The sum $\sum_{k \geq 1} k^{2n} \pi(k)$ is finite according to \eqref{majoqueuepiknew}. By \eqref{majoqueuepiknew} again, we see that if $\epsilon_n$ is chosen small enough then for all $x \in (0,\epsilon_n)$ and $k \geq x^{-1/2}$ we have $\pi(k) \leq k^{-4n-3}$. We thus get 
\begin{eqnarray}
|S_T(n,x)| \leq C(n) \left ( \mathbb{P}(Y_1 \leq x^{-1/2})^{2n+1} + x^{n+1/2} \right ), \label{bornereste1}
\end{eqnarray}
where $C(n)$ is a constant depending on $n$ (and not on $T$ or $x$). Moreover we have 
\[ \mathbb{P}(Y_1 \leq x^{-1/2}) = \mathbb{P}(e^{-x^{1/2} Y_1} \geq e^{-1}) \leq e^{1} \mathbb{E} [e^{-x^{1/2} Y_1}] = \frac{e^{1} x e^{-x^{1/2}}}{1-(1-x)e^{-x^{1/2}}} \leq 2e^1x^{1/2}, \]
where the last inequality holds for all $x \in (0,\epsilon_n)$ if $\epsilon_n$ is chosen small enough. Plugging into \eqref{bornereste1} we get $|S_T(n,x)| \leq C'(n) x^{n+1/2}$,
where $C'(n)$ is a constant depending on $n$ (and not on $T$ or $x$). By choosing $\epsilon_n$ even smaller if necessary we obtain \eqref{bornereste0}. Note that $\epsilon_n$ does not depend on $T$. 
\end{proof}

We can now conclude the proof of Theorem \ref{taylorexp}. We fix $n \geq 1$ and $x \in (0,\epsilon_n)$, where $\epsilon_n$ is given by Proposition \ref{taylorexp4}. We have $h^1_T(x) \longrightarrow_{T \rightarrow \infty} h(x)$ according to Proposition \ref{limht} and $\sum_{j=1}^{2n} Q_T(1,j) x^{j} \longrightarrow_{T \rightarrow \infty} \sum_{j=1}^{2n} b_j x^{j}$ according to Theorem \ref{cvcoeff1indice}. Combining with \eqref{taylorexp3} we get that $S_T(n,x)$ has a limit as $T$ goes to infinity, that we denote by $S_{\infty}(n,x)$, and 
\begin{eqnarray}
h(x) = \sum_{j=1}^{n} b_j x^{j} + \sum_{j=n+1}^{2n} b_j x^{j} + S_{\infty}(n,x). \label{dl1}
\end{eqnarray}
Taking the limit as $T$ goes to infinity into \eqref{bornereste0} we get 
\begin{eqnarray}
\forall x \in (0,\epsilon_n), \ |S_{\infty}(n,x)| \leq x^{n+1/4}. \label{dl2}
\end{eqnarray}
%Therefore, reducing $\epsilon_n$ if necessary we get \eqref{taylorexp1}. 
The combination of \eqref{dl1} and \eqref{dl2} concludes the proof of Theorem \ref{taylorexp}. 

\appendix 
\section{ } \label{append}

\subsection{Line counting process of the E-ASG: Proof of Proposition \ref{recpilambda} from Section \ref{enlargedasglcp}}\label{A2}

For the line counting process of the E-ASG $(G_{\beta})_{\beta \geq 0}$ we use the notation $(|V_{G_{\beta}}|)_{\beta\geq 0}$ from Section \ref{enlargedasgdefnot}. Recall from Section \ref{enlargedasglcp} that $(|V_{G_{\beta}}|)_{\beta\geq 0}$ is an irreducible, positive recurrent, continuous-time Markov process on $\mathbb{N}=\{1,2,...\}$, and that 
%and infinitesimal rates:
%\[ q(i,j):=\left\{\begin{array}{ll}
%            i(i-1)/2 &\text{if $j=i-1$},\\
%            \lambda &\text{if $j=2i$}. 
%            \end{array}\right. \]
%It is easy to see that $(|V_{G_{t}}|)_{t\geq 0}$ is an irreducible Markov chain and that the state $\{1\}$ is positive recurrent, so $(|V_{G_{t}}|)_{t\geq 0}$ is a positive recurrent Markov chain. 
%In particular $(|V_{G_{t}}|)_{t\geq 0}$ admits a 
its stationary distribution is denoted by $\pi$. We now justify Proposition \ref{recpilambda}. 
%\begin{remark}
%It is possible to prove that for any $\delta >1$ there exists $K(\delta,\lambda) \geq 1$ such that 
%\[ \forall k \geq K(\delta,\lambda), \ \pi([k,\infty)) \leq \exp \left ( - \delta \frac{(\log(k))^2}{2 \log(2)} \right ). \]
%\end{remark}
In order to establish the recursion formula \eqref{recpik} for the probabilities $\pi(k)$ we use Siegmund duality. Let us consider the Markov process $(D_{\beta})_{\beta\geq 0}$ with rates
\[ q_D(i,j):=\left\{\begin{array}{ll}
            i(i-1) &\textrm{if $j=i+1$},\\
            \sigma(i-1) &\textrm{if $j=i-1$},\\
            \lambda &\textrm{if $j=\lfloor \frac{i+1}{2} \rfloor$}. 
            \end{array}\right. \]
Note that $1$ is an absorbing state and that the process may explode to $\infty$ in finite time. 
\begin{lemme}[Siegmund duality]\label{sd}
The processes $(|V_{G_{\beta}}|)_{\beta\geq 0}$ and $(D_{\beta})_{\beta\geq 0}$ are Siegmund duals, i.e. for all $\ell, d\in \mathbb{N}$ and $t\geq 0$, we have
\begin{align}
\mathbb{P}\left(|V_{G_{t}}|\geq d\mid |V_{G_{0}}|=\ell\right)& =\mathbb{P}\left(\ell\geq D_t \mid D_0=d\right), \label{duality} \\
\pi([d, \infty))&=\mathbb{P}\left( \exists t \geq 0 \ \text{s.t.} \ D_t= 1 \mid D_0=d\right). \label{limitduality}
\end{align}
\end{lemme}
\begin{proof}
We consider the function $H:\mathbb{N}\times\mathbb{N} \rightarrow \{0,1\}$ defined via $H(\ell,d):=\mathds{1}_{\ell\geq d}$ for $\ell,d\in\mathbb{N}$. Let $\mathcal{G}_1$ and $\mathcal{G}_2$ be the infinitesimal generators of $(|V_{G_{\beta}}|)_{\beta\geq 0}$ and $(D_{\beta})_{\beta\geq 0}$, respectively. By \cite[Prop. 1.2]{Jaku} we only have to show that $\mathcal{G}_1 H(\cdot,d)(\ell)=\mathcal{G}_2 H(\ell,\cdot)(d)$ for all $\ell,d\in\mathbb{N}$ in order to justify \eqref{duality}. We have
\begin{align*}
\mathcal{G}_1 H(\cdot,d)(\ell) = - \ell (\ell-1) \mathds{1}_{\ell=d} + \ell \sigma \mathds{1}_{\ell=d-1} + \lambda \mathds{1}_{\ell< d \leq 2\ell}, 
\end{align*}
and 
\begin{align*}
\mathcal{G}_2 H(\ell,\cdot)(d)=- d (d-1) \mathds{1}_{d=\ell} + (d-1) \sigma \mathds{1}_{d=l+1} + \lambda \mathds{1}_{\lfloor \frac{d+1}{2} \rfloor \leq l, d > \ell} = \mathcal{G}_1 H(\cdot,d)(\ell). 
\end{align*}
\eqref{duality} follows. Then \eqref{limitduality} follows letting $t$ go to infinity into \eqref{duality}. 

\end{proof}

\begin{proof} [Proof of Proposition \ref{recpilambda}]

Using \eqref{limitduality} and applying a first step decomposition to the process $D$, we obtain that for all $k \geq 1$ we have
\begin{align}
& \left ( k(k-1) + \sigma (k-1) + \lambda \right ) \pi([k, \infty)) \nonumber \\
= & k(k-1) \pi([k+1, \infty)) + \sigma (k-1) \pi([k-1, \infty)) + \lambda \pi \left ([\left \lfloor \frac{k+1}{2} \right \rfloor, \infty) \right ). \label{r1}
\end{align}
Using \eqref{r1} we get that, for any $k \geq 2$, $\pi(k)$ equals 
\begin{align*}
& \pi([k, \infty)) - \pi([k+1, \infty)) \\
= & \pi([k, \infty)) - \frac{\left ( k(k-1) + \sigma (k-1) + \lambda \right ) \pi([k, \infty)) - \sigma (k-1) \pi([k-1, \infty)) - \lambda \pi \left ([\lfloor \frac{k+1}{2} \rfloor, \infty) \right )}{k(k-1)} \\
%= & \frac{\sigma}{k} \pi(k-1) + \frac{\lambda}{k(k-1)} \left ( \pi \left ([\left \lfloor \frac{k+1}{2} \right \rfloor, \infty) \right ) - \pi([k, \infty)) \right ) \\
= & \frac{\sigma}{k} \pi(k-1) + \frac{\lambda}{k(k-1)} \left ( \pi\left (\left \lfloor \frac{k+1}{2} \right \rfloor \right ) + ... + \pi(k-1) \right ), 
\end{align*}
which is \eqref{recpik}. In order to prove \eqref{majoqueuepiknew}, let us first prove by induction on $n$ that for any $n \geq 0$, 
\begin{eqnarray}
\forall k \geq 3 \times 2^n, \ \pi(k) \leq \frac{(\sigma(\sigma + \lambda)+\lambda) (\sigma + \lambda)^{n}}{k(k-1)\prod_{i=1}^n(2^{-i}k - 1)}, \label{recmajo}
\end{eqnarray}
with the convention $\prod_{i=1}^0 \cdots = 1$. Using \eqref{recpik} and that $\pi(\cdot)$ is a probability measure we get for all $k \geq 2$, 
\begin{eqnarray}
\pi(k) \leq \frac{\sigma}{k} + \frac{\lambda}{k(k-1)} \leq \frac{\sigma + \lambda}{k}. \label{majoprelpik}
\end{eqnarray}
Using now \eqref{recpik} together with \eqref{majoprelpik}, and that $\pi(\cdot)$ is a probability measure, we get for $k \geq 3$, 
\[ \pi(k) \leq \frac{\sigma}{k} \times \frac{\sigma + \lambda}{k-1} + \frac{\lambda}{k(k-1)} = \frac{\sigma(\sigma + \lambda)+\lambda}{k(k-1)}. \]
\eqref{recmajo} thus holds for $n=0$. Let us now assume that \eqref{recmajo} holds for some $n$. We fix $k \geq 3 \times 2^{n+1}$ and apply \eqref{recmajo} to each term in the right-hand side of \eqref{recpik} which yields 
\begin{align*}
\pi(k) & \leq \frac{\sigma (\sigma(\sigma + \lambda)+\lambda) (\sigma + \lambda)^{n}}{k(k-1)(k-2)\prod_{i=1}^n(2^{-i}(k-1) - 1)} + \frac{\lambda}{k(k-1)} \sum_{j=\lfloor \frac{k+1}{2} \rfloor}^{k-1} \frac{(\sigma(\sigma + \lambda)+\lambda) (\sigma + \lambda)^{n}}{j(j-1)\prod_{i=1}^n(2^{-i}j - 1)} \\
& \leq \frac{(\sigma(\sigma + \lambda)+\lambda) (\sigma + \lambda)^{n}}{k(k-1)\prod_{i=1}^n(2^{-i} \lfloor \frac{k+1}{2} \rfloor - 1)} \left ( \frac{\sigma}{k-2} + \sum_{j=\lfloor \frac{k+1}{2} \rfloor}^{k-1} \frac{\lambda}{j(j-1)} \right ) \\
& \leq \frac{(\sigma(\sigma + \lambda)+\lambda) (\sigma + \lambda)^{n}}{k(k-1)\prod_{i=1}^n(2^{-(i+1)}k - 1)} \times \frac{\sigma+\lambda}{\lfloor \frac{k+1}{2} \rfloor - 1} \\
& \leq \frac{(\sigma(\sigma + \lambda)+\lambda) (\sigma + \lambda)^{n+1}}{k(k-1)\prod_{i=1}^n(2^{-(i+1)}k - 1)} \times \frac{1}{2^{-1}k - 1} = \frac{(\sigma(\sigma + \lambda)+\lambda) (\sigma + \lambda)^{n+1}}{k(k-1)\prod_{i=1}^{n+1}(2^{-i}k - 1)}. 
\end{align*}
The induction is thus proved. Let $\log_2(\cdot) := \log(\cdot)/\log(2)$. We now fix $k \geq 4$ and apply \eqref{recmajo} at $n = \lfloor \log_2(k) \rfloor - 2$: 
\begin{align*}
\pi(k) & \leq \frac{\sigma(\sigma + \lambda)+\lambda}{k(k-1)} \exp \left ( \log(\sigma + \lambda) (\lfloor \log_2(k) \rfloor -2) - \sum_{i=1}^{\lfloor \log_2(k) \rfloor - 2} \log(2^{-i}k - 1) \right ). 
\end{align*}

Note that for $i\leq \lfloor \log_2(k) \rfloor - 2$, $\log(2^{-i}k - 1)$ equals 
\[ \log(2^{\lfloor \log_2(k) \rfloor - 2 -i} \times 2^{2 + (\log_2(k)-\lfloor \log_2(k) \rfloor)} - 1) \geq \log(2^{\lfloor \log_2(k) \rfloor - 2 -i}) = (\lfloor \log_2(k) \rfloor - 2 -i) \log(2). \] 
Therefore the above yields 
\begin{align*}
\pi(k) & \leq \frac{\sigma(\sigma + \lambda)+\lambda}{k(k-1)} \exp \left ( \log(\sigma + \lambda) (\lfloor \log_2(k) \rfloor -2) - \log(2) \sum_{i=1}^{\lfloor \log_2(k) \rfloor - 2} (\lfloor \log_2(k) \rfloor - 2 -i) \right ) \\
& = \frac{\sigma(\sigma + \lambda)+\lambda}{k(k-1)} \exp \left ( \log(\sigma + \lambda) (\lfloor \log_2(k) \rfloor -2) - \frac{\log(2)}{2} (\lfloor \log_2(k) \rfloor - 2)(\lfloor \log_2(k) \rfloor - 3) \right ). 
\end{align*}
Applying the above inequality to all $j \geq k$ and summing we get \eqref{majoqueuepiknew}. 
%Then it is easy to see that, in \eqref{majopik}, the coefficient in the exponential is decreasing in $k$ as long as $k \geq 2^{7/2} (\sigma + \lambda)$ so \eqref{majoqueuepik} follows from \eqref{majopik}. 

\end{proof}

\subsection{Technical results for the E-ASG: some lemmas for Section \ref{behaviour01} and after} \label{A3}

\begin{lemme}\label{esteasg}
For any $m \geq 1$, 
\begin{align}
\mathbb{P}_m \left ( \sharp \ \text{of transitions of E-ASG on} \ [0,\epsilon] \geq 2 \right ) & \leq 16 m^4 (1+\lambda + \sigma)^2 \epsilon^2, \label{esteasg>2} \\
\mathbb{P}_m \left ( N_{mult}(\epsilon) = 1, N_{coal}(\epsilon) = 0, N_{sing}(\epsilon) = 0 \right ) & \underset{\epsilon \rightarrow 0}{\sim} \lambda \epsilon, \label{esteasg1br} \\
\mathbb{P}_m \left ( N_{mult}(\epsilon) = 0, N_{coal}(\epsilon) = 1, N_{sing}(\epsilon) = 0 \right ) & \underset{\epsilon \rightarrow 0}{\sim} m(m-1) \epsilon, \label{esteasg1coal} \\
\mathbb{P}_m \left ( N_{mult}(\epsilon) = 0, N_{coal}(\epsilon) = 0, N_{sing}(\epsilon) = 1 \right ) & \underset{\epsilon \rightarrow 0}{\sim} m \sigma \epsilon, \label{esteasg1sbr} 
\end{align}
where $N_{mult}(\epsilon)$, $N_{coal}(\epsilon)$, and $N_{sing}(\epsilon)$ denote respectively the number of multiple branchings, coalescences, and single branchings of the E-ASG on $[0,\epsilon]$ (as in the proof of Lemma \ref{asymptcoeffnew}). 
\end{lemme}

\begin{proof}
From Definition \ref{enlargedasgdef} we see that, under $\mathbb{P}_m$, the first transition of the E-ASG occurs with rate $\lambda + m \sigma + m(m-1)$. Immediately after the first transition, the number of lines in the E-ASG is at most $2m$, so the second transition always occurs with a rate smaller than $\lambda + 2m \sigma + 4m^2$. We thus see that both the first and the second transition occur with a rate smaller than $4 m^2 (1+\lambda + \sigma)$. Therefore, if $e_1$ and $e_2$ denote two independent exponential random variables with parameter $4 m^2 (1+\lambda + \sigma)$, the probability in the left-hand side of \eqref{esteasg>2} is smaller than 
\[ \mathbb{P} (e_1 + e_2 \leq \epsilon) \leq \mathbb{P} (e_1 \leq \epsilon, e_2 \leq \epsilon) \leq (1-e^{-4 m^2 (1+\lambda + \sigma) \epsilon})^2 \leq 16 m^4 (1+\lambda + \sigma)^2 \epsilon^2. \]
This yields \eqref{esteasg>2}. We now prove \eqref{esteasg1br}. Let $T_1 < T_2 <...$ denote the transition times of the E-ASG and let $E_{MB}$ denote the event where the first transition of the E-ASG is a multiple branching. Note that the left-hand side of \eqref{esteasg1br} equals 
\begin{align}
\mathbb{P}_m \left ( T_1 \leq \epsilon, E_{MB}\right ) - \mathbb{P}_m \left ( T_2 \leq \epsilon, E_{MB} \right ). \label{esteasgproof1}
\end{align}
$T_1$ follows an exponential distribution with parameter $\lambda + m \sigma + m(m-1)$ and, conditionally on $T_1$, $E_{MB}$ has probability $\lambda/(\lambda + m \sigma + m(m-1))$. The first term in \eqref{esteasgproof1} thus equals $(1-e^{-(\lambda + m \sigma + m(m-1)) \epsilon}) \times \lambda/(\lambda + m \sigma + m(m-1)) \sim \lambda \epsilon$, while the second term in \eqref{esteasgproof1} is smaller than $16 m^4 (1+\lambda + \sigma)^2 \epsilon^2$ because of \eqref{esteasg>2}. \eqref{esteasg1br} follows. \eqref{esteasg1coal} and \eqref{esteasg1sbr} are proved exactly in the same way. 
%Since $N_{mult}(\epsilon) \sim \mathcal{P}(\lambda \epsilon)$ we have $\mathbb{P}_m(N_{mult}(\epsilon) = 1)=e^{-\lambda \epsilon} \lambda \epsilon$. Moreover, conditionally on $\{N_{mult}(\epsilon) = 1\}$, the rate, on $[0,\epsilon]$, at which occurs the first transition that is not a multiple branching is smaller than $2m\sigma+4m^2$. Therefore, if $e_1$ denotes an exponential random variables with parameter $2m\sigma+4m^2$, we have $\mathbb{P}_m(N_{coal}(\epsilon) + N_{sing}(\epsilon) \geq 1|N_{mult}(\epsilon) = 1)\leq \mathbb{P} (e_1 \leq \epsilon)=1-e^{-(2m\sigma+4m^2) \epsilon}\leq (2m\sigma+4m^2) \epsilon$. Plugging all this into \eqref{esteasgproof1} we get \eqref{esteasg1br}. 
\end{proof}

\begin{lemme}\label{combjumpeasg}
Let $m \geq 1, j \in \{ 1,..., 2m \}, i \in \{ \lceil j/2 \rceil,..., j \wedge m \}$, and $G \in \mathbb{G}_m^1$. If generation $1$ of $G$ is a multiple branching generation then 
\begin{eqnarray}
\sum_{A \in \mathcal{P}(V_{G}) ; |A|=j} F^i_{G}(A) = \binom{i}{j-i} \times (1+S_1)^{2i-j} \times (-S_1)^{j-i}, \label{asymptcoeff5new}
\end{eqnarray}
where, as in Section \ref{enlargedasgdefnot}, $S_1$ denotes the weight of generation $1$ of $G$. 
\end{lemme}

\begin{proof}
Assume first that $S_1 < 0$. By \eqref{defFjump} we have 
\[ F^i_{G}(A) := F^i_{\pi_0(G)}(P(A)) \times \mathds{1}_{\beta(A) = 0} \times (1+S_1)^{\alpha(A)} \times (-S_1)^{\gamma(A)}, \]
%Recall that $P(A)$ is uniquely determined from $A$. 
and by definition of $F^i_{\cdot}(\cdot)$ we have $F^i_{\pi_0(G)}(P(A)) = \mathds{1}_{P(A)=\{ L_1,..., L_i \}}$. We thus get 
\begin{align*}
\sum_{A \in \mathcal{P}(V_{G}) ; |A|=j} F^i_{G}(A) 
%& = \sum_{A \in \mathcal{P}(V_{G}) ; |A|=j} F^i_{\pi_0(G)}(P(A)) \times \mathds{1}_{\beta(A) = 0} \times (1+S_1)^{\alpha(A)} \times (-S_1)^{\gamma(A)} \\ 
= \sum_{A \in \mathcal{P}(V_{G}) ; P(A)=\{ L_1,..., L_i \} ; |A|=j ; \beta(A) = 0} (1+S_1)^{\alpha(A)} \times (-S_1)^{\gamma(A)}. 
\end{align*}
For $A \in \mathcal{P}(V_{G})$ such that $P(A)=\{ L_1,..., L_i \}$ we see by definition of $\alpha(A)$, $\beta(A)$ and $\gamma(A)$ that $\alpha(A)+\beta(A)+\gamma(A)=i$ and $i + \gamma(A) = |A|$. We thus see that, for $A \in \mathcal{P}(V_{G})$ such that $P(A)=\{ L_1,..., L_i \}$, $|A|=j$ is equivalent to $\gamma(A)=j-i$ and there exist such sets $A$ if and only if $j-i \leq i$, which is the case here since $i \in \{ \lceil j/2 \rceil,..., j \wedge m \}$. Therefore, the above equals 
\begin{align*}
\sum_{A \in \mathcal{P}(V_{G}) ; P(A)=\{ L_1,..., L_i \} ; \gamma(A)=j-i ; \beta(A) = 0} (1+S_1)^{2i-j} \times (-S_1)^{j-i}. 
\end{align*}
The number of terms (all equal to $(1+S_1)^{2i-j} \times (-S_1)^{j-i}$) in the above sum is the number of choices for the $j-i$ lines that contribute to $\gamma(A)$, among the lines of $\{ L_1,..., L_i \}$. This number of choices is $\binom{i}{j-i}$ so we get \eqref{asymptcoeff5new}. 

Assume now that $S_1 > 0$. By \eqref{defFjump} we have 
\[ F^i_{G}(A) := F^i_{\pi_0(G)}(P(A)) \times S_1^{\beta(A)} \times (-S_1)^{\gamma(A)}. \]
Let us now use again that $F^i_{\pi_0(G)}(P(A)) = \mathds{1}_{P(A) = \{ L_1,..., L_i \}}$ and that, for $A \in \mathcal{P}(V_{G})$ with $P(A)=\{ L_1,..., L_i \}$, $|A| = j$ is equivalent to $\gamma(A) = j-i$. We get 
\begin{align*}
\sum_{A \in \mathcal{P}(V_{G}) ; |A|=j} F^i_{G}(A) 
%& = \sum_{A \in \mathcal{P}(V_{G}) ; P(A)=\{ L_1,..., L_i \} ; |A|=j} S_1^{\beta(A)} \times (-S_1)^{\gamma(A)} \\ 
& = \sum_{A \in \mathcal{P}(V_{G}) ; P(A)=\{ L_1,..., L_i \} ; \gamma(A)=j-i} S_1^{\beta(A)} \times (-S_1)^{j-i}. 
\end{align*}
%and note that such sets $A$ exist if and only if $j-i \leq i$, which is the case here since $i \in \{ \lceil j/2 \rceil,..., j \wedge m \}$. 
For the sets $A$ in the above sum, since $\alpha(A)+\beta(A)+\gamma(A)=i$, the possible values for $\beta(A)$ are $0,..., 2i-j$. We thus get 
\[ \sum_{A \in \mathcal{P}(V_{G}) ; |A|=j} F^i_{G}(A) = \sum_{l=0}^{2i-j} \sum_{\substack{A \in \mathcal{P}(V_{G}) ; P(A)=\{ L_1,..., L_i \} ; \\ \gamma(A)=j-i ; \beta(A)=l}} S_1^{l} \times (-S_1)^{j-i}. \]
The number of terms (all equal to $S_1^{l} \times (-S_1)^{j-i}$) in the second sum is the number of choices for the $j-i$ lines that contribute to $\gamma(A)$, among the lines of $\{ L_1,..., L_i \}$, and then for the $l$ lines that contribute to $\beta(A)$, among the $2i-j$ remaining lines. The total number of choices is $\binom{i}{j-i} \binom{2i-j}{l}$. We thus get 
\[ \sum_{A \in \mathcal{P}(V_{G}) ; |A|=j} F^i_{G}(A) = \binom{i}{j-i} (-S_1)^{j-i} \sum_{l=0}^{2i-j} \binom{2i-j}{l} S_1^{l} = \binom{i}{j-i} (-S_1)^{j-i} (1+S_1)^{2i-j}. \]
%which is \eqref{asymptcoeff5new}. 
\end{proof}

\subsection{Proof of Proposition \ref{h(x)ht(x)} from Section \ref{relwf-asgbis}}\label{A1}

We first prove the quenched identity \eqref{quencheddual0}. For this we need to prove two lemmas: Lemma \ref{classicdual} shows that \eqref{quencheddual0} holds on intervals were there are no jumps, which is a consequence of the classical moment duality. Lemma \ref{betjumps} is a composition property for the function $h^{l,\omega}_{\cdot,\cdot}(\cdot)$, it is a consequence of the quenched renewal structure appearing in Remark \ref{markpropfquenchedrmk}. Then, \eqref{quencheddual0} will be proved by showing iteratively the identity $\mathbb{E}^{\omega} [ (X(\omega,T))^l \mid X(\omega,r)=y ]= h^{l,\omega}_{r,T}(y)$, starting with $r=T$ and decreasing $r$, going successively along 1) intervals without jumps, and 2) jumping times of the environment, until $r=0$. \eqref{momentstoht} will follow by integrating \eqref{quencheddual0} with respect to the environment. 
%so that we deal with deterministic jumping times and, between those jumping times the SDE satisfies by $X(t)$ is \eqref{sdewoj}. 

%LA FONCTION $F$ S'APPLIQUE SANS PROBLEME ICI, DONC CHANGER LA DEFINITION ORIGINALE DE LA FONCTION $f_{\cdot,\cdot}(B,A)$ POUR QU'ELLE S'APPLIQUE ICI. IL FAUT FAIRE ATTENTION AU FAIT QUE LE TEMPS NE VA PLUS DANS LE MEME SENS. PAR EXEMPLE ON PEUT DEFINIR $f_{G, \tilde G}(B,A)$ DES LORS QUE $G$ EST L'ANCETRE DE $\tilde G$ (IE $\pi_{\mathsf{depth}(G)}(\tilde G)=G$) ET $f_{t, s}(B,A)$ DEVIENT ALORS $f_{G_t, G_s}(B,A)$. ET PAS BESOIN DE MODIFIER LA DEFINITION ICI. IL FAUT AUSSI CHANGER LA DEF DU TRUC SHIFTE, POUR QUE CA S'APPLIQUE FACILEMENT MEME DANS NOTRE CAS DU COUP CA DOIT ETRE UNE SORTE DE DIFFERENCE ENTRE LES DEUX GRAPHES, ON PEUT LA NOTER $\tilde G \setminus_B G$. OK

%\begin{lemme}[Quenched duality] \label{quencheddual}
%For any fixed environment $\omega$, $x \in [0,1]$, $l \geq 1$ and $T \geq 0$, we have 
%%, if $\omega$ has a jump at time $t < T$, then we have 
%\begin{eqnarray}
%\mathbb{E}^{\omega} \left [ (X(\omega,T))^l \mid X(\omega,0)=x \right ]= h^{l, \omega}_{0,T}(x). \label{quencheddual0}
%\end{eqnarray}
%\end{lemme}

\begin{lemme} \label{classicdual}
Let $\omega$ be a fixed environment and $0 \leq r < t$. If there are no jumps of $\omega$ on the time interval $(r,t)$ then for any $l \geq 1$ and $x\in[0,1]$ we have 
\begin{eqnarray}
h^{l,\omega}_{r,t-}(x) = \mathbb{E}^{\omega} \left [ (X(\omega,t-))^l \mid X(\omega,r)=x \right ]. \label{classicdual1}
\end{eqnarray}
If moreover $t$ is not a jumping time of $\omega$ we have 
\begin{eqnarray}
h^{l,\omega}_{r,t}(x) = \mathbb{E}^{\omega} \left [ (X(\omega,t))^l \mid X(\omega,r)=x \right ]. \label{classicdual2}
\end{eqnarray}
\end{lemme}

\begin{proof}

Recall that $h^{l,\omega}_{r,t-}(x)$ is defined in Definition \ref{defquenchedasgbtd}. Since, by assumption, $\omega$ has no jumps on $(r,t)$, the only events that may occur in the ASG are single branchings (that favor type $1$) and coalescences. Therefore, in the type assignment procedure from Definition \ref{typeaspr}, all $l$ lines from time $s=t-$ receive type $0$ if and only if all lines from time $s=r$ receive type $0$. We thus get 
\begin{eqnarray}
h^{l,\omega}_{r,t-}(x) = \mathbb{E}^{\omega,t-}_l \left [ x^{N^{\omega,t-}_r} \mid N^{\omega,t-}_{t-} = l \right ], \label{classicdual3}
\end{eqnarray}
where $(N^{\omega,t-}_{s})_{s \in [r,t)}$ denotes the line counting process of the ASG $(A^{\omega,t-}_{s})_{s \in [r,t)}$ (introduced in Definition \ref{defquenchedasg}). Note that $(N^{\omega,t-}_{t-v})_{v \in (0,t-r]}$ is a continuous-time Markov chain with transition rates $q(i,j) = i(i-1)$ for every $i \geq 2$ and $j=i-1$, $q(i,j) = \sigma i$ for every $i \geq 1$ and $j=i+1$, and $q(i,j) = 0$ otherwise. On the other hand, since $\omega$ has no jumps on $(r,t)$, we see from \eqref{levymodelsdesimp} that $(1-X(\omega,s))_{s \in (r,t)}$ satisfies the SDE "$dX(r) = \sigma X(r)(1-X(r)) dr + \sqrt{2X(r)(1-X(r))} dB(r)$". Applying Theorem 2.3 of \cite{cordvech} to $(1-X(\omega,s))_{s \in (r,t)}$ we get 
\[ \mathbb{E}^{\omega} \left [ (X(\omega,t-))^l \mid X(\omega,r)=x \right ] = \mathbb{E}^{\omega,t-}_l \left [ x^{N^{\omega,t-}_r} \mid N^{\omega,t-}_{t-} = l \right ]. \]
Combining with \eqref{classicdual3} we get \eqref{classicdual1}. If moreover $t$ is not a jumping time of $\omega$ we have $X(\omega,t)=X(\omega,t-)$ and $h^{l, \omega}_{r,t}(x)=h^{l, \omega}_{r,t-}(x)$ so \eqref{classicdual2} follows. 

\end{proof}

The following lemma relies on Theorem \ref{propagationformule} (via Remark \ref{propagationformulequenched0} and \eqref{decompattendueentquenched}), on Lemma \ref{markpropf} (via Remark \ref{markpropfquenchedrmk}), and on Lemma \ref{classicdual}. 
\begin{lemme} \label{betjumps} 
Let $\omega$ be a fixed environment and $0 \leq r < t \leq T$. If $\omega$ has no jumps on $(r,t)$ then for any $l \geq 1$ and $x\in[0,1]$ we have 
\begin{align*}
\mathbb{E}^{\omega} \left [ h^{l,\omega}_{t-,T}(X(\omega,t-)) \mid X(\omega,r)=x \right ] = h^{l,\omega}_{r,T}(x). 
\end{align*}
\end{lemme}

%PEUT-ETRE QU'IL EST PLUS PRATIQUE DE TRAVAILLER AVEC UNE EXPRESSION PLUS LOCALE, C'EST-A-DIRE QUI N'IMPLIQUE PAS LE TRUC AU TEMPS $T$ MAIS JUSTE AU BOUT DE L'INTERVAL, SUR LE MODEL DE CE QU'ON A FAIT DANS L'ARTICLE PRECEDENT. VOIR. LE PROBLEME C'EST QUE C'EST COMPLIQUE DE FAIRE CA QUAND, COMME ICI, ON DOIT PRENDRE EN COMPT TOUTE LA COMBINATOIRE ET PAS JUSTE LE NOMBRE DE LIGNES. 

%IL FAUT UTILISER QUE, PAR LA PREUVE DU THEOREM \ref{propagationformule} ON A : 
%\begin{eqnarray}
%h^l_T(x)= \mathbb{E}_m \left [ \sum_{A \in \mathcal{P}(V_{G_{T-t}})} F^l_{G_{T-t}}(A) \mathbb{E}_{|A|} \left [ E(...,T,A,x) | ... \right ] \right ]. \label{resultinterm}
%\end{eqnarray}

%\eqref{resultinterm} PERMET DE SE RAMENER AU CAS $T=t$ ET DONC DE CONCLURE PAR UNE DUALITE CLASSIQUE. 

\begin{proof}

Recall the definition of the sigma-field $\mathcal{F}^{\omega,T}_{t-}$ in Section \ref{enlargedasgdefnot}. Using \eqref{propagationformulequenched}, \eqref{markpropfquenched}, \eqref{markpropfquenched2}, \eqref{decompattendueentquenched} we get that $h^{l,\omega}_{r,T}(x)$ equals 
\begin{align*}
& \mathbb{E}^{\omega,T}_m \left [ \sum_{A \in \mathcal{P} ( V_{G^{\omega,T}_r} )} F^l_{G^{\omega,T}_r}(A) x^{|A|} \right ] = \mathbb{E}^{\omega,T}_m \left [ \sum_{A \in \mathcal{P}(V_{G^{\omega,T}_r})} \sum_{B \in \mathcal{P}(V_{G^{\omega,T}_{t-}})} F^l_{G^{\omega,T}_{t-}}(B) f_{G^{\omega,T}_{t-},G^{\omega,T}_{r}}(B,A) x^{|A|} \right ] \nonumber \\
= & \mathbb{E}^{\omega,T}_m \left [ \sum_{B \in \mathcal{P}(V_{G^{\omega,T}_{t-}})} F^l_{G^{\omega,T}_{t-}}(B) \mathbb{E}^{\omega,T}_m \left [ \sum_{A \in \mathcal{P}(V_{G^{\omega,T}_r})} f_{G^{\omega,T}_{t-},G^{\omega,T}_{r}}(B,A) x^{|A|} \big | \mathcal{F}^{\omega,T}_{t-} \right ] \right ] \nonumber \\
= & \mathbb{E}^{\omega,T}_m \left [ \sum_{B \in \mathcal{P}(V_{G^{\omega,T}_{t-}})} F^l_{G^{\omega,T}_{t-}}(B) \sum_{k \geq 1} \sum_{j = 1}^k \mathbb{E}^{\omega,T}_m \left [ \mathds{1}_{|V_{G^{\omega,T}_{r}}| = k} \sum_{A \in \mathcal{P}(V_{G^{\omega,T}_r}) ; |A|=j} f_{G^{\omega,T}_{t-},G^{\omega,T}_{r}}(B,A) \big | \mathcal{F}^{\omega,T}_{t-} \right ] x^{j} \right ] \nonumber \\
= & \mathbb{E}^{\omega,T}_m \left [ \sum_{B \in \mathcal{P}(V_{G^{\omega,T}_{t-}})} F^l_{G^{\omega,T}_{t-}}(B) \sum_{k \geq 1} \sum_{j = 1}^k R^{|V_{G^{\omega,T}_{t-}}|,k, \omega}_{r,t-}(|B|,j) x^{j} \right ] = \mathbb{E}^{\omega,T}_m \left [ \sum_{B \in \mathcal{P}(V_{G^{\omega,T}_{t-}})} F^l_{G^{\omega,T}_{t-}}(B) h^{|B|, \omega}_{r,t-}(x) \right ]. 
\end{align*}
%ALTERNATIVE POUR 1)' : AU LIEU DE TRIFOUILLER DANS LE PREUVE DU THEOREM \ref{propagationformule}, PEUT-ETRE QU'ON PEUT UTILISER CE QU'ON AVAIT APPELE LA MARKOV PROPERTY POUR $F_{\cdot}^{\cdot}(\cdot)$, CE QUI SEMBLE PLUS SIMPLE ET CA PREND DEJA EN COMPTE LE TRUC A ENVIRONMENT FIXE (CA MONTRE AU PASSAGE QU'UNE ALTERNATIVE POTENTIELLE A LA MARKOV PROPERTY POUR $F_{\cdot}^{\cdot}(\cdot)$ PEUT SE TROUVER DANS LE PREUVE DU THEOREM \ref{propagationformule}). 
By assumption, $\omega$ has no jumps on $(r,t)$. Therefore, combining with Lemma \ref{classicdual} and using Fubini's Theorem and \eqref{propagationformulequenched} we get 
\begin{align*}
h^{l, \omega}_{r,T}(x) & = \mathbb{E}^{\omega,T}_m \left [ \sum_{B \in \mathcal{P}(V_{G^{\omega,T}_{t-}})} F^l_{G^{\omega,T}_{t-}}(B) \mathbb{E}^{\omega} \left [ (X(\omega,t-))^{|B|} \big | X(\omega,r)=x \right ] \right ] \\
& = \mathbb{E}^{\omega} \left [ \mathbb{E}^{\omega,T}_m \left [ \sum_{B \in \mathcal{P}(V_{G^{\omega,T}_{t-}})} F^l_{G^{\omega,T}_{t-}}(B) (X(\omega,t-))^{|B|} \right ] \big | X(\omega,r)=x \right ] \\
& = \mathbb{E}^{\omega} \left [ h^{l,\omega}_{t-,T}(X(\omega,t-)) \big | X(\omega,r)=x \right ]. 
\end{align*}
\end{proof}

We can now prove Proposition \ref{h(x)ht(x)}, relying only on the above two lemmas. 

\begin{proof} [Proof of Proposition \ref{h(x)ht(x)}]

We first prove \eqref{quencheddual0}. If $\omega$ has no jumps on $[0,T]$ the result follows by \eqref{classicdual2} applied with $r=0$ and $t=T$. Let $t^{\omega}_1< ...< t^{\omega}_{N}$ denote the jumping times of $\omega$ on $[0,T]$ if there are $N \geq 1$ such jumping times. For convenience we assume that $T$ is not a jumping time of $\omega$ (the following proof can be easily adapted to the case where it is). Since there is no jump of $\omega$ on $(t^{\omega}_N,T]$ we have by \eqref{classicdual2}, applied with $r=t^{\omega}_N$ and $t=T$, that for all $y \in [0,1]$: 
\begin{eqnarray}
\mathbb{E}^{\omega} \left [ (X(\omega,T))^l \mid X(\omega,t^{\omega}_N)=y \right ]= h^{l,\omega}_{t^{\omega}_N,T}(y). \label{tnt}
\end{eqnarray}
%On one hand, the equation satisfied by $X(\omega,\cdot)$, that is, \eqref{levymodelsdesimp}, tells us that we have almost surely $X(\omega,t) = X(\omega,t-) + X(\omega,t-)(1-X(\omega,t-))\Delta \omega(t)$ so that 
%\begin{align}
%\mathbb{E}^{\omega} \left [ (X(\omega,t))^l \mid X(\omega,t-)=x \right ] & = \left [ x + x(1-x)\omega (t) \right ]^l \label{jumponx}
%\end{align}
%One the other hand, 
Let us assign \textit{iid} types with law $y \delta_0 + (1-y) \delta_1$ to lines of $V_{G^{\omega,T}_{t^{\omega}_N-}}$ and propagate types as in Definition \ref{typeaspr}. If $\Delta \omega (t^{\omega}_N) > 0$, each line in $V_{G^{\omega ,T}_{t^{\omega}_N}}$ receives type $0$ with probability \\ $\Delta \omega (t^{\omega}_N) (y + (1-y)y) + (1-\Delta \omega (t^{\omega}_N)) y = y + y(1-y)\Delta \omega (t^{\omega}_N)$. If $\Delta \omega (t^{\omega}_N) < 0$, each line in $V_{G^{\omega ,T}_{t^{\omega}_N}}$ receives type $0$ with probability $-\Delta \omega (t^{\omega}_N) y^2 + (1+\Delta \omega (t^{\omega}_N)) y = y + y(1-y)\Delta \omega (t^{\omega}_N)$. Using Definition \ref{defquenchedasgbtd} and \eqref{tnt}, we get that in any case we have 
\begin{align}
h^{l,\omega}_{t^{\omega}_N-,T}(y) = h^{l,\omega}_{t^{\omega}_N,T}(y + y(1-y)\Delta \omega (t^{\omega}_N)) = \mathbb{E}^{\omega} \left [ (X(\omega,T))^l \mid X(\omega,t^{\omega}_N)=y + y(1-y)\Delta \omega (t^{\omega}_N) \right ]. \label{tnt2}
\end{align}
Now, note that according to \eqref{levymodelsdesimp} we have $X(\omega,t^{\omega}_N) = X(\omega,t^{\omega}_N-) + X(\omega,t^{\omega}_N-)(1-X(\omega,t^{\omega}_N-))\Delta \omega(t^{\omega}_N)$ $\mathbb{P}^{\omega}$-almost surely. 
%\begin{align*}
%\mathbb{E}^{\omega} \left [ (X(\omega,T))^l \mid X(\omega,t^{\omega}_N-)=y \right ] = \mathbb{E}^{\omega} \left [ (X(\omega,T))^l \mid X(\omega,t^{\omega}_N)=y + y(1-y)\Delta \omega (t^{\omega}_N) \right ]. 
%\end{align*}
Plugging into \eqref{tnt2} we get 
\begin{eqnarray}
\mathbb{E}^{\omega} \left [ (X(\omega,T))^l \mid X(\omega,t^{\omega}_N-)=y \right ] = h^{l,\omega}_{t^{\omega}_N-,T}(y). \label{tnt3}
\end{eqnarray}
Then, disintegrating on the value of $X(\omega,t^{\omega}_{N}-)$ and using \eqref{tnt3} we get that for any $y \in [0,1]$, $\mathbb{E}^{\omega} [ (X(\omega,T))^l \mid X(\omega,t^{\omega}_{N-1})=y ]$ equals 
\begin{align*}
& \int_0^1 \mathbb{E}^{\omega} \left [ (X(\omega,T))^l \mid X(\omega,t^{\omega}_{N}-)=z \right ] \times \mathbb{P}^{\omega} \left ( X(\omega,t^{\omega}_{N}-) \in dz \mid X(\omega,t^{\omega}_{N-1})=y \right ) \\
= & \int_0^1 h^{l,\omega}_{t^{\omega}_N-,T}(z) \times \mathbb{P}^{\omega} \left ( X(\omega,t^{\omega}_{N}-) \in dz \mid X(\omega,t^{\omega}_{N-1})=y \right ) \\
= & \mathbb{E}^{\omega} \left [ h^{l,\omega}_{t^{\omega}_N-,T}(X(\omega,t^{\omega}_{N}-)) \mid X(\omega,t^{\omega}_{N-1})=y \right ] = h^{l,\omega}_{t^{\omega}_{N-1},T}(y). 
\end{align*}
We have used Lemma \ref{betjumps} for the last equality. Iterating the above two procedures at respectively each jumping time and each interval between two jumping times, we eventually get \eqref{quencheddual0}. The combination of \eqref{quencheddual0} with \eqref{qtoannealed} and \eqref{recoverannealed0} yields \eqref{momentstoht}, concluding the proof. 

\end{proof}

\subsection{Proof of Lemma \ref{reconstitutiondercoeff} from Section \ref{dlh(x)subsec1}} \label{A4}

Using Theorem \ref{equadiffcoeffnew} and the definition of the coefficients $d_j$, $e_{k,j}$, $f_{k,j}$ in Definition \ref{defcoefedo}, we get that, for $i,j,n\geq 1$, $\sum_{k=1}^{n} \frac{d}{dt} R^{i,k}_t(i,j)$ equals 
\begin{align*}
& \sum_{k=1}^{n} \left ( \tau(j+1,j) R^{i,k+1}_t(i,j+1) + e_{k,j} R^{i,k+1}_t(i,j) + \mathds{1}_{k \geq 2} f_j R^{i,k-1}_t(i,j-1) \right. \\
& \left. + \mathds{1}_{j \leq k-1} f_{k,j} R^{i,k-1}_t(i,j) -d_k R^{i,k}_t(i,j) + \mathds{1}_{\{k\text{ is even}\}} \sum_{l =1}^{j \wedge (k/2)} \tau(l,j) R^{i,k/2}_t(i,l) \right ) \\
= & \tau(j+1,j) \left ( \sum_{k=1}^{n} R^{i,k+1}_t(i,j+1) \right ) + \left ( \sum_{k=1}^{n} (k+1)k R^{i,k+1}_t(i,j) \right ) - j(j-1) \left ( \sum_{k=1}^{n} R^{i,k+1}_t(i,j) \right ) \\
+ & f_j \left ( \sum_{k=2}^{n} R^{i,k-1}_t(i,j-1) \right ) + \sigma \left ( \sum_{k=j+1}^{n} (k-1) R^{i,k-1}_t(i,j) \right ) - j\sigma \left ( \sum_{k=j+1}^{n} R^{i,k-1}_t(i,j) \right ) \\
- & \lambda \left ( \sum_{k=1}^{n} R^{i,k}_t(i,j) \right ) - \left ( \sum_{k=1}^{n} k(k-1) R^{i,k}_t(i,j) \right ) \\
- & \sigma \left ( \sum_{k=1}^{n} k R^{i,k}_t(i,j) \right ) + \sum_{l =1}^{j} \tau(l,j) \left ( \sum_{\substack{k \ \text{even} ; \\ l \leq k/2 \leq n/2}} R^{i,k/2}_t(i,l) \right ) \\
= & \tau(j+1,j) \left ( \sum_{k=2}^{n+1} R^{i,k}_t(i,j+1) \right ) - j(j-1) \left ( \sum_{k=2}^{n+1} R^{i,k}_t(i,j) \right ) + f_j \left ( \sum_{k=1}^{n-1} R^{i,k}_t(i,j-1) \right ) \\ 
- & j\sigma \left ( \sum_{k=j}^{n-1} R^{i,k}_t(i,j) \right ) - \lambda \left ( \sum_{k=1}^{n} R^{i,k}_t(i,j) \right ) + (n+1)n R^{i,n+1}_t(i,j) \\ 
- & \sigma \left ( n R^{i,n}_t(i,j) + \sum_{k=1}^{j-1} k R^{i,k}_t(i,j) \right ) + \sum_{l =1}^{j} \tau(l,j) \left ( \sum_{k \ \text{even} ; l \leq k/2 \leq n/2} R^{i,k/2}_t(i,l) \right ). 
\end{align*}
By definition of $R^{m,k}_t(i,j)$ we have $R^{m,k}_t(i,j) = 0$ whenever $k < j$. Therefore the above equals 
\begin{align}
& \tau(j+1,j) \left ( \sum_{k=1}^{n+1} R^{i,k}_t(i,j+1) \right ) + f_j \left ( \sum_{k=1}^{n-1} R^{i,k}_t(i,j-1) \right ) - d_j \left ( \sum_{k=1}^{n} R^{i,k}_t(i,j) \right ) \nonumber \\ 
+ & \sum_{l =1}^{j} \left [ \tau(l,j) \left ( \sum_{k = 1}^{\lfloor n/2 \rfloor} R^{i,k}_t(i,l) \right ) \right ] + \left ( (n+1)n - j(j-1) \right ) R^{i,n+1}_t(i,j) + (j-n) \sigma R^{i,n}_t(i,j). \label{micromacro}
\end{align}
By the combination of \eqref{majonouveausgnew} and \eqref{majointemporelle} we have 
\[ \left | \left ( (n+1)n - j(j-1) \right ) R^{i,n+1}_t(i,j) \right | \leq \left ( (n+1)n + j(j-1) \right ) \frac{j^j (n+1)^j}{j! \pi(i)} \pi(n+1). \]
By \eqref{majoqueuepiknew}, the right-hand side converges to $0$ as $n$ goes to infinity. Therefore the term in the left-hand side converges to $0$, uniformly in $t \in [0,\infty)$, as $n$ goes to infinity. The same holds for $(j-n) \sigma R^{i,n}_t(i,j)$. Applying Lemma \ref{reconstitutioncoeff} to each of the other terms in \eqref{micromacro} we get the result. 

\newpage 
\begin{table}
\caption{Table of notations}
\begin{tabularx}{\textwidth}{@{}XXX@{}}
\toprule
  \textit{Notation} & \textit{Meaning} & \textit{Definition location} \\ 
  $X, (X(t))_{t \geq 0}$ & solution of \eqref{levymodelsdesimp} & Section \ref{firstteps}, Section \ref{diffusion} \\
%  $B, (B(t))_{t \geq 0}$ & Brownian motion in \eqref{levymodelsdesimp} & Section \ref{firstteps} \\
  $L, (L(t))_{t \geq 0}$ & L\'evy environment & Section \ref{diffusion} \\
  $\sigma, \lambda, \nu$ & parameters of $L$ & Section \ref{diffusion} \\
  $P$ & law of $L$ & Section \ref{diffusion} \\
  $\mathbb{P},\mathbb{E}$ & annealed law of $X$ & Section \ref{diffusion} \\
  $h(x)$ & fixation probability & \eqref{defh(x)} in Section \ref{firstteps} \\
  $\omega, (\omega(t))_{t \geq 0}$ & fixed environment & Section \ref{diffusion} \\
  $(X(\omega,t))_{t \geq 0}$ & quenched version of $X$ & Section \ref{diffusion} \\
  $\mathbb{P}^{\omega},\mathbb{E}^{\omega}$ & quenched law of $X$ & Section \ref{diffusion} \\
  $(A^{\omega,T}_{s})_{s \in [0,T]}, (A_{\beta})_{\beta \geq 0}$ & quenched/annealed ASG & Definition \ref{defquenchedasg} in Section \ref{asg} \\
  $h^l_T(x), h^{l, \omega}_{0,T}(x), h^{l, \omega}_{T_1,T_2}(x)$ & backward type distribution & Definition \ref{defquenchedasgbtd} in Section \ref{relwf-asg} \\
  $(G^{\omega,T}_{s})_{s \in [0,T]}, (G_{\beta})_{\beta \geq 0}$ & quenched/annealed E-ASG & Definition \ref{defquenchedeasg} in Section \ref{enlargedasgdef} \\
  $\mathbb{P}_m, \mathbb{E}_m$ & annealed law of (E-)ASG & Section \ref{asg}, Section \ref{enlargedasgdef} \\
  $\mathbb{P}^{\omega,T}_m, \mathbb{E}^{\omega,T}_m$ & quenched law of (E-)ASG & Section \ref{asg}, Section \ref{enlargedasgdef} \\
  $E(t,T,A,x)$ & type $0$ event & Definition \ref{type0events} in Section \ref{enlargedasgdefnot} \\
  $(\mathcal{F}_t)_{t \geq 0}, (\mathcal{F}^{\omega,t}_{r})_{r \in [0,t]}$ & filtrations for E-ASG & Section \ref{enlargedasgdefnot} \\
  $\mathbb{G}_m$ & set of realizations of E-ASG & Section \ref{enlargedasgdefnot} \\
  $(S_n)_{n \geq 0}$ & weights of E-ASG & Section \ref{enlargedasgdefnot} \\
  $\mathsf{depth}(G)$ & counts generations of $G \in \mathbb{G}_m$ & Section \ref{enlargedasgdefnot} \\
  $\mathbb{G}_m^n$ & elements of $\mathbb{G}_m$ with depth $n$ & Section \ref{enlargedasgdefnot} \\
  $\pi_n(G)$ & projection of $G \in \mathbb{G}_m$ on $\mathbb{G}_m^n$ & Section \ref{enlargedasgdefnot} \\
  $V_G$ & lines in last generation of $G$ & Section \ref{enlargedasgdefnot} \\
  $\pi$ & stationary law for $(|V_{G_{\beta}}|)_{\beta \geq 0}$ & Section \ref{enlargedasglcp} \\
  $\mathcal{P}(V_G)$ & non-empty parts of $V_G$ & Section \ref{enlargedasgdefnot} \\
  $P(A)$ & parents of lines in $A$ & Section \ref{enlargedasgdefnot} \\
  $D(A)$ & sons of lines in $A$ & Section \ref{enlargedasgdefnot} \\
  $F^l_G(\cdot)$ & encoding function & Section \ref{enlargedasgencodingfct} \\
  $N(A), \alpha(A), \beta(A), \gamma(A)$ & quantities used to define $F^l_G(\cdot)$ & Section \ref{enlargedasgencodingfct} \\
  $Q_t(i,j), R^{m,k}_t(i,j), R^{m,k, \omega}_{r,t}(i,j)$ & duality coefficients & Definition \ref{defcoefsg} in Section \ref{mainresults} \\
  $d_j, e_{k,j}, f_j, f_{k,j}, \tau(i,j)$ & ODEs coefficients & Definition \ref{defcoefedo} in Section \ref{mainresults} \\
  $a^k_j, b_j$ & limits of $R^{m,k}_t(i,j)$ and $Q_t(i,j)$ & \eqref{cvcoeff1new} and \eqref{cvcoeff1} in Section \ref{mainresults} \\
  $P_k(x)$ & $\sum_{j=1}^k a^k_j x^{j}$ & Theorem \ref{finalformula} in Section \ref{mainresults} \\
  $\tilde G \setminus_B G$ & shifted E-ASG & Section \ref{behaviour021} \\ 
  $f_{G, \tilde G}(\cdot,\cdot)$ & extended encoding function & \eqref{defpetitfg} in Section \ref{behaviour021} \\ 
  $P^{m,k,i}_t (y)$ & $\sum_{j=1}^k R^{m,k}_t(i,j) y^j$ & Lemma \ref{cvn} in Section \ref{ccl} \\ 
  $S_T(n,x), S_{\infty}(n,x)$ & remainders in \eqref{taylorexp3}, \eqref{dl1} & Section \ref{proofoftaylorexp} \\ \\
\bottomrule
\end{tabularx}
\end{table}

\textbf{Acknowledgements:}
This paper is supported by NSFC grant No. 11688101. The author is grateful to Fernando Cordero for many interesting discussions and to Sebastian Hummel for useful references and help in improving the writing. The author is also grateful to two anonymous referees for their careful reading and valuable suggestions. 

\bibliographystyle{plain}
\bibliography{thbiblio}

\end{document}